\documentclass{amsart}
\usepackage{amsmath, amsthm, amssymb, mathtools}
\usepackage{mathrsfs}
\usepackage{algorithm}
\usepackage{algorithmic}
\usepackage[all]{xy}
\usepackage{graphicx}
\usepackage{mathptmx}
\usepackage[margin=4cm]{geometry}
\usepackage{url}
\usepackage{combelow}
\usepackage{subfig}
\usepackage{xcolor}
\usepackage{array}
\usepackage{booktabs}

\allowdisplaybreaks

\newcommand{\algref}[1]{Algorithm~\ref{#1}}

\theoremstyle{plain}
\newtheorem{thm}{Theorem}[section]

\newtheorem{lemma}[thm]{Lemma}

\theoremstyle{definition}
\newtheorem{defn}[thm]{Definition}
\newtheorem{ex}[thm]{Example}
\theoremstyle{remark}
\newtheorem{rem}[thm]{Remark}
\theoremstyle{definition}
\newtheorem*{defn*}{Definition}
\theoremstyle{plain}
\newtheorem*{thm*}{Theorem}

\newcommand{\thmref}[1]{Theorem~\ref{#1}}
\newcommand{\lemmaref}[1]{Lemma~\ref{#1}}

\newcommand{\exref}[1]{Example~\ref{#1}}
\newcommand{\secref}[1]{Section~\ref{#1}}

\newcommand{\tabref}[1]{Table~\ref{#1}}
\newcommand{\figref}[1]{Figure~\ref{#1}}

\newcommand{\NN}{\mathbb{N}}
\newcommand{\ZZ}{\mathbb{Z}}
\newcommand{\RR}{\mathbb{R}}
\newcommand{\CC}{\mathbb{C}}

\newcommand{\PP}{\mathbb{P}}

\newcommand{\N}[1]{N_{#1}}

\newcommand{\codim}[1]{\operatorname{codim} (#1)}

\renewcommand{\dim}[1]{\operatorname{dim} (#1)}
\renewcommand{\deg}[1]{\operatorname{deg} (#1)}

\begin{document}

\title{Sampling and homology via bottlenecks}%Sampling algebraic varieties} %Or geometric data analysis??

\author{Sandra Di Rocco}
\address{Department of mathematics, KTH, 10044,
  Stockholm, Sweden}
\email{dirocco@math.kth.se}
\urladdr{https://people.kth.se/~dirocco/}

\author{David Eklund}
\address{DTU Compute, Richard Petersens Plads,
Building 321, DK-2800 Kgs. Lyngby, Denmark}
\email{daek@math.kth.se}
\urladdr{https://www2.compute.dtu.dk/~daek/}

\author{Oliver G\"afvert}
\address{Department of mathematics, KTH, 10044,
	Stockholm, Sweden}
\email{oliverg@math.kth.se}
\urladdr{https://people.kth.se/~oliverg/}

% want to have the following?
% picture illustrating proof of prop 6.4 (last prop)
% that is a circle of rad epsilon centered at e
% with p,q,x,y,z,L and traingle e-p-z marked? (oliver?)

% comment that we ignore the presence of noise in samples

% also:
% Comment somewhere that dimension is input
% to sampling proc but may be computed with homotopy methods

%\keywords{systems of polynomials, triangulation of
%  manifolds, reach of manifolds}
%\subjclass[2010]{14Q20, 65D18}

\begin{abstract}
In this paper we present an efficient algorithm to produce a provably
dense sample of a smooth compact variety. The procedure is partly
based on computing \emph{bottlenecks} of the variety. Using geometric information such as the
bottlenecks and the \textit{local reach} we also provide bounds on the density of the sample needed
in order to guarantee that the homology of the variety
can be recovered from the sample. An implementation of the algorithm
is provided together with numerical experiments and a computational
comparison to the algorithm by Dufresne et. al. \cite{DEHH18}.
\end{abstract}

\maketitle

\section{Introduction}

%Sampling is a fundamental problem in data analysis. 
%
%Constructing a probably dense sample is necessary to accurately estimate geometrical and topological invariants such as curvature, reach and homology. This pipline rivals that of persistent homology, as its main objective is to estimate the homology of an algebraic variety. The pipeline is as follows:
%
%1. Compute the bottlenecks and estimate the reach, value $=\delta$.
%2. Construct a provably $\delta$-dense sample 
%3. Compute homology of complex at filtration value $\delta$.
%
%This pipeline depends on the reach estimate. If we at least have a lower bound on the reach we have a way of provably computing the homology of varieties. 
%
%

Sampling an object is an important problem when trying to recover
information about its structure \cite{NSW08, diaconis, DEHH18, CS09,
  BO19}. Efficiency of the sampling algorithm becomes  central as
many techniques tend to grow exponentially in complexity with the
ambient dimension. In our setting, the objects we study are zero-sets
of polynomial equations. More precisely, we study smooth, compact
algebraic varieties in $\RR^n$. The information we want to recover
are topological invariants describing the shape of the variety, such
as the homology groups. For instance, the zeroth homology group counts
the number of connected components of the variety. The techniques we
develop can be applied in a range of areas where the object of study
is algebraic in nature, for instance kinematics \cite{DRESW10, selig}
and biochemistry \cite{CMTW10}.

Existing algorithms for sampling algebraic varieties can be divided
into two categories: probabilistic methods \cite{BO19, diaconis, CS09}
and provably dense sampling \cite{DEHH18}. In \secref{sec:sampling} we
develop a sampling technique that constructs a provably dense sample
on $X$. We also provide a computational comparison to the algorithm by
Dufresne et. al. \cite{DEHH18} (see Section \ref{sec:comp}) showing
better performance of our algorithm as the codimension increases. A public implementation of our algorithm is available at \cite{sampling_bottlenecks2020}. 

By provably dense sample we mean an $\epsilon$-sample in the sense of the
following definition:
\begin{defn*}
A {\it sample} of a variety $X\subset\RR^n$ is a finite subset
$E\subset X$. For $\epsilon>0$ a sample $E\subset X$ is called an
$\epsilon$-sample if for every $x\in X$ there is an element $e\in E$
such that $\|x-e\|<\epsilon.$ In this case we also say that $E$ is
$\epsilon$-dense.
\end{defn*}
To produce an $\epsilon$-sample from $X$ we start from its defining
equations. The basic idea is then to intersect the variety with a grid
of linear spaces of complementary dimension. Figure
\ref{fig:planar-curve-intro} shows a curve in $\RR^2$ intersected with
a 2-dimensional grid. The grid-size, denoted by $\delta,$ is determined by the narrowest  {\it bottleneck} of the variety. A bottleneck of $X$ is a pair
of distinct points $x,y \in X$ such that $(x-y)$ is normal to $X$ at
both $x$ and $y$. Estimates for the number of bottlenecks and
numerical algorithms to compute them have been recently studied, see
\cite{E18, DEW19}. In \secref{sec:bottlenecks} we give a system of
equations defining the bottlenecks of $X$. In case $X$ has only
finitely many bottlenecks, these equations can be used to find all
bottlenecks, including the narrowest one. In \thmref{thm:finite-bn} we
provide conditions under which finiteness is achieved:
\begin{thm*}\ref{thm:finite-bn}. 
A generic complete intersection has only finitely many bottlenecks.
\end{thm*}

\begin{figure}[!ht]
	\centering
	\subfloat[][Planar curve\label{fig:planar-curve-intro}]{
		\includegraphics[trim={0cm 0cm 0cm 0cm},clip,scale=0.25]
		{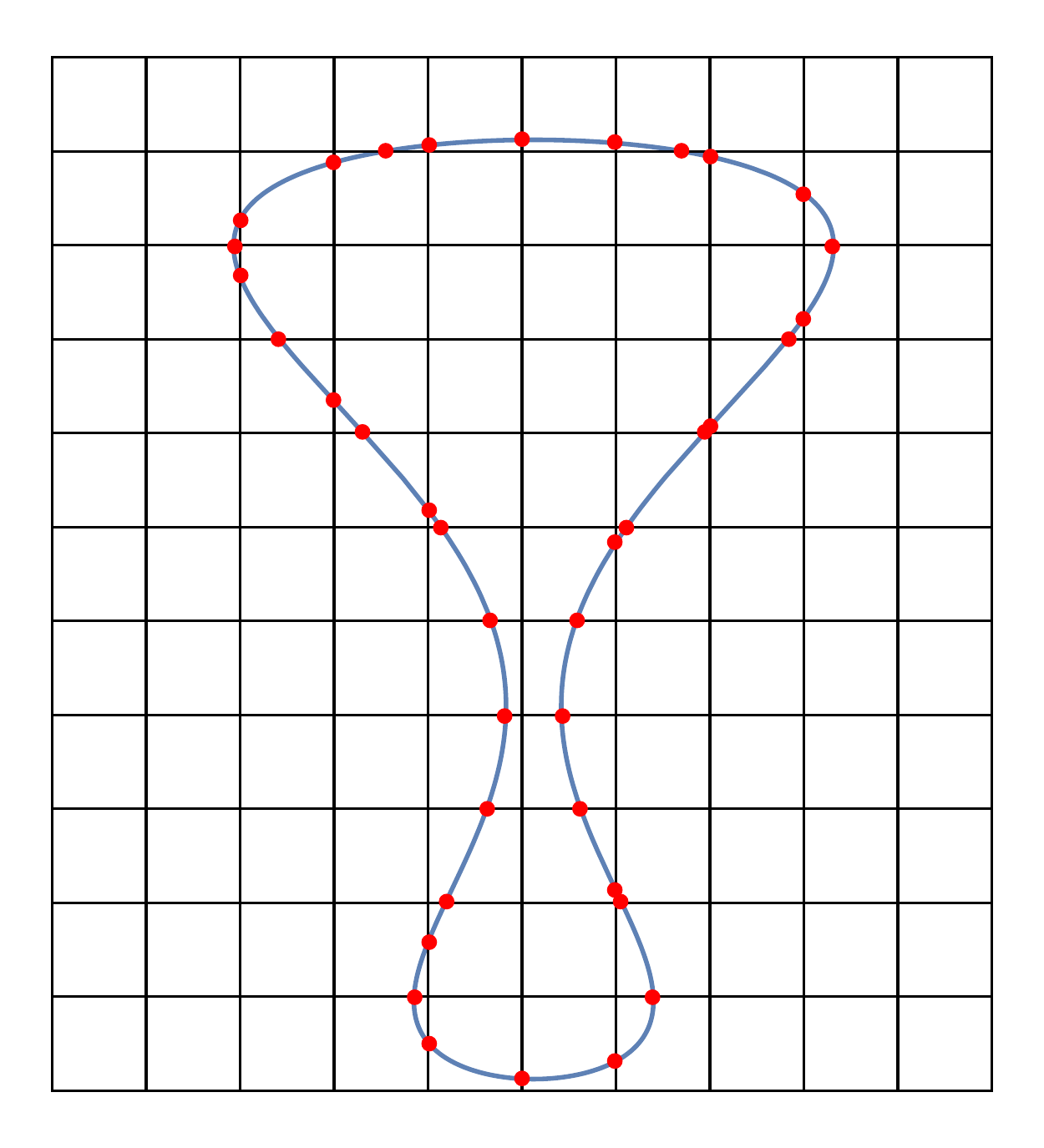}}
	\qquad
	\subfloat[][A quartic surface in $\RR^4$\label{fig:surface}]{
		\includegraphics[trim={2cm 3cm 2cm 3cm},clip,scale=0.1]
		{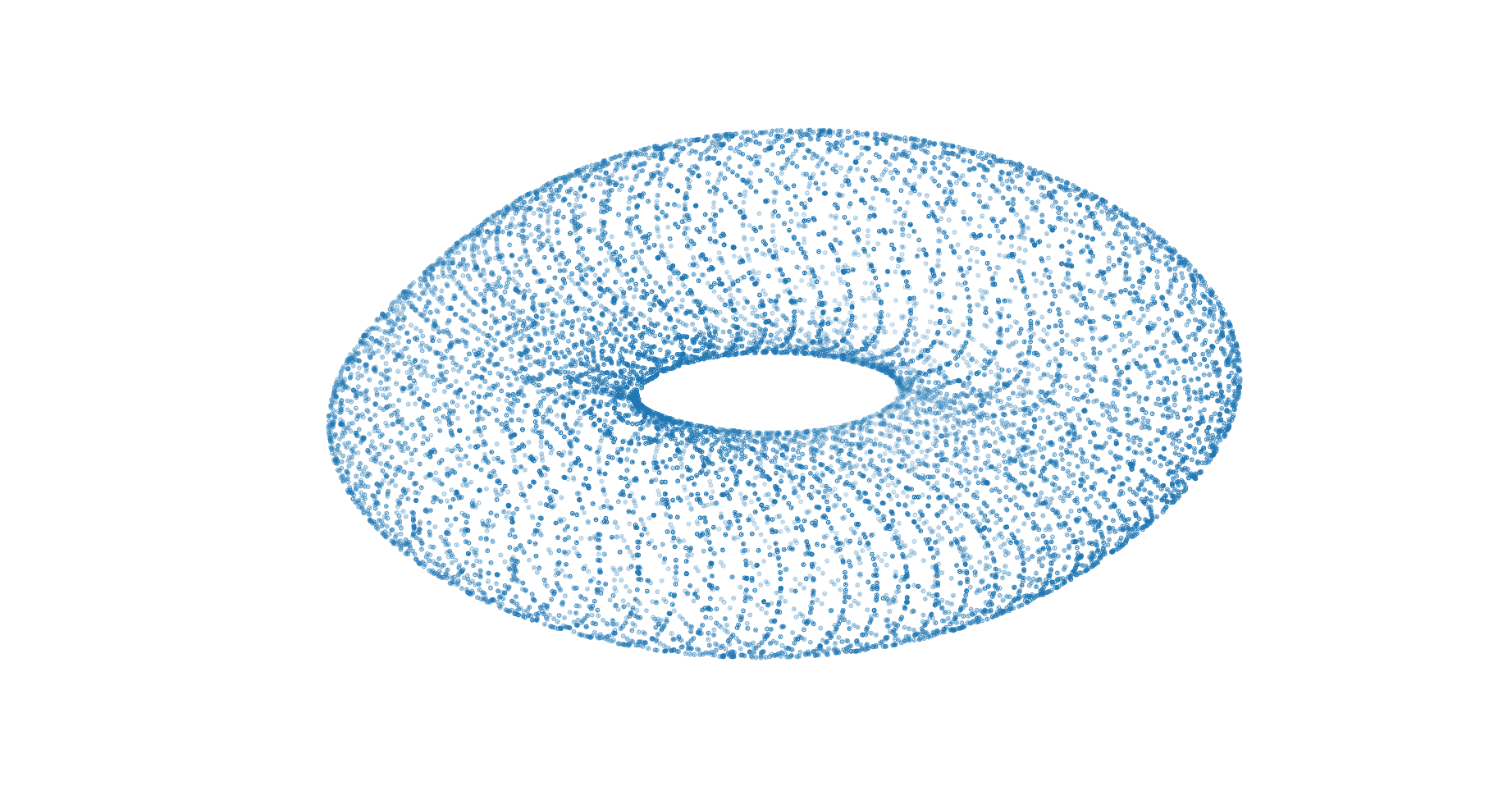}}
\end{figure}

The density guarantee is achieved by introducing extra sample points
given by ad hoc slicing and nearest point computations (see Section
\ref{sec:extra}). Figure \ref{fig:surface} shows an $\epsilon$-sample
of a surface in $\RR^4$ with $\epsilon=0.1$.

Specifically we prove that:

\begin{thm*}\ref{thm:grid}
  Let $b_2$ be the radius of the narrowest bottleneck of $X$ and let
  $\epsilon > 0$. If $0<\delta \sqrt{n} < \min\{\epsilon, 2b_2\}$, then
  $S_{\delta}$ is an $\epsilon$-sample of $X$.
\end{thm*}

\secref{sec:homology} is devoted to homology computations from a
given sample. Given an $\epsilon$-sample of $X$ it is straightforward
to compute the homology groups of $X$ assuming that $\epsilon$ is
small enough. In \cite{BCL17} it is shown that the homology groups
of $X$ can be recovered from the sample if $\epsilon < \tau_X$,
where $\tau_X$ is the \emph{reach} of $X$ (defined in
\secref{sec:edd}). The reach and bottlenecks are geometric invariants
of $X$ specific to its embedding in $\RR^n$. They contribute to estimates on  how dense we have to sample $X$ in order to
be able to recover its shape. In this paper we show that the density
of the sample needed to compute the zeroth and first homology groups
of $X$ is controlled by the radii of the \textit{generalized} bottlenecks (see Section \ref{sec:bottlenecks}). The radius of the narrowest generalized bottleneck is equivalent to the \textit{weak feature size (wfs)} introduced in \cite{10.1016/j.gmod.2005.01.002}. One of the main contributions of this paper is the following
result:
\begin{thm*}\ref{thm:bn_hom}.
	Let $\epsilon < wfs(X)$ and let $i\in \{0, 1\}$. Let $E$ be an $\epsilon$-dense sample of $X$ and let $S$ be the modified Vietoris-Rips complex constructed from $E$ as in Section \ref{sec:modified}. Then, $$H_i(S) \cong H_i(X)$$
\end{thm*}
Computing the bottlenecks is in general much easier than computing the reach. Therefore, in cases where the radius of the narrowest bottleneck is equal to the wfs we obtain sparser samples and considerable computational improvements for varieties with high curvature. We believe it to be true that the narrowest bottleneck equals the wfs for generic varieties, but this needs further investigation. In \exref{ex:sparse} we provide a comparison with existing algorithms in \cite{BCL17}.

For the higher homology
groups we provide the following result saying that we can bound the reach
from below using estimates of the \emph{local reach} $\tau_X(p)$ (see Section \ref{sec:edd}) for
sample points $p \in X$. The homology of $X$ can then be computed from
the \u{C}ech complex, $C(E, \epsilon)$, given by the nerve of the
union of all closed $\epsilon$-balls $\bar{B}_e(\epsilon)$ for
each $e\in E$.
\begin{thm*}\ref{homology}.
%\label{thm:local}
If $E \subset X$, $\epsilon > 0$, $E$ is an $(\epsilon/2)$-sample and
$\epsilon < \frac{4}{5}\tau_X(e)$ for all $e \in E$, then $X$ is a
deformation retract of $C(E, \epsilon)$. In particular, $X$ and $C(E,
\epsilon)$ have the same homology groups.
\end{thm*}

As estimating the reach of a manifold is computationally a very challenging problem, our approach provides an important simplification. The estimate is based on a finite process i.e. computing the local reach at finetly many sample points. It
is an interesting problem to further investigate the implications of
local approaches like this for the computational complexity of
homology groups. Roughly speaking we expect the size of a sample of
the ambient space $\RR^n$ to be exponential in $n$ but a sample of $X$ to be
exponential in $\dim{X}$.

\subsection*{Acknowledgements.} The authors would like to thank Parker Edwards for interesting discussions and Sascha Timme for helping with the implementation, which uses the HomotopyContinuation.jl \cite{HomotopyContinuation.jl} software package. Mathematica \cite{WolframMathematica} and Matplotlib \cite{matplotlib}
were used for creating images. The third author was supported by the Thematic Einstein Semester \textit{''Varieties, Polyhedra, Computation''} by the Berlin Einstein Foundation. All three authors acknowledge generous support by Vetenstapsr{\aa}det grants [NT:2014-4763], [NT:2018-03688].

\section{Geometrical background} \label{sec:edd}

Let $X \subseteq \CC^n$ be a smooth subvariety and let $q \in
\RR^n$. For $x \in X$, $N_xX \subseteq \CC^n$ denotes the embedded
normal space to $X$ at $x$. Consider the set of critical points with
respect to the squared Euclidean distance function from $X$ to
$q$: \[I(X,q)=\{x \in X: q \in N_xX\}.\] We call $I(X,q)$ the
\emph{normal locus} of $X$ with respect to $q$. Consider the real part
$X_{\RR} \subset \RR^n$ of $X$. We will use the notation
$I(X_{\RR},q)$ to denote the real points in $I(X,q)$.

The degree of $I(X,q)$ for generic $q \in \CC^n$ is called the
Euclidean Distance Degree (EDD) of $X$. If $X$ is a complete
intersection defined by $f_1,\dots,f_c \in \CC[x_1,\dots,x_n]$, then
$I(X,q)$ may be computed via the following system of equations:
\begin{align} \label{eq:edd}
  f_1(x)= \ldots =f_c(x) = 0 \\
  x-q = \sum_{i=1}^c \lambda_i \nabla f_i(x),
\end{align}
where $\lambda_i$ are auxiliary variables and $\nabla f_i$ denotes the
gradient of $f_i$. To compute $I(X,q)$, in case it is finite, one can
for instance use numerical homotopy methods \cite{bertini, SW05,
  BHSW13}. Note that the system \eqref{eq:edd} is linear in
$\lambda_i$ and so it is advantageous to use a multihomogeneous
homotopy to solve it.

Consider a smooth variety $X_{\CC} \subset \CC^n$ and its real part $X
\subseteq \RR^n$ and let $q \in \RR^n \setminus X$. One reason that
$I(X,q)$ is interesting is that it contains at least one point on
every connected component of $X$, namely the point closest to
$q$. This will be useful in the sampling algorithm presented in
\secref{sec:sampling}.

Below we will need to find a bounding box of a non-empty compact
smooth variety $X \subset \RR^n$. One way to do this is to compute the
normal locus $I(X,q)=\{x \in X:q \in \N{x}{X}\}$ with respect to a
random point $q \in \RR^n$. From $I(X,q)$ we get enough information to
find a bounding box or a bounding ball of $X$, that is a box or ball
that contains $X$. Let $\bar{B}_q(r) \subset \RR^n$ be the closed
ball of radius $r$ centered at $q \in \RR^n$. Since $X$ is compact,
for any $q \in \RR^n$ there is a point $x \in I(X,q)$ that maximizes
the distance from $q$ to points of $X$.  Hence $X \subset
\bar{B}_q(||x-q||)$. The solid hypercube with side length
$2||x-q||$ centered at $q$ is then a bounding box for $X$.

As above let $X \subset \RR^n$ be a compact smooth variety. A very
important invariant of the embedding $X \subset \RR^n$ of a variety is the \emph{reach} of $X$. For $p \in \RR^n$, let $\pi_X(p)$
denote the set of points $x \in X$ such that $||x-p||=\min_{x' \in X}
||x'-p||$. Also, for $r \geq 0$ let $X_r$ denote the tubular
neighborhood of $X$ of radius $r$, that is the set of points $x \in
\RR^n$ at distance less than $r$ from $X$. Then the reach $\tau_X$ of
$X$ may be defined as $\sup \{r \geq 0: |\pi_X(p)|=1\textrm{ for all }
p \in X_r \}$. For a point $q \in X$ we have a local notion called the
local reach, or the \emph{local feature size}, which is defined as
$\tau_X(q)=\sup \{r \geq 0: |\pi_X(p)|=1\textrm{ for all } p \in
B_q(r) \}$ where $B_q(r)=\{p \in \RR^n: ||p-q|| < r\}$. Then $\tau_X =
\inf_{q \in X} \tau_X(q)$. Since $X$ is compact the reach is positive
and if we in addition assume that $\dim{X}>0$, $X$ is not convex and
therefore the reach of $X$ is finite. The reach can be computed as a
minimum of two quantities \cite{ACKMRW17}, one of which derives from
local considerations in terms of the curvature of $X$ and the other
one is a global quantity arising from the bottlenecks of $X$. More
precisely, the global quantity is the radius of the narrowest
bottleneck of $X$. The local quantity, which we call $\rho$, is the
minimal radius of curvature on $X$, where the radius of curvature at a
point $x \in X$ is the reciprocal of the maximal curvature of a
geodesic passing through $x$. The global quantity is $\frac{1}{2}
\inf\{||x-y||:(x,y) \in X\times X, \, x\neq y, \, y \in \N{x}{X}, \,
x\in \N{y}{X}\}$ where $N_pX \subset \RR^n$ denotes the normal space
of $X$ at a point $p \in X$.

Let $\Delta_X = \{p \in \RR^n: |\pi_X(p)|>1\}$ and define the
\emph{medial axis} $M_X$ as the closure of $\Delta_X$. Let $U_X =
\RR^n \setminus M_X$ and note that the projection $\pi_X$ is a well
defined map $\pi_X:U_X \rightarrow X$ which maps a point $p$ to the
unique closest point of $X$. Also, the local reach $\tau_X(q)$ is the
distance from $q \in X$ to the medial axis and the reach $\tau_X$ is
the distance from $X$ to the medial axis.

\section{Bottlenecks and Reach} \label{sec:bottlenecks}

\subsection{Weak feature size}
We will first define a generalized version of bottlenecks inspired by the construction in \cite{10.1016/j.gmod.2005.01.002}. Let $X \subset \CC^n$ be a smooth variety. For a set of points $x_1, \dots, x_k \in X$ to define a bottleneck we first require there to be a point $y\in \CC^n$ such that the Euclidean distance from $y$ to $x_i$ is the same for each $0\leq i\leq k$ and that $y$ is in the normal space $N_X(x_i)$ for each $0\leq i\leq k$. Moreover, we would like $y$ to be \textit{critical} in some sense. It turns out that a good notion of critical is that $y$ is in the \textit{convex hull} of $x_1, \dots, x_k$. This means that there exists real non-negative numbers $a_1, \dots, a_k$ such that $\sum_{i=1}^k a_i = 1$ and $y = \sum_{i=1}^k a_i x_i$. We refer to \cite{10.1016/j.gmod.2005.01.002} for more details. To formalize this idea we make the following definitions.
\begin{defn}\label{def:bn_locus}
	Let $C \coloneqq \{(y, r) \in \CC^n \times \CC \mid ||x_i-y||^2 - r = 0, \text{ for $1\leq i\leq k$}\}$ and denote the closure of the projection $C \to \CC^n$ by $\rho(x_1, \dots, x_k)$. Then $\rho(x_1, \dots, x_k)$ is the set of centers of all $(n-1)$-spheres passing through the points $x_1, \dots, x_k$. Next, let $Conv(x_1, \dots, x_k)$ denote the convex hull of the points $x_1, \dots, x_k$. Define the $k$th \textit{bottleneck locus} to be the following set:
	$$ B_k \coloneqq \{ (x_1, \dots, x_k)\in \CC^{n\times k}\setminus \Delta \mid \cap_{i=1}^k N_X(x_i) \cap \rho(x_1, \dots, x_k) \cap Conv(x_1, \dots, x_k) \neq \emptyset \}$$
	where $\Delta \coloneqq \{ (x_1, \dots, x_k) \mid x_i=x_j \text{ for some $i, j$} \} \subset \CC^{n\times k}$.
\end{defn}
Note that $B_1=X$ and $B_2$ is what is usually defined as the bottleneck locus. In the case $k=2$ the conditions in the above definition simplify to algebraic conditions which means that $B_2$ is an algebraic variety in $\Bbb C^{2n}$. For $k>2$ the set $B_k$ is not algebraic but if we restrict $B_k$ to only allow real tuples $(x_1, \dots, x_k) \in \Bbb R^{n\times k}$ it is a semi-algebraic set. 

\begin{rem}
	In general, $B_k$ is empty for $k> EDD(X)$ since there is no point in $\CC^n$ which is in the normal space of more than $EDD(X)$ number of points on $X$. The only case for which $B_k$ is not empty for $k> EDD(X)$ is if there is a point $q\in \Bbb C^n$ which is in the normal space of infinitely many points on $X$. This is the case for instance if $X$ contains a component which is an $m$-sphere for some $m\geq 1$. In this case $B_k$ is infinite for all $k\in \Bbb N$. The sets $B_k$ for $k>EDD(X)$ don't contribute to new information for our purposes in this case and we will therefore not consider them.  % $B_k$ for $2\leq k\leq EDD(X)$, 
\end{rem} 

By restricting the sets $B_k$ to only allow tuples in $\Bbb R^n$ we can define a notion of \textit{narrowest} bottleneck. Let $B^{\Bbb R}_k$ denote this restriction.
\begin{defn}\label{def:wfs}
	Let $\ell_k \colon B^{\Bbb R}_k \to \RR$ be the function that takes $(x_1, \dots, x_k)$ to the radius of the smallest $(n-1)$-sphere satisfying the condition in Definition \ref{def:bn_locus}. Define $b_k$ to be the minimal value (or infimum in case $B^{\Bbb R}_k$ is not finite) of the image of $\ell_k$, called the \textit{radius of the narrowest bottleneck of degree $k$}. We now define the weak feature size as the following quantity:
	$$wfs(X) \coloneqq \underset{2\leq k \leq EDD(X)}{min} b_k$$
\end{defn}
For the remainder of this paper we will refer to elements of $B_2$ as the bottlenecks of $X$ and elements of $B_k$ for $k>2$ as generalized bottlenecks of degree $k$. The narrowest bottleneck of $X$ thus refers to the narrowest bottleneck of degree 2 of $X$.

Computing the generalized bottlenecks of $X$ is in general much harder than computing $b_2$. It is therefore interesting to understand in which cases the equality $wfs(X) = b_2$ holds. 

\begin{ex}
	Let $f_1(x, y) = 3x^2y - y^3 + 2(x^2+y^2)(x^2 + y^2 - 1)$ and $f_2(x, y) = 10x^2y^3 - 5x^4y - y^5+ 5(x^2+y^2)^2(x^2 + y^2 - 1)$. The curves $\{ f_1=0 \}$ and $\{ f_2=0 \}$ are two examples of varieties in $\Bbb R^2$ for which the weak feature size, $wfs(X)$, is smaller than the narrowest bottleneck $b_2$.
	% for varieties in $\Bbb R^2$. Let 
	%$$f_1(x, y) = 3x^2y - y^3 + 2(x^2+y^2)(x^2 + y^2 - 1) =  0$$
	%$$f_2(x, y) = 10x^2y^3 - 5x^4y - y^5+ 5(x^2+y^2)^2(x^2 + y^2 - 1) =  0$$
	\begin{figure}[ht] \vspace{-10pt}
	\centering
%	\subfloat[first][$f_1 = 0$\label{fig:wfs_bn_3}]{
%		\includegraphics[trim={8cm 6cm 8cm 2.5cm},clip,width=0.45\linewidth]
%		{wfs_example_3.pdf}}
%	\;
%	\subfloat[second][$f_2 = 0$\label{fig:wfs_bn_5}]{
%		\includegraphics[trim={9cm 5cm 9cm 3cm},clip,width=0.45\linewidth]
%		{wfs_example_5.pdf}}
%	\;
	\subfloat[third][Illustration of the $wfs$ and $b_2$ of the curve $\{f_1=0\}$.\label{fig:wfs_bn_3_lines}]{
		\includegraphics[trim={4.5cm 1cm 4.5cm 0cm},clip,width=0.45\linewidth]
		{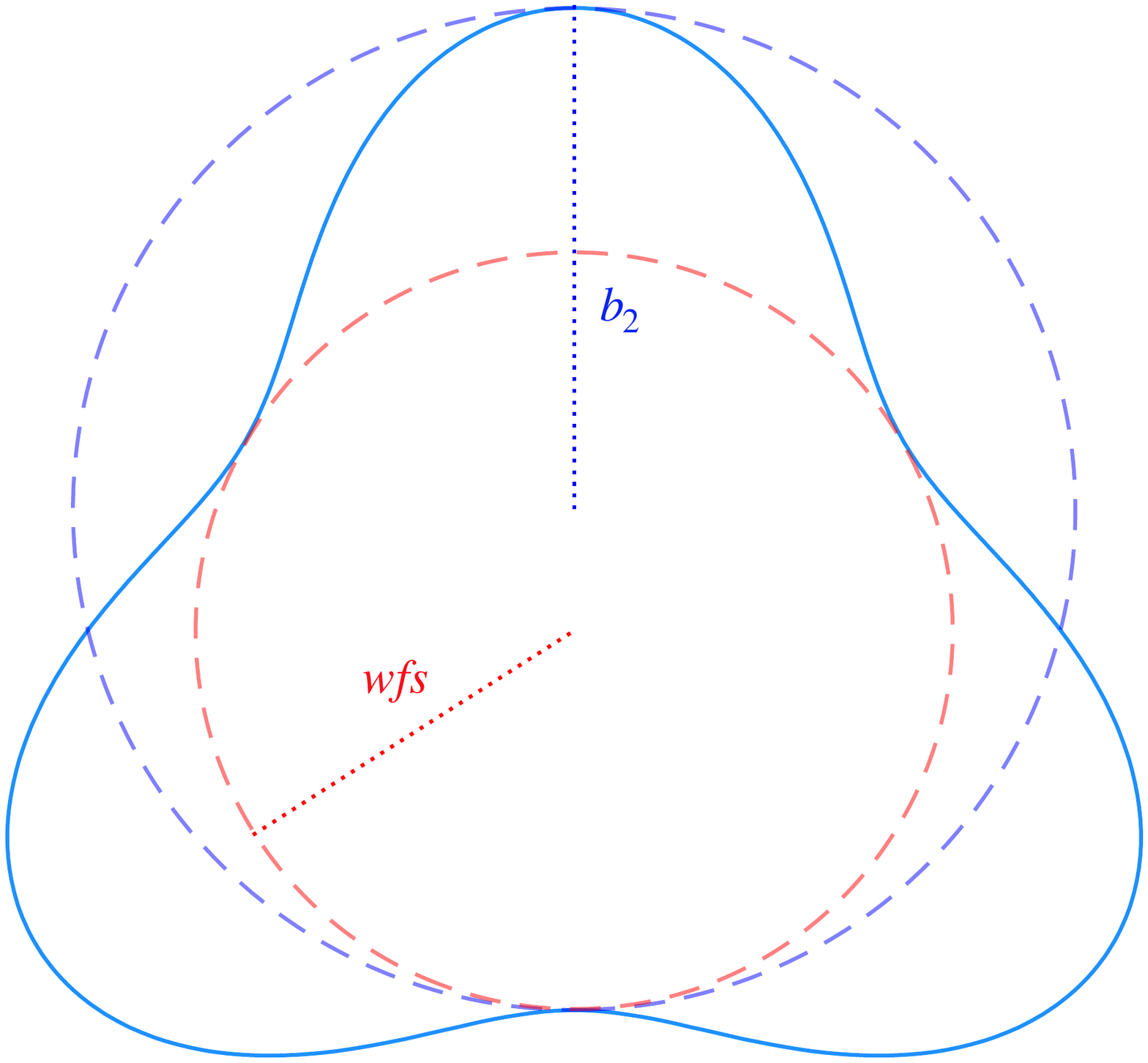}}
	\;
	\subfloat[fourth][Illustration of the $wfs$ and $b_2$ of the curve $\{f_2=0\}$.\label{fig:wfs_bn_5_lines}]{
		\includegraphics[trim={4cm 0cm 4cm 0cm},clip,width=0.45\linewidth]
		{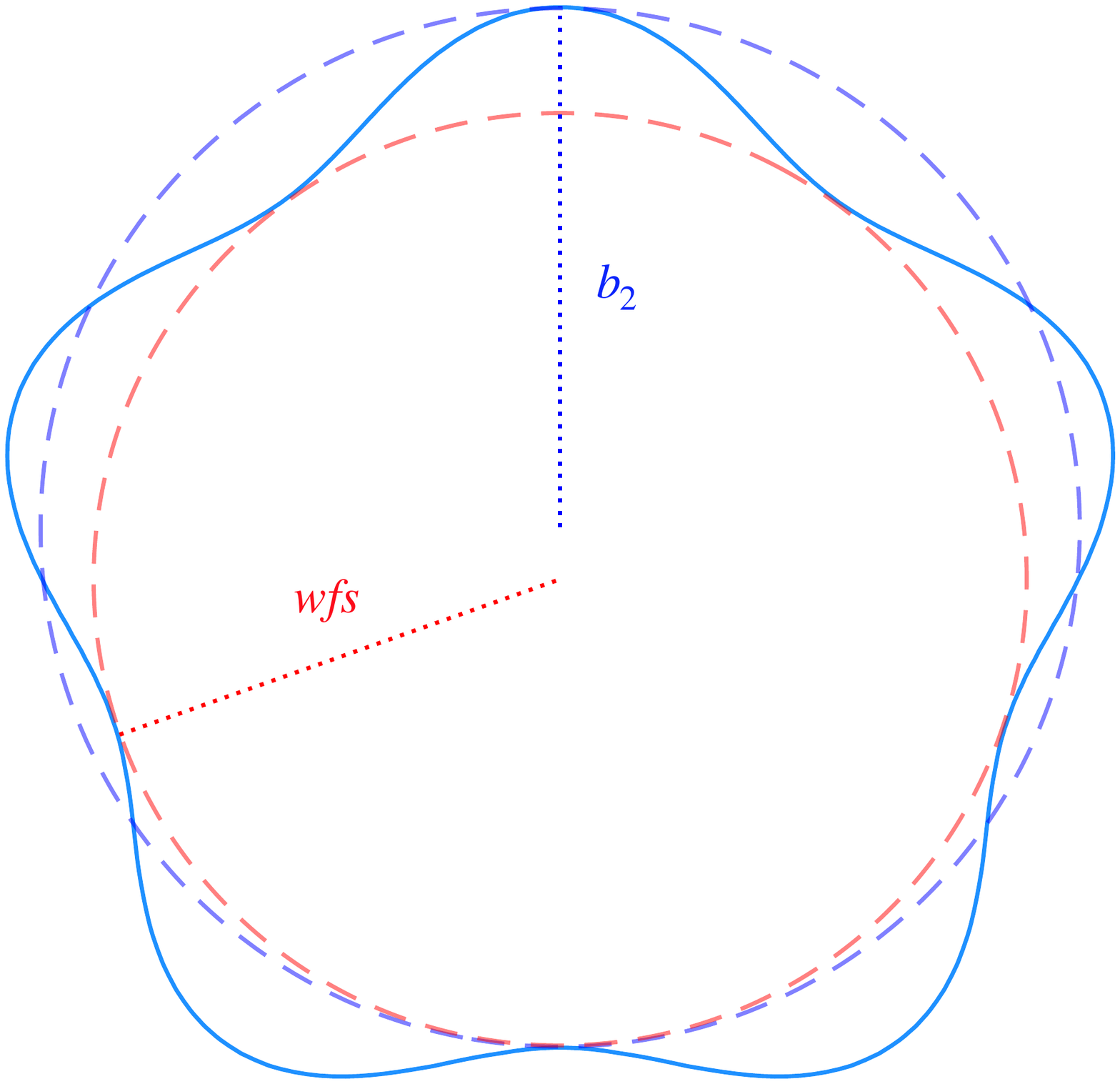}}
	\caption{Two examples of when the weak feature size is smaller than the narrowest bottleneck of a variety.
		\label{fig:wfs_bn_example}}
\end{figure}
\end{ex}
The varieties in Figure \ref{fig:wfs_bn_example} illustrate two examples in $\Bbb R^2$ of when the equality $wfs(X) = b_2$ fails. Both of these examples are oscillations of a circle in the sense that they can be described in the polar plane as $r = 1 + f(\cos(n\theta))$ for some function $f\colon \Bbb R\to \Bbb R$ and $n\in \Bbb Z$. Note that the curve segments between points of the generalized bottlenecks in Figure \ref{fig:wfs_bn_example} cannot bend too much, since then we would have a degree 2 bottleneck of smaller radius. It seems probable in $\Bbb R^2$ that $wfs(X) < b_2$ only if $X$ contains a component which is an oscillation of a circle, although this requires further investigation. We therefore make the conjecture that the equality $wfs(X) = b_2$ holds in general.

\subsection{Computing bottlenecks}
Let $X \subseteq \CC^n$ be a smooth variety defined by a system of
equations $F=(f_1,\dots,f_r) \subset \CC[x_1,\dots,x_n]$. The bottlenecks
of $X$ consists of pairs of distinct points $x,y \in X$ such that $(x-y) \perp
T_xX$ and $(x-y) \perp T_yX$. Here, $v \perp w$ means that
$\sum_{i=1}^n v_iw_i=0$ where $v=(v_1,\dots,v_n) \in \CC^n$ and
$w=(w_1,\dots,w_n) \in \CC^n$. Let $J_F$ denote the Jacobian matrix of
$F$, let $c=\codim{X}$ and assume that $c<n$. Expressed in terms of
defining equations, $(x,y)$ with $x \neq y$ is a bottleneck precisely
when the matrices
\[
\begin{pmatrix}
x-y\\
J_F(x)
\end{pmatrix}, \quad
\begin{pmatrix}
x-y\\
J_F(y)
\end{pmatrix}
\]
both have rank $\leq c$. The bottlenecks of $X$ are thus part of the
zero locus $Z \subseteq X\times X$ of the system given by the $(c+1)
\times (c+1)$-minors of these matrices together with the defining
equations $F(x)$ and $F(y)$. Then $Z \setminus \Delta$ describes the bottleneck locus $B_2$, where $\Delta \subset X \times X$ is the
diagonal.

An efficient method to compute all isolated bottlenecks of a smooth
algebraic variety in general position is presented in \cite{E18}. The
method makes use of numerical homotopy methods \cite{SW05,
  BHSW13}. Suppose that $X \subset \RR^n$ is smooth and compact. As we
saw in \secref{sec:edd}, the radius of the narrowest bottleneck is
intimately connected to the reach of $X$. Even though the bottleneck
locus is expected to be finite, it might not be and in this case we
would be required to solve an optimization problem to find the
narrowest bottleneck. The case of finitely many bottlenecks is
simpler since then computing all the bottlenecks suffices. For these
reasons (among others) it is useful to know when the bottleneck locus
is finite. In this section we provide a first step in this direction
by proving that general complete intersections have only finitely many
bottlenecks. Regarding the number of bottlenecks for general enough
algebraic varieties, a formula for low dimensional varieties is given in \cite{DEW19}.

Let $(d_1,\dots,d_c) \in \NN^c$ with $d_i \geq 1$ and let
$P_i=\PP^{N_i}$ where $N_i={d_i+n \choose d_i}-1$ be the parameter
space of hypersurfaces in $\CC^n$ of degree at most $d_i$. Complete
intersections in $\CC^n$ of codimension $c$ and degree type
$(d_1,\dots,d_c)$ are parameterized by an open subset $U \subset
\prod_{i=1}^c P_i$ and we write $X_u \subseteq \CC^n$ for the complete
intersection corresponding to $u \in U$. Let $V=\CC^n\times \CC^n
\setminus \Delta$, where $\Delta$ is the diagonal.

If $c=n$ the complete intersection itself is
finite, hence we assume that $c<n$. We may also assume that
none of the equations are linear, that is $d_i \geq 2$ for all
$i$. For a generic hyperplane $H \subset \CC^n$ there is a complex
orthogonal linear map which takes $H$ to the hyperplane $\{x_1=0\}$
and such a linear map preserves bottlenecks. Moreover, a smooth
variety $X \subseteq \CC^{n-1}$ with $\CC^{n-1} = \{x_1=0\} \subset
\CC^n$ has the same bottlenecks when seen as a subvariety of
$\CC^n$. By induction on the number of linear equations we may assume
that $d_i \geq 2$ for all $i$.

\begin{defn}
	Let the \emph{bottleneck correspondence} denote the following set:
	\[
	U \times V \supseteq \mathcal{B} \coloneqq \{(u,x,y): x,y \in X_u, (x-y) \perp
	T_xX_u, (x-y) \perp T_yX_u\}.
	\]
\end{defn}
In order to show that a general complete intersection has finite
bottleneck locus we need to show that the generic fiber of the
projection $\eta:\mathcal{B} \to U$ is finite. We do this by first
studying the dimension of the fibers of the projection
$\pi:\mathcal{B} \to V$.

Consider invertible affine transformations $g:\CC^n \to \CC^n$ of the
form $gx=R(sx+t)$ where $t \in \CC^n$ is a translation, $s \in \CC^*$
is a scaling and $R \in O(n)$ is a complex orthogonal map. Let $G$
denote the subgroup of all such transformations of the group of
invertible affine transformations of $\CC^n$. Then $G$ acts on $\CC^n$
and $G$ acts on $U$ by pre-composition of the defining equations with
$g^{-1}$ for $g \in G$. Moreover, $G$ acts on $\mathcal{B}$ via
$g(u,x,y)=(gu,gx,gy)$ and on $V$.
%% Notice in fact that moving any bottleneck for $u\in U,$ $(x,y)$ to the
%% two points $a,b$ such that $l^2=\|a-b\|^2=\|x-y\|^2$ via $g\in G$
%% preserves the property of $(a,b)$ being a bottleneck for $gu$ since
%% orthogonality is preserved.  Observe moreover that the orthogonal
%% group acts transitively on the quadric $\|x-y\|^2=l^2.$ Hence
%% $\dim{\pi^{-1}(x,y)}$ is the same for all $(x,y) \in V$. Let
%% $a=(0,\dots,0)$ and $b=(l,0,\dots,0)$. The variety $\pi^{-1}(a,b)$ is
%% a product of two determinantal varieties intersected with $2c$ linear
%% spaces.

% Find example of when bottlenecks are in the second orbit. For hypersurfaces, is it even possible to be in the second orbit generically or does it have probability zero?

\begin{lemma}
	Let $e_1, \dots, e_n$ denote the standard basis in $\CC^n$. The action of $G$ on $V$ has two orbits generated by the pairs $(e_1, 0)$ and $(e_1+i e_2, 0) $.
%	$$  \Bigg(\begin{pmatrix} 1\\ 0\\ \vdots \\ 0 \end{pmatrix}, \begin{pmatrix} 0\\ 0\\ \vdots \\ 0 \end{pmatrix}\Bigg),  \Bigg(\begin{pmatrix} 1\\ i\\ \vdots \\ 0 \end{pmatrix}, \begin{pmatrix} 0\\ 0\\ \vdots \\ 0 \end{pmatrix}\Bigg) $$
\end{lemma}
\begin{proof}
  The group $O(n)$ acts transitively on the sphere $S^{n-1}=\{x \in \CC^n:\sum_i x_i^2=1\}$. Therefore the equation $Mx = s e_1$, where $M\in O(n)$ and $s\in \CC^*$, has a solution whenever $x^T x \neq 0$.  Consequently, for any $(a,b) \in V$ with $(a-b)^T(a-b) \neq 0$, there exists a $g\in G$ such that $g(a, b) = (e_1, 0)$. Let $Q\coloneqq \{ (a, b)\in V \mid (a-b)^T(a-b) = 0 \}$. It remains to show that for any $(a, b)\in Q$, there is a $g\in G$ such that $g(a, b)=(e_1+ie_2, 0)$. This means that there is an $M \in O(n)$ and $s\in\CC^*$ such that $$sM(a-b) = e_1+ie_2 \Leftrightarrow s(a-b) = M^{-1}(e_1+ie_2) = m_1+im_2$$
	where $m_1, m_2$ are the first two columns of $M^{-1}$. Note that $(m_1+im_2)^T(m_1+im_2)=0$ since $m_1^T m_2 = 0$ and $m_1^Tm_1=m_2^Tm_2 =1$. It is straightforward to see that any non-zero vector $x\in \CC^n$ such that $x^Tx=0$ can be written as a sum $\frac{1}{s}(m_1+im_2)$ of two orthonormal vectors $m_1, m_2$ and a scalar $s$. Namely, for any non-zero entry $x_i$ of $x$, we have that $x=x_i (e_i + \frac{1}{x_i}(x-x_ie_i))$. 
\end{proof}
Note that the projection $\pi:\mathcal{B}\to V$ is equivariant under
the action of $G$. This means that $\dim{\pi^{-1}(x,y)}$ is the same
for all $(x,y) \in V$ within the same orbit. We will first compute
$\dim{\pi^{-1}(x,y)}$ for elements in the orbit generated by $(e_1,
0)$.
\begin{lemma}\label{lem:orbit}
	The codimension of $\pi^{-1}(e_1,0)$ as a subvariety of $U$ is $2n$.
\end{lemma}
\begin{proof}
For $u \in U$ let $F_u$ denote the corresponding system of equations and consider the $(c+1) \times n$-matrices
\[
\begin{pmatrix}
e_1\\
J_{F_u}(e_1)
\end{pmatrix}=
\begin{pmatrix}
1 &0 &\dots &0\\
&&J_{F_u}(e_1)&
\end{pmatrix}, \quad
\begin{pmatrix}
e_1\\
J_{F_u}(0)
\end{pmatrix}=
\begin{pmatrix}
1 &0 &\dots &0\\
&&J_{F_u}(0)&
\end{pmatrix}.
\]
The condition defining $\pi^{-1}(e_1,0)$ is that these matrices drop
rank and that $F_u(e_1) = 0$ and $F_u(0) = 0$. The condition that the above matrices drop in rank is equivalent to the condition that the $c \times
(n-1)$-matrices ${J}_{F'_u}(e_1)$ and ${J}_{F'_u}(0)$ drop rank, where
${J}_{F'_u}$ is $J_{F_u}$ with the first column removed.

We shall see that the entries of $F_u(e_1)$, $F_u(0)$,
${J}_{F'_u}(e_1)$ and ${J}_{F'_u}(0)$ are all independent as
homogeneous linear forms in $u \in U$. Independence across different
equations is clear since these involve disjoint groups of variables
coming from different spaces $P_i$. It is therefore enough to consider
the case $c=1$ of one polynomial $f_u \in
\CC[x_1,\dots,x_n]$. Moreover, it suffices to assume that $d_1 =
2$. Namely, if the there is a linear relation between the forms for
some $d_1>2$, the same linear relation holds between the forms in the
case $d_1=2$ (a set of dependent vectors in a vector space remain
dependent after projection to a lower dimensional space).

Suppose $f_u \in \CC[x_1,\dots,x_n]$ is a generic quadric and let
$\nabla f'_u$ denote $\nabla f_u$ with the first entry removed. We
will show that $f(e_1)$, $f(0)$ and the entries of $\nabla f_u'(e_1)$
and $\nabla f_u'(0)$ are all independent linear forms. Let
\begin{equation} \label{eq:quadric}
f_u(x) :=
\sum_{1\leq i \leq j \leq n } u_{ij} x_ix_j + \sum_{i=1}^n u_{0i}x_i +
u_{00}.
\end{equation}
Then by inspection we note that the following linear forms
are all linearly independent:
	\begin{align*}
	f(e_1) &= u_{11} + u_{01} + u_{00}\\
	f(0) &= u_{00}\\
	\nabla f_u'(e_1) &= \begin{bmatrix}
	u_{12} + u_{02} & u_{13} + u_{03} & \dots & u_{1n} + u_{0n}
	\end{bmatrix} \\
	\nabla f_u'(0) &= \begin{bmatrix}
	u_{02} & u_{03} & \dots & u_{0n}
	\end{bmatrix}
	\end{align*}
	
To compute $\dim{\pi^{-1}(e_1,0)}$ we may now apply standard results on the dimensions of linear determinantal varieties. We can express $\pi^{-1}(e_1,0)$ as a proper intersection of
three varieties: two determinantal varieties given by the rank
conditions on ${J}_{f'_u}(e_1)$ and ${J}_{f'_u}(0)$ and a linear space
defined by the $2c$ linear forms $f_u(e_1), f_u(0)$. The determinantal varieties both have codimension $n-c$ as subvarieties of $U$ by \cite{H92} Proposition 12.2. It follows that $\pi^{-1}(e_1,0)$ has codimension $2n$ as a subvariety of $U$. 
\end{proof}

It remains to compute $\dim{\pi^{-1}(x, y)}$ for the orbit generated by $(e_1+ie_2, 0)$.

\begin{lemma}\label{lem:orbit2}
	The codimension of $\pi^{-1}(e_1+ie_2, 0)$ as a subvariety of $U$ is $2n$ except in the special case when $c=1$ and deg$(X_u)=2$ in which case it is $2n-1$.
\end{lemma}
\begin{proof}
	Similar to the proof of Lemma \ref{lem:orbit} we consider the following $(c+1)\times n$-matrices:
	\begin{equation}\label{eq:lem34}
	\begin{pmatrix}
	e_1+ie_2\\
	J_{F_u}(e_1+ie_2)
	\end{pmatrix}=
	\begin{pmatrix}
	1 &i &\dots &0\\
	&&J_{F_u}(e_1+ie_2)&
	\end{pmatrix}, \quad
	\begin{pmatrix}
	e_1+ie_2\\
	J_{F_u}(0)
	\end{pmatrix}=
	\begin{pmatrix}
	1 &i &\dots &0\\
	&&J_{F_u}(0)&
	\end{pmatrix},
	\end{equation}
	and again, the condition defining $\pi^{-1}(e_1+ie_2, 0)$ is that
	these matrices drop in rank and that $F_u(e_1+ie_2)=0$ and
	$F_u(0)=0$. We may now perform the column operation where we subtract
	$i$ times the first column from the second column. Let $J_{F''_u}$
	denote the result when this operation is performed on $J_{F_u}$ and
	the first column removed. The condition defining
	$\pi^{-1}(e_1+ie_2, 0)$ can now be expressed as $J_{F''_u}(e_1+ie_2)$
	and $J_{F''_u}(0)$ dropping rank together with $F_u(e_1+ie_2)=0$ and
	$F_u(0)=0$.
	
	Consider first the case $c=1$. In this case the equations defining
	$\pi^{-1}(e_1+ie_2, 0)$ are linear and we need to determine the rank
	of this linear system. 
	% We need to determine the rank
	% it might drop after projection to degree 2 part...
	%As in the proof of Lemma \ref{lem:orbit} it
	%suffices to consider $d_1=2$ (how?).
	Let $f_u \in \CC[x_1,\dots,x_n]$ be the polynomial defining $X_u$. We assume that $f_u$ is a general polynomial of degree three in $ \CC[x_1,\dots,x_n]$. We will see that this describes the general case and that we get a special case when deg$(f_u)=2$. The polynomial is represented as follows:
	\begin{equation} \label{eq:cubic}
	f_u(x) := \sum_{1\leq i \leq j \leq k \leq n } u_{ijk} x_ix_jx_k +
	\sum_{1\leq i \leq j \leq n } u_{0ij} x_ix_j + \sum_{i=1}^n u_{00i}x_i +
	u_{000}.
	\end{equation}

	\begin{align*}
	f_u(e_1+ie_2) &= u_{111} + iu_{112} - u_{122} - iu_{222} + u_{011} + iu_{012} - u_{022} + u_{001} + iu_{002} + u_{000}\\
	f_u(0) &= u_{000}\\
	\nabla f_u''(e_1+ie_2) &= [u_{112} + 2i u_{122} - 3u_{222} + u_{12}+ 2iu_{22}  + u_{02} - i(3u_{111} + 2i u_{112} - u_{122}+ 2u_{11} + \\
	&  \ \ \ \ + iu_{12} + u_{01} )\ \ \ \    u_{113} + iu_{123} - u_{223} + u_{13} + iu_{23} + u_{03} \ \ \ \   \dots \ \ \ \   \\
	& \ \ \ \ \ u_{11n} + iu_{12n} - u_{22n} + u_{1n} + iu_{2n} + u_{0n}]\\
	\nabla f_u''(0) &= \begin{bmatrix} u_{02} - iu_{01} &  u_{03} & \dots &  u_{0n}\end{bmatrix}
	\end{align*}
	
	%\begin{align*}
	%f_u(e_1+ie_2) &= u_{11} + iu_{12} - u_{22} + u_{01} + iu_{02} + u_{00}\\
	%f_u(0) &= u_{00}\\
	%\nabla f_u''(e_1+ie_2) &= \begin{bmatrix}
	%u_{12}+ 2iu_{22}  + u_{02} - i( 2u_{11} + iu_{12} + u_{01} ) & u_{13} + iu_{23} + u_{03} & \dots & u_{1n} + iu_{2n} + u_{0n}
	%\end{bmatrix}\\
	%\nabla f_u''(0) &= \begin{bmatrix}
	%u_{02} - iu_{01} & u_{03} & \dots & u_{0n}
	%\end{bmatrix}
	%\end{align*}

	%Let $l_1$ and $l_2$ denote the hyperplanes in $U$ given by $f_u(e_1+ie_2) =0$ and $f_u(0)=0$ respectively. 
	The principal minors of $\nabla f_u''(e_1+ie_2)$ and $\nabla f_u''(0)$ are systems of linear equations. Consequently, the system of equations determining $\pi^{-1}(e_1+ie_2, 0)$ is the zero set of the above linear forms. It is easily verified that all the above linear forms are linearly independent, which means that the linear system determining $\pi^{-1}(e_1+ie_2, 0)$ has full rank. Since there are $2n$ linear forms this means that the rank of the system is $2n$ and thus  the codimension of $\pi^{-1}(e_1+ie_2, 0)$ as a subvariety in $U$ is $2n$. By an analogous argument as in Lemma \ref{lem:orbit} it then follows that this result holds for any general polynomial of degree $\geq 3$ in $\CC[x_1,\dots,x_n]$. 
	
	When $f_u$ has degree two we get a special case. The system describing $\pi^{-1}(e_1+ie_2, 0)$ is now the above system where we delete all coefficients corresponding to degree three terms. We then note that  there is a linear dependence between the linear forms $f_u(e_1+ie_2)$, $f_u(0)$ and the first linear forms of the vectors $\nabla f_u''(e_1+ie_2)$ and $\nabla f_u''(0)$. This dependence is described by the following relation: 
	$$2f_u(e_1+ie_2) - 2f_u(0) - \big[u_{12}+ 2iu_{22}  + u_{02} - i(2u_{11} + iu_{12} + u_{01})\big] + \big[u_{02} - iu_{01}\big] = 0$$
	This is the only relation between the above linear forms, let $\ell$ denote this linear relation. Hence the system of equations defining $\pi^{-1}(e_1+ie_2, 0)$ is a linear system of rank $2n-1$. It follows that the codimension of $\pi^{-1}(e_1+ie_2, 0)$ as a subvariety in $U$ is $2n-1$.

	%two determinental varieties in (\ref{eq:lem34}) and $l_1$. In this case these varieties are defined by linear  equations. By the same arguments as in Lemma \ref{lem:orbit}, the intersection of the two determinental varieties and $l_1$ is a variety, $W$, of codimension $2n-1$ in $U$. Note now however, that the intersection $W\cap l_2 = W$ since $l_2$ is linearly dependent with the equations describing $W$. Consequently, the codimension of the intersection $l_2\cap W$ is $2n-1$.
	
	%These consists of independent linear forms and therefore, by the same arguments as in Lemma \ref{lem:orbit}, their intersection is a variety, $W$, of codimension $2n-1$ in $U$. We now claim that the hyperplane $l_2$ intersected with $W$ cuts its dimension down by one, even though it introduces a linear dependence. For the dimension to not change, we must have that a component of $W$ is contained in $l_2$. This happens only in special cases and thus for a generic $u$ it will not happen. Consequently, the codimension of the intersection $l_2\cap W$ is $2n$. From this the statement follows.
	
	%Not correct. Look at the determinental varieties. They are just linear.. This is just a linear system and it contains the equation $l_2$, therefore the dimension does not change and the codimension is $2n-1$. 
	
	Consider now the case $c>1$ and $d_1=\dots =d_c=2$. First note that the two determinantal varieties cut out by the principal minors of $J_{F''_u}(e_1+ie_2)$ and $J_{F''_u}(0)$ consist of independent linear forms and their intersection thus has codimension $2n-2c$. Note that since they each consist of disjoint variable groups it follows that their intersection is in fact a product. Each of the two determinantal varieties are by \cite{CM_1971__23_4_407_0} irreducible and reduced and the defining equations are homogeneous of degree $c$. Now let $u=(u_1, \dots, u_c)$ denote the coefficients of the polynomials determining $X_u$. By the above results, we have two linear equations, $f_{u_j}(e_1+ie_2) = 0$ and $f_{u_j}(0) = 0$, for each $1\leq j\leq c$ yielding $2c$ linear equations which are all linearly independent. Let $\ell_{u_i}$ denote the linear relation in the $c=1$ case corresponding to the $i$th defining equation of $X_u$. Note that since the determinantal varieties have homogeneous defining equations of degree $c$, there is no polynomial of degree $\leq c$ contained in these ideals. Thus no product $\prod_{i \in J\subset\{ 1, \dots, c \}} \ell_{u_i}$ is contained in the ideals if $|J|<c$. It is then a question if $\prod_{i \in \{ 1, \dots, c \}} \ell_{u_i}$ is in either of the determinantal ideals. But this cannot be since it would correspond to the product of the elements in one column of $J_{F''_u}(e_1+ie_2)$ or $J_{F''_u}(0)$  to be in the corresponding determinantal ideal. This cannot happen since the linear forms of each matrix are linearly independent. It follows that the intersection of the determinantal varieties and the $2c$ linear forms is transverse and that the codimension of the intersection is $2n$ in $U$. The codimension of the analogous intersection when $d_i\geq 2$ for $1\leq i\leq c$ when $c>1$ is also $2n$ in $U$. Thus we conclude that for $c>1$, the codimension of $\pi^{-1}(e_1+ie_2, 0)$ as a subvariety in $U$ is $2n$. 
\end{proof}

%This means that intersecting the two determinantal varieties with the $2c$ hyperplanes yields a variety of codimension $2n$, since the equations defining the hyperplanes are all linearly independent.

%The hyperplanes $l_{1i}$ for $1\leq i\leq c$ are also all independent of the linear forms in the determinantal varieties and the intersection between all these varieties is a variety, $W$, of codimension $2n-c$. The question is now if the linear hyperplanes $l_{2i}$ will cut the dimension of $W$. 

%Note that the determinantal varieties are both irreducible and thus not contained in a hyperplane. Their intersection can again be expressed as a determinantal variety and thus is also irreducible. Consequently the intersection is not contained  in any affine linear space. This means that the codimension is $2n$ of the intersection of 

%Therefore, by the same arguments as in Lemma \ref{lem:orbit}, $\codimsub{U}{\pi^{-1}(e_1+ie_2, 0)} = 2n$.
%Therefore, by the same argument as in Lemma \ref{lem:orbit},
%$\codimsub{U}{\pi^{-1}(e_1+ie_2, 0)} = 2n$.

We are now ready to state the following:
\begin{thm} \label{thm:finite-bn}
A generic complete intersection has only finitely many bottlenecks.
\end{thm}
\begin{proof}
We will show that $\dim{\mathcal{B}}=\dim{U}$. It follows that the
general fiber of the projection $\eta: \mathcal{B} \to U$ is finite.

By Lemma \ref{lem:orbit} and \ref{lem:orbit2}, $\pi: \mathcal{B} \to
V$ is onto and the general fiber has dimension $\dim{U}-2n$. Since
$\dim{V}=2n$, it follows that $\mathcal{B}$ has an irreducible
component of dimension $\dim{U}$. Now let $\mathcal{B}' \subseteq
\mathcal{B}$ be an irreducible component. Suppose first that
$\pi(\mathcal{B}')$ is not contained in $Q = \{ (x, y)\in V \mid
(x-y)^T(x-y) = 0 \}$. Then $\dim{\mathcal{B}' \cap \pi^{-1}(x,y)} \leq
\dim{U}-2n$ for generic $(x,y) \in \pi(\mathcal{B}')$ by
\lemmaref{lem:orbit}. Hence $\dim{\mathcal{B}'} \leq \dim{U}$. Now
assume that $\pi(\mathcal{B}') \subseteq Q$. Then $\dim{\mathcal{B}'
  \cap \pi^{-1}(x,y)} \leq \dim{U}-2n+1$ for generic $(x,y) \in
\pi(\mathcal{B}')$ by \lemmaref{lem:orbit2}. But
$\dim{\pi(\mathcal{B}')} \leq 2n-1$ in this case and hence
$\dim{\mathcal{B}'} \leq \dim{U}$.
\end{proof}

\subsection{Locally estimating the reach}\label{sec:local_reach}
In this section we briefly explain how the reach of a smooth variety
$X \subset \RR^n$ can be estimated from below by sampling
the variety and estimating the
local reach. We will assume that $X$ is a subvariety of the unit
sphere $S^{n-1} \subset \RR^n$ defined by an ideal generated by
homogeneous polynomials $F=(f_1,\dots,f_q)$ of degrees $d_1,\dots,d_q$
with $q \leq n$.

The local reach is estimated using a result of \cite{CKS18, BCL17}
together with a result from Schub Smale theory. The first result
bounds the local reach at a point $x \in X$ in terms of the
\emph{$\gamma$-number} of the system $F$ at $x$ and the second result
bounds the $\gamma$-number in terms of the \emph{condition number} of
$F$ at $x$. To set this up we will restate some notation from
\cite{BCL17}. Consider a homogeneous polynomial $h = \sum_{|a|=d} h_a
x^a$ where $(a_1,\dots,a_n)$ is a multi index, $|a|=a_1 + \dots +a_n$,
$d$ is the degree of $h$, $h_a$ its coefficients and
$x^a=x_1^{a_1}\dots x_n^{a_n}$. The Weil norm $||h||_w$ of $h$ is
defined by $||h||_w^2=\sum_{|a|=d} h_a^2\binom{d}{a}^{-1}$ where
$\binom{d}{a}=\frac{d!}{a_1! \cdots a_n!}$ are multinomial
coefficients. The norm of the system $F=(f_1,\dots,f_q)$ is defined by
$||F||^2=\sum_{i=1}^q ||f_i||_w^2$. Let $X$ be the zero-set of $F$ and
let let $DF$ be the Jacobian matrix of $F$. For $x \in X$ consider the
pseudo-inverse $DF(x)^{\dagger}=DF(x)^T(DF(x)DF(x)^T)^{-1}$ of
$DF(x)$. Also, let $\Delta$ denote the diagonal matrix with diagonal
$(\sqrt{d_1},\dots,\sqrt{d_q})$ and for a matrix $M$ let
$||M||_{\text{spec}}$ denote the spectral norm of $M$, that is the
square root of the largest eigenvalue of the matrix $M^TM$. We may now
define the condition number of $F$ at $x \in X$ as
$\mu_{\text{norm}}(F,x)=||F|| \cdot
||DF(x)^{\dagger}\Delta||_{\text{spec}}$. Letting $D^kF$ for $k\geq 2$
denote the higher order derivatives of $F$, the $\gamma$-number is
defined as
\[
\gamma (F,x) = \sup_{k \geq 2}
\frac{1}{k!}||DF(x)D^kF(x)||^{1/(k-1)}.
\]
Let $D=\max\{d_1,\dots,d_q\}$. Now we state the first inequality,
which is Lemma 2.1(b) of \cite{SS96}:
% statement is actually for mu_proj...
% but mu_proj(F,x)=mu_norm(F,x) if x in X Lemma 4.3 in BCL17
\[2\gamma(F,x) \leq D^{3/2} \mu_{\text{norm}}(F,x).\]
The second inequality we use was first proved in \cite{CKS18} and
appears as Theorem 3.3 of \cite{BCL17}:
\[
\frac{1}{14\gamma (F,x)} \leq \tau_X(x).
\]
Putting this together we get a lower bound on the local reach:
\begin{equation} \label{eq:localreach}
	\frac{1}{7D^{3/2}\mu_{\text{norm}}(F,x)} \leq \tau_X(x).
\end{equation}

We now turn to the (global) reach $\tau_X$.
\begin{lemma} \label{lemma:reach}
	Let $E \subseteq X$ be an $\epsilon$-sample and let $m \leq \tau_X(e)$
	for all $e \in E$. Then $m-\epsilon \leq \tau_X$.
\end{lemma}
\begin{proof}
	Let $x \in X$ be a point that realizes the distance from $X$ to the
	medial axis $M_X$ of $X$ and let $y \in M_X$ be a closest point to $x$
	on $M_X$. This means that $||x-y||=d(x,M_X)=\tau_X$ where
	$d(\cdot,\cdot)$ is the distance function. Since $E$ is an
	$\epsilon$-sample, there is an $e \in E$ with $||e-x|| \leq
	\epsilon$. The claim now follows from \[m \leq \tau_X(e) = d(e,M_X)
	\leq ||e-y|| \leq ||e-x|| + ||x-y|| = ||e-x|| + \tau_X \leq \epsilon +
	\tau_X.\]
\end{proof}

Let $E \subset X$ be a finite $\epsilon$-sample. By
\lemmaref{lemma:reach} and (\ref{eq:localreach}) we have that
$m-\epsilon \leq \tau_X$, if $m = \min_{e \in E} \eta(e)$ where
$\eta(x)^{-1}=7D^{3/2}\mu_{\text{norm}}(F,x)$. This, together with the
sampling method of \secref{sec:sampling}, gives the following
algorithm to bound the reach of $X$ from below.
%
%\begin{algorithm} 
%	\caption{Find sampling density.}\label{alg:reach}
%	\begin{algorithmic} 
%		\REQUIRE Variety $X \subset S^{n-1}$ and $\epsilon_0 >0$.
%		\ENSURE Required sampling density $\epsilon$.
%		\STATE Let $\epsilon = 2\epsilon_0$ and $m=0$.
%		\WHILE{$\frac{4}{5}m-\epsilon \leq 0$}
%		\STATE Let $\epsilon=\epsilon/2$.
%		\STATE Compute a $(\epsilon/2)$-sample $E \subset X$ of $X$.
%		\STATE Compute $m=\min_{e \in E} \eta(e)$.
%		\ENDWHILE
%		\RETURN $\epsilon$
%	\end{algorithmic}
%\end{algorithm}

\begin{algorithm} 
	\caption{Bound reach from below.}\label{alg:reach}
	\begin{algorithmic} 
		\REQUIRE Variety $X \subset S^{n-1}$ and $\epsilon_0 >0$.
		\ENSURE Lower bound on $\tau_X$.
		\STATE Let $\epsilon = 2\epsilon_0$ and $m=0$.
		\WHILE{$m-\epsilon \leq 0$}
		\STATE Let $\epsilon=\epsilon/2$.
		\STATE Compute a $\epsilon$-sample $E \subset X$ of $X$.
		\STATE Compute $m=\min_{e \in E} \eta(e)$.
		\ENDWHILE
		\RETURN $m-\epsilon$
	\end{algorithmic}
\end{algorithm}

Let $\eta = \min_{x \in X} \eta(x)$. Since (after the first pass
through the while loop) $\eta \leq m$, the algorithm terminates when
$\epsilon < \eta$. One may try to achieve this after one sampling by
choosing $\epsilon_0$ small enough.
%The algorithm terminates as soon
%as the lower bound $m-\epsilon$ on $\tau_X$ is positive.
%To try to
%improve this, one might want to let $\epsilon_0$ be a fraction of
%$\epsilon$ and run the algorithm again.
% a bit weird explanation...

% this example took 50 min with split=nprocs=8

% narrowest bottleneck: 0.7653668647301793

% explain both basic and extra sample in words!
% comment more on pertubation
\section{Sampling} \label{sec:sampling}

Let $X \subseteq \RR^n$ be a smooth compact variety of pure dimension
$d \geq 1$. In this section we describe a simple method to compute a
finite sample on $X$ given defining equations for $X$. The sample will
have a guaranteed prescribed density. To compute the sample one may
use tools from numerical algebraic geometry \cite{bertini}.
%The purpose of this section is to
%give a simple sampling procedure
% with complexity $\mathcal{O}(\epsilon^{-d})$, where 1/epsilon is density?

\subsection{The basic sample}
For $1 \leq k \leq n$ let $T_k$ be the set of subsets of
$\{1,\dots,n\}$ with $k$ elements. For $t=\{t_1,\dots,t_k\} \in T_k$
let $V_t \subseteq \RR^n$ be the $k$-dimensional coordinate plane
spanned by $e_{t_1},\dots,e_{t_k}$ where $e_1,\dots,e_n$ are the
standard basis vectors of $\RR^n$. Let $\delta>0$ and consider the
grid $G_t(\delta) \subset V_t$ which is equivalent to the grid $\delta
\ZZ^k \subset \RR^k$.
%When $k=1$ we will simply write $G_i(\delta)$
%for $G_{\{i\}}(\delta)$.
Let $\pi_t:\RR^n \rightarrow V_t$ be the projection. We then have the
basic sample
\begin{equation} \label{eq:sample}
E_{\delta} = \bigcup_{\substack{t \in T_d \\ g \in G_t(\delta)}} X\cap
\pi_t^{-1}(g).
\end{equation}
The basic sample $E_{\delta}$ consists of intersections between $X$
and complementary dimensional linear spaces $\pi_t^{-1}(g)$ ranging
over $t \in V_t$ and $g \in G_t(\delta)$. The linear spaces are made
up of the complementary dimensional faces of a cubical tessellation
with side length $\delta$. If the grids $G_t(\delta)$ are modified by a
general perturbation this sample is finite. In practice we let the
perturbation be random and simply translate the grids by a random
vector.
%In the notation above, $G_i(\delta)$ is a 1-dimensional grid on the
%$i$-th coordinate axis.

\begin{ex} \label{ex:curve}
In this example we illustrate the basic sample $E_{\delta}$ in the
case of a plane curve $X \subset \RR^2$. The basic sample consists of
the intersection of $X$ with a square grid with side length $\delta$
filling the plane. To acquire the grid we perturb the standard grid
(which has a vertex at the origin) by a random vector. An example
curve and its intersection with the grid may be seen in
\figref{fig:curve-gridlines}. Let $\pi_1$ and $\pi_2$ be the projections to
the $x$-axis and $y$-axis, respectively. The grid lines may be split
in vertical and horizontal lines. The vertical lines correspond to
$\pi_1^{-1}(g)$ with $g$ in a grid of size $\delta$ on the $x$-axis,
see \figref{fig:curve-vert}. Similarly, \figref{fig:curve-hor} shows
the intersection of $X$ with the horizontal lines corresponding to
$\pi_2^{-1}(g)$ with $g$ in a grid on the $y$-axis.
\begin{figure}[ht]
\centering
\subfloat[][Basic sample.\label{fig:curve-gridlines}]{
  \includegraphics[trim={0cm 0cm 0cm 0cm},clip,scale=0.33]
                  {curve_tot.pdf}}
\;
\subfloat[][Vertical grid lines.\label{fig:curve-vert}]{
  \includegraphics[trim={0cm 0cm 0cm 0cm},clip,scale=0.33]
                  {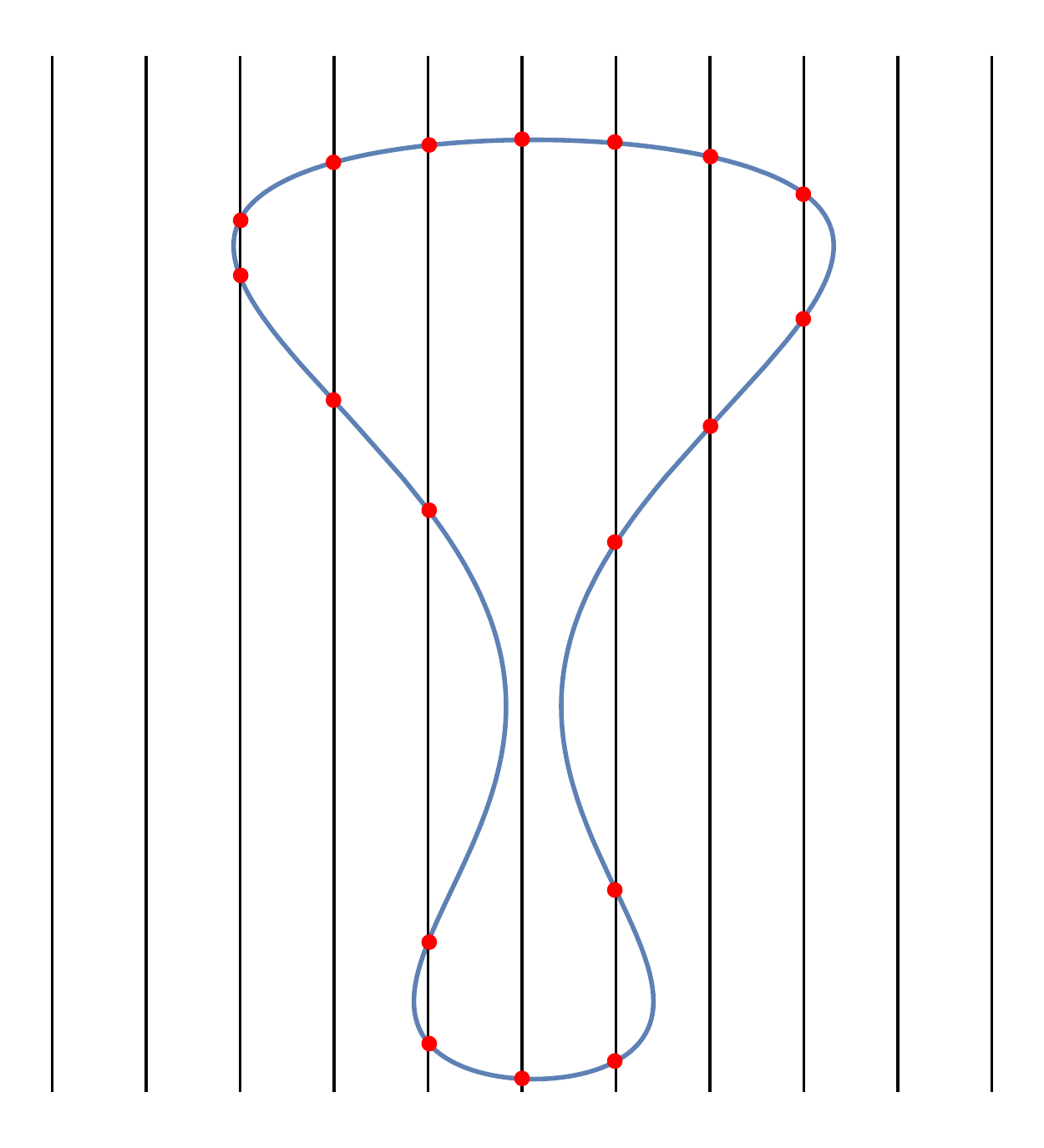}}
\;
\subfloat[][Horizontal grid lines.\label{fig:curve-hor}]{
  \includegraphics[trim={0cm 0cm 0cm 0cm},clip,scale=0.33]
                  {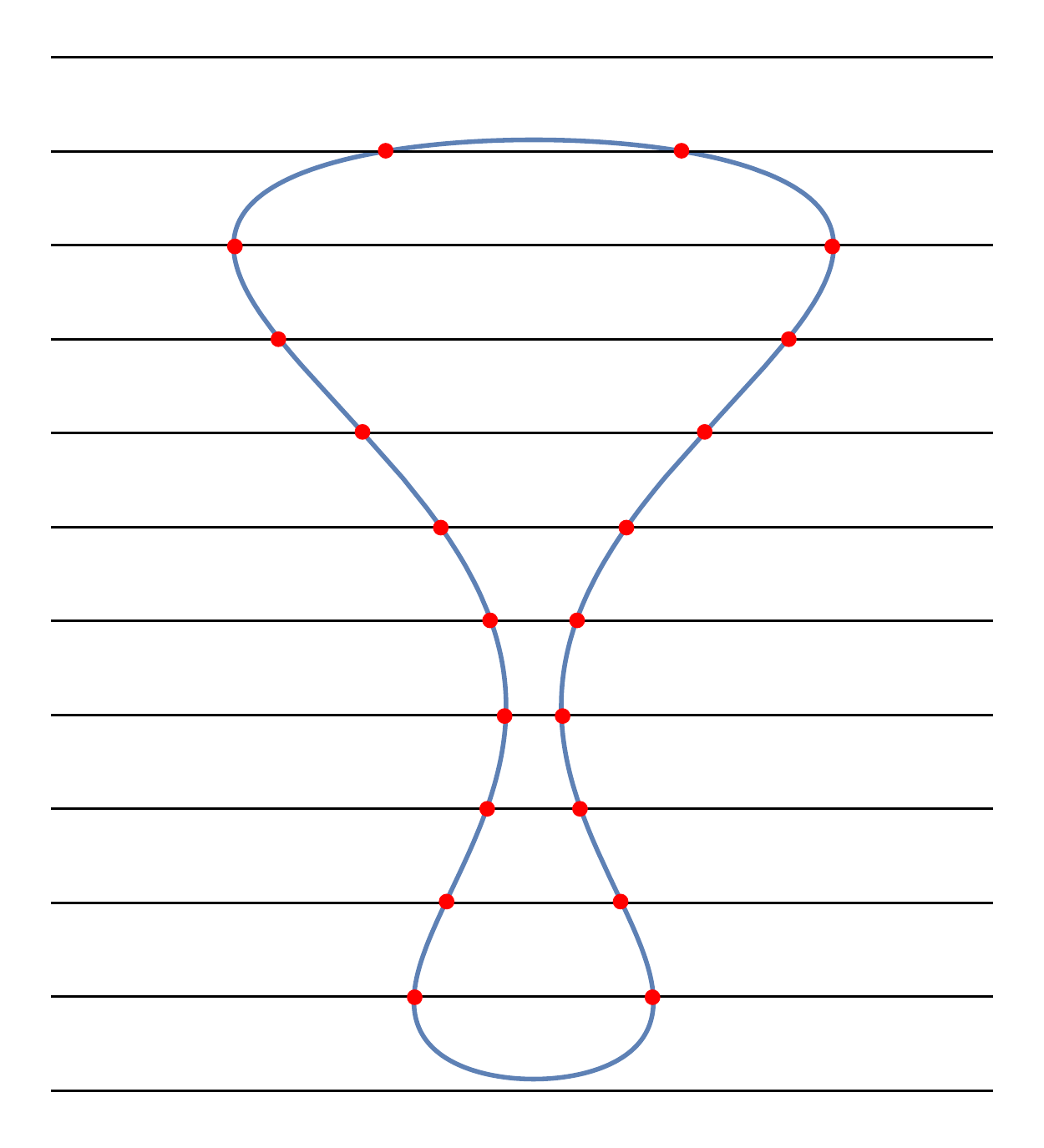}}
\caption{Basic sample of planar curve.
\label{fig:planar-curve}}
\end{figure}
\end{ex}

\begin{ex} \label{ex:surf}
As another example consider the surface $X \subset \RR^3$ shown in
\figref{fig:surf_grid}. The figure also shows the 1-dimensional
skeleton of a cubical complex, or tessellation, filling $\RR^3$. The
tessellation is a random translation of the standard cubical
tessellation with cube side length $\delta > 0$ and the image is shown
from an angle for better visualization. The 1-dimensional skeleton is
the union of all the edges of the cubes in the complex and the basic
sample $E_{\delta}$ is the intersection of $X$ with the 1-dimensional
skeleton. As mentioned above we have in general that the basic sample
$E_{\delta}$ of a variety $X \subset \RR^n$ is the intersection of $X$
with the $\codim{X}$-skeleton of a perturbed cubical complex. In the case of a surface in $\RR^3$, this
intersection may be done by intersecting $X$ with three families of
lines as shown in Figures~\ref{fig:surf_xygrid},
\ref{fig:surf_xzgrid} and \ref{fig:surf_yzgrid}, respectively. All
the lines together are shown in \figref{fig:surf_fullgrid}. These
lines can be expressed in terms of projections as follows. Let
$\pi_{12}$, $\pi_{13}$ and $\pi_{23}$ be the projections to the
$xy$-plane, the $xz$-plane and the $yz$-plane, respectively. Then the
lines in \figref{fig:surf_xygrid} are given by $\pi^{-1}_{12}(g)$
with $g$ in a two-dimensional grid in the $xy$-plane like the one
shown in \figref{fig:curve-gridlines}. Similarly, the lines in
Figures~\ref{fig:surf_xzgrid} and \ref{fig:surf_yzgrid} may be
expressed using $\pi_{13}$ and $\pi_{23}$, respectively.

\begin{figure}[ht]
\centering
\subfloat[][\label{fig:surf_fullgrid}]{
  \includegraphics[trim={0cm 0cm 0cm 0cm},clip,scale=0.25]
                  {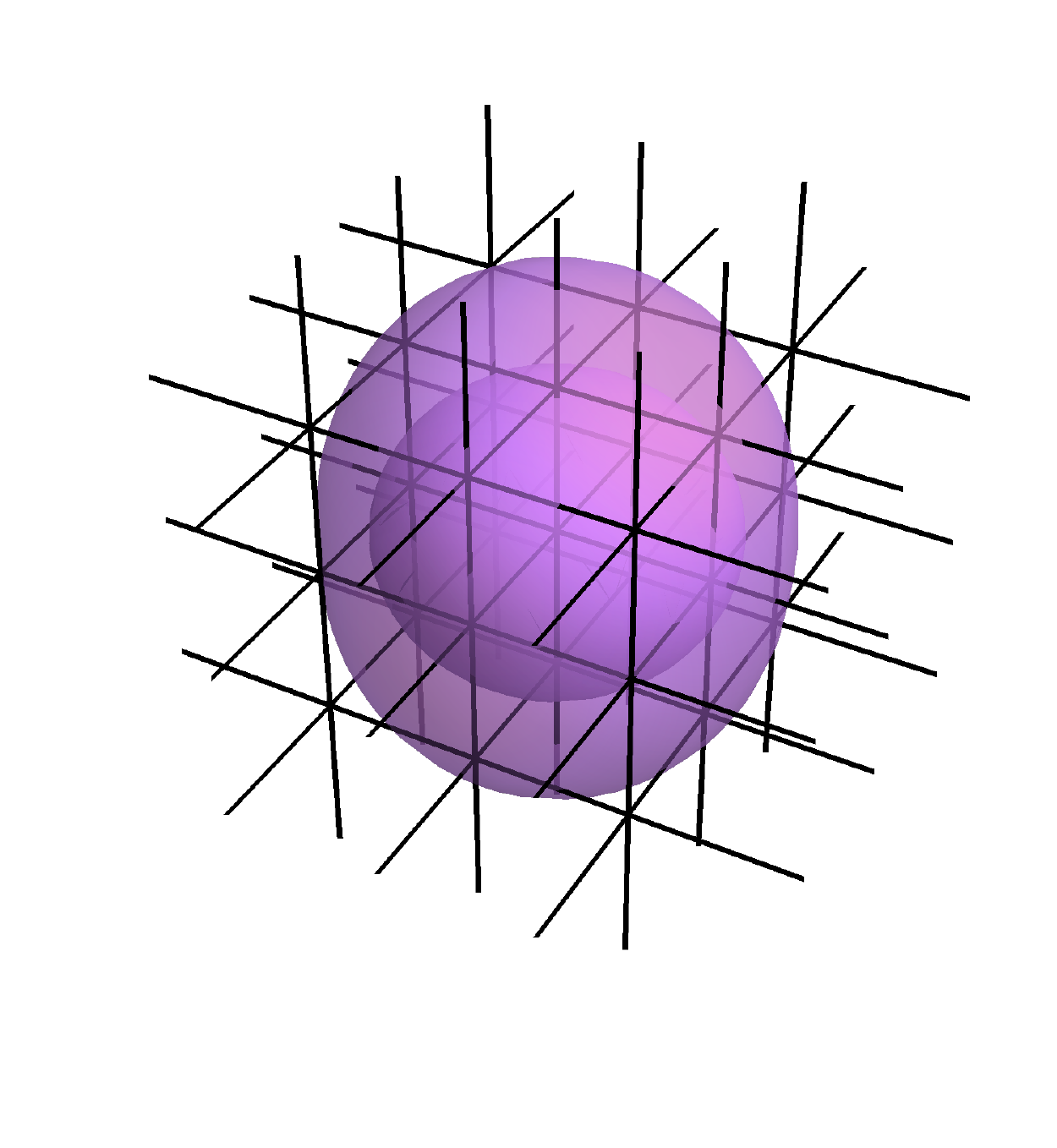}}
\;
\subfloat[][\label{fig:surf_xygrid}]{
  \includegraphics[trim={0cm 0cm 0cm 0cm},clip,scale=0.25]
                  {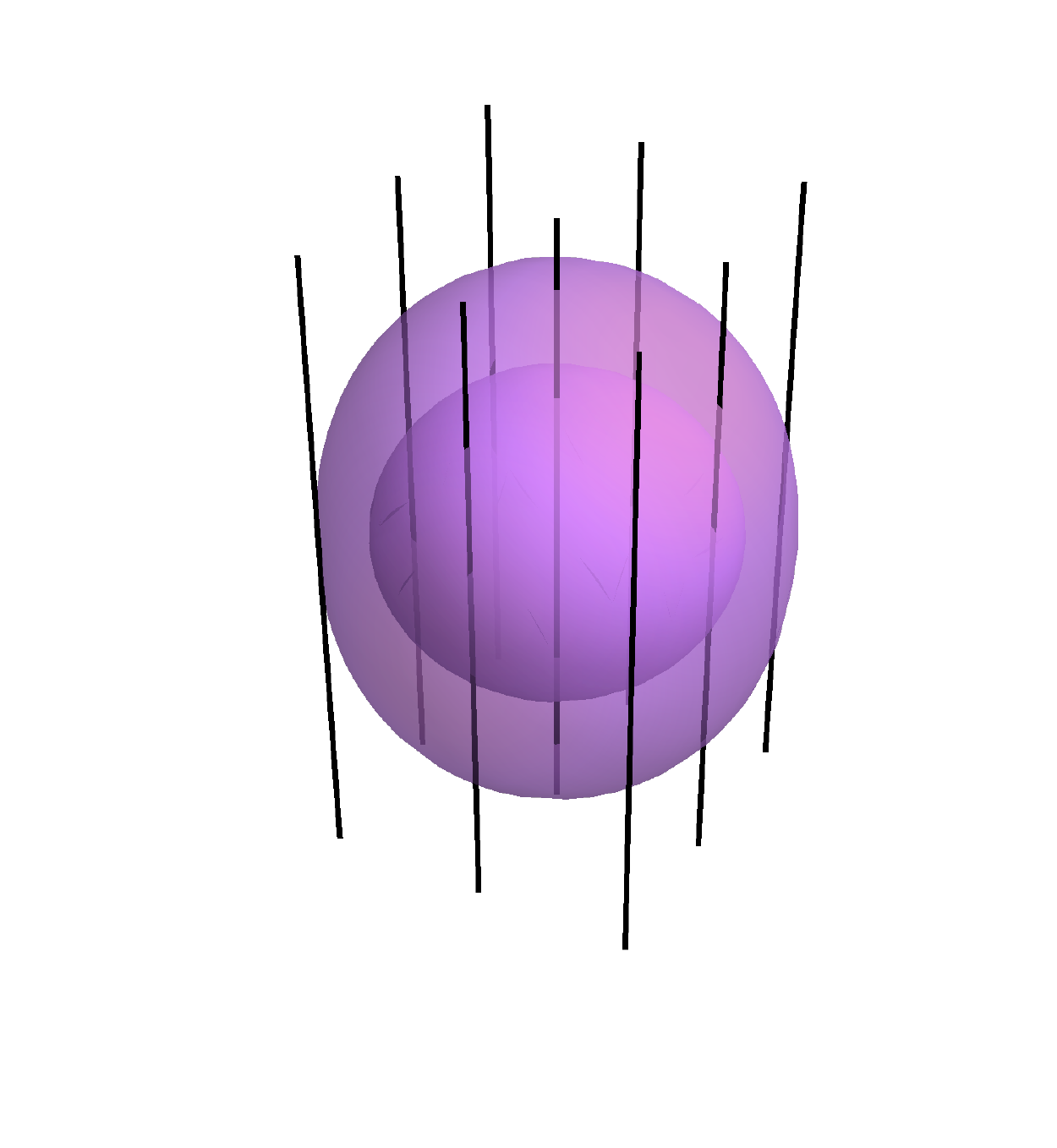}}
\;
\subfloat[][\label{fig:surf_xzgrid}]{
  \includegraphics[trim={0cm 0cm 0cm 0cm},clip,scale=0.25]
                  {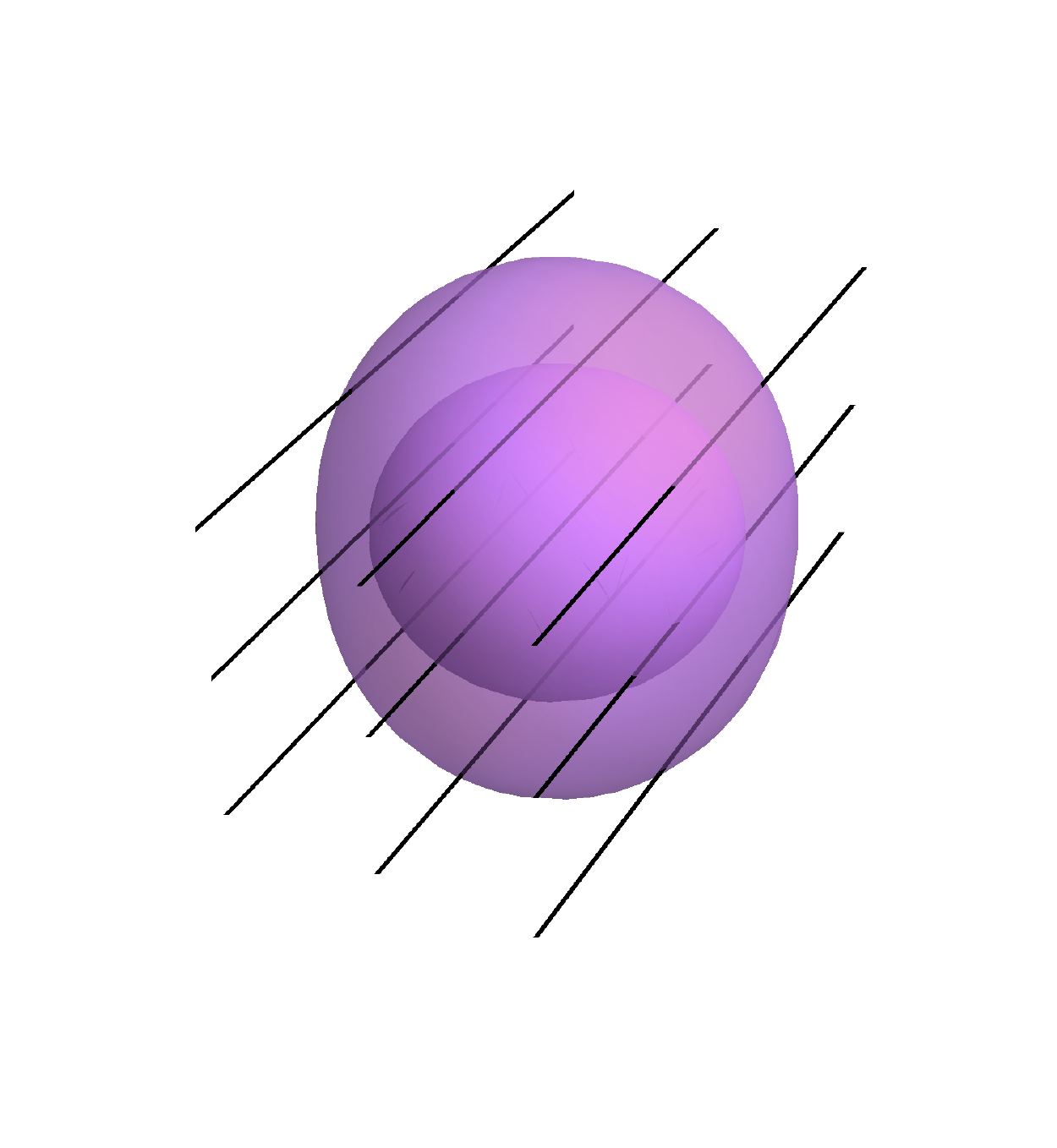}}
\;
\subfloat[][\label{fig:surf_yzgrid}]{
  \includegraphics[trim={0cm 0cm 0cm 0cm},clip,scale=0.25]
                  {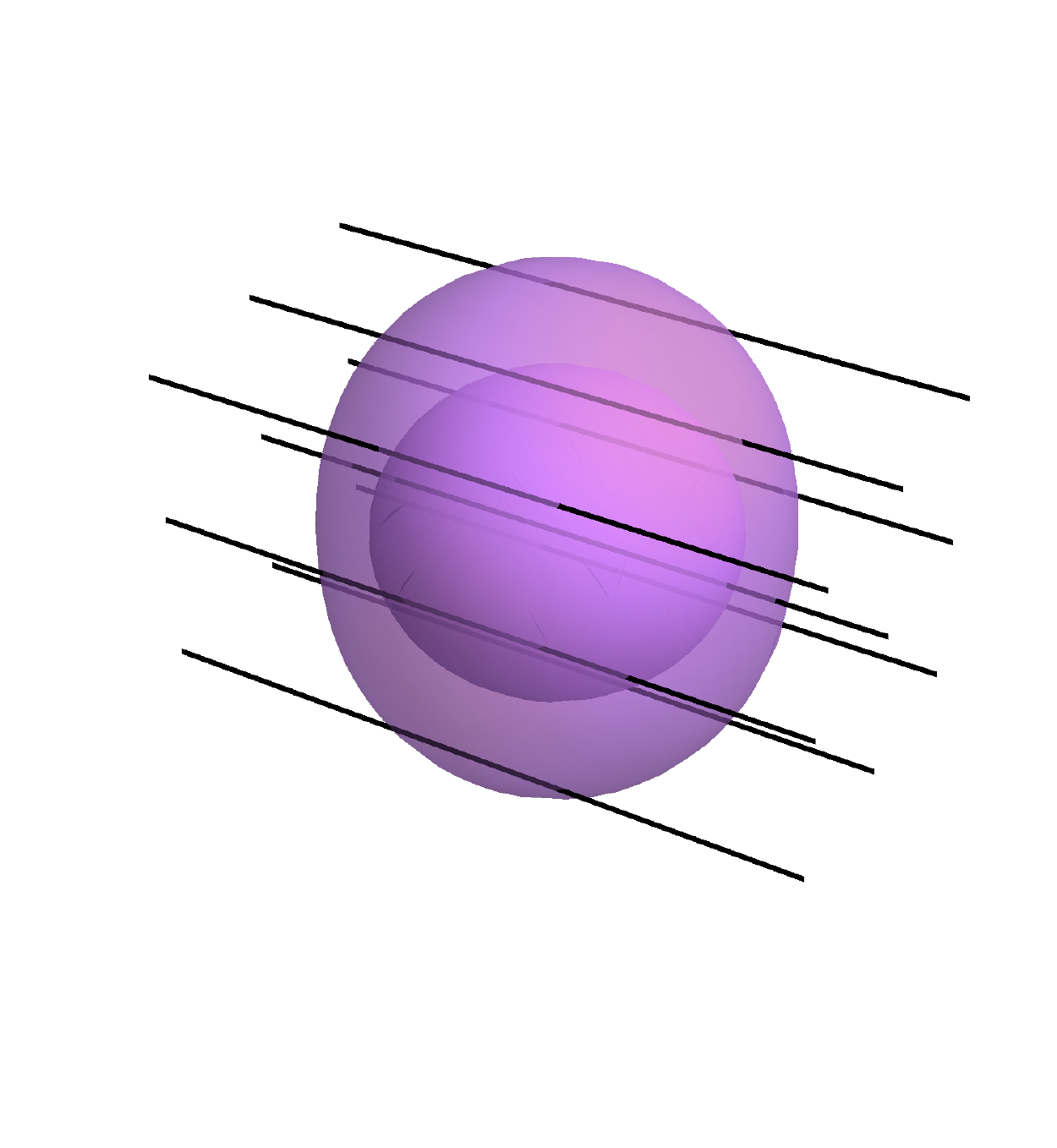}}
\caption{Basic sample of a surface $X \subset \RR^3$ as intersection
  of $X$ with the 1-dimensional skeleton of a cubic tessellation.
\label{fig:surf_grid}}
\end{figure}
\end{ex}

\subsection{The extra sample} \label{sec:extra}
In the case $\dim{X} > 1$ we need to complement the basic sample with
an extra sample $E'_{\delta}$ to guarantee density. We will do this
without knowing the reach of $X$.
%For $1 \leq i \leq n$ and $g \in G_i(\delta)$, let
%$X_{i,g} =X \cap \{x_i-g\}$.
For each $1 \leq k < d$ we have an extra sample given by
\begin{equation} \label{eq:extra}
  E'_{\delta,k} = \bigcup_{\substack{t \in T_k \\ g \in G_t(\delta)}} I(X\cap \pi_t^{-1}(g), q),
\end{equation}
where $I(X\cap \pi_t^{-1}(g), q)$ is the normal locus discussed in
\secref{sec:edd} and $q \in \RR^n$ is generic. The total extra sample
$E'_{\delta}$ is the union of $E'_{\delta,k}$ for $1 \leq k \leq d-1$
and the complete sample of $X$ is $$S_{\delta}=E_{\delta} \cup
E'_{\delta}$$

Again, the grid $G_t(\delta)$ is perturbed by a random translation
which guarantees with probability 1 that the intersections $X\cap
\pi_t^{-1}(g)$ are smooth for all $g \in G_t(\delta)$. Recall that for
every connected component $Y$ of $X\cap \pi_t^{-1}(g)$, $I(X\cap
\pi_t^{-1}(g), q)$ contains a point on $Y$. This means that the extra
sample $E'_{\delta}$ has points of every connected component of $X\cap
\pi_t^{-1}(g)$ for $g \in G_t(\delta)$, $t \in T_k$ and $1 \leq k \leq
d-1$. This will be used below to show density properties of the total
sample $S_{\delta}$.

\begin{rem}
An alternative to introducing the extra sample would be to use only
the basic sample but compute the reach of $X$ first and let $\delta$
be small with respect to the reach. The details of showing density
requirements for such a simple sampling procedure are yet to be
carried out. Moreover the complexity involved in the computation of the
reach and in particular the computation of the maximal curvature of $X,$ in
the sense explained in \secref{sec:edd},  is to our knowledge still to be explored.
\end{rem}

In the case of curves there is no extra sample. For surfaces, the
extra sample is constructed by intersecting $X$ with hyperplanes
parallel to the coordinate axes and computing normal loci of the
resulting curves. In higher dimension we must intersect $X$ with
hyperplanes, (codimension 2) planes and so on all the way down to planes
of dimension $n-\dim{X}+1$.

% have right dimension of intersection with prob. 1?
% comment here

% Make use (computationally) of the fact that slices
% lie in hyperplanes (or lin spaces of smaller dim)
% and eliminate variable when computing EDD?
% Does it matter?

\begin{ex} \label{ex:surf-extra}
  Continuing \exref{ex:surf} we illustrate the extra sample in the
  case of a surface $X \subset \RR^3$. The extra sample $E'_{\delta}$
  is acquired by intersecting $X$ with a number of planes $P \subset
  \RR^3$ and computing the normal locus $I(X\cap P, q)$ of the
  resulting curve. As mentioned above, the normal locus has a point on
  every connected component of $X \cap P$. The planes come in three
  families of planes parallel to the $xy$-plane, $xz$-plane and
  $yz$-plane respectively, see \figref{fig:surf_planes}. The distance
  between adjacent planes in the same family is the density parameter
  $\delta>0$. To tie this to (\ref{eq:extra}), let $\pi_1$ be the
  projection to the $x$-axis and consider for instance the planes
  parallel to the $yz$-axis. These are expressed in (\ref{eq:extra})
  as $\pi_1^{-1}(g)$ for $g \in G_1(\delta)$ in a grid of size
  $\delta$ on the $x$-axis.
% have fiddled with order of figures...
\begin{figure}[ht]
\centering
\subfloat[][\label{fig:surf_yplanesall}]{
  \includegraphics[trim={0cm 0cm 0cm 0cm},clip,scale=0.2]
                  {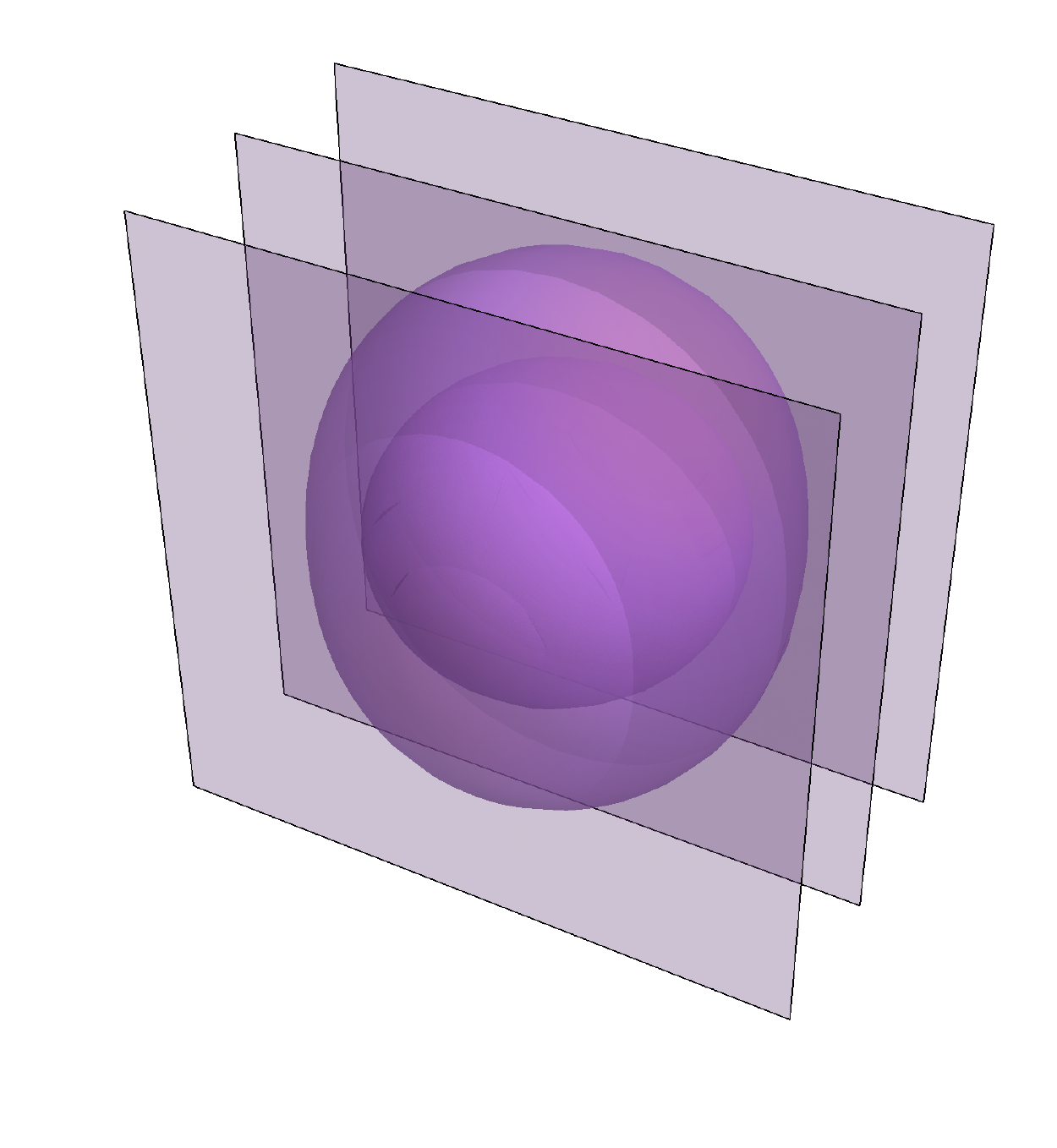}}
\;
\subfloat[][\label{fig:surf_yplanes1}]{
  \includegraphics[trim={0cm 0cm 0cm 0cm},clip,scale=0.2]
                  {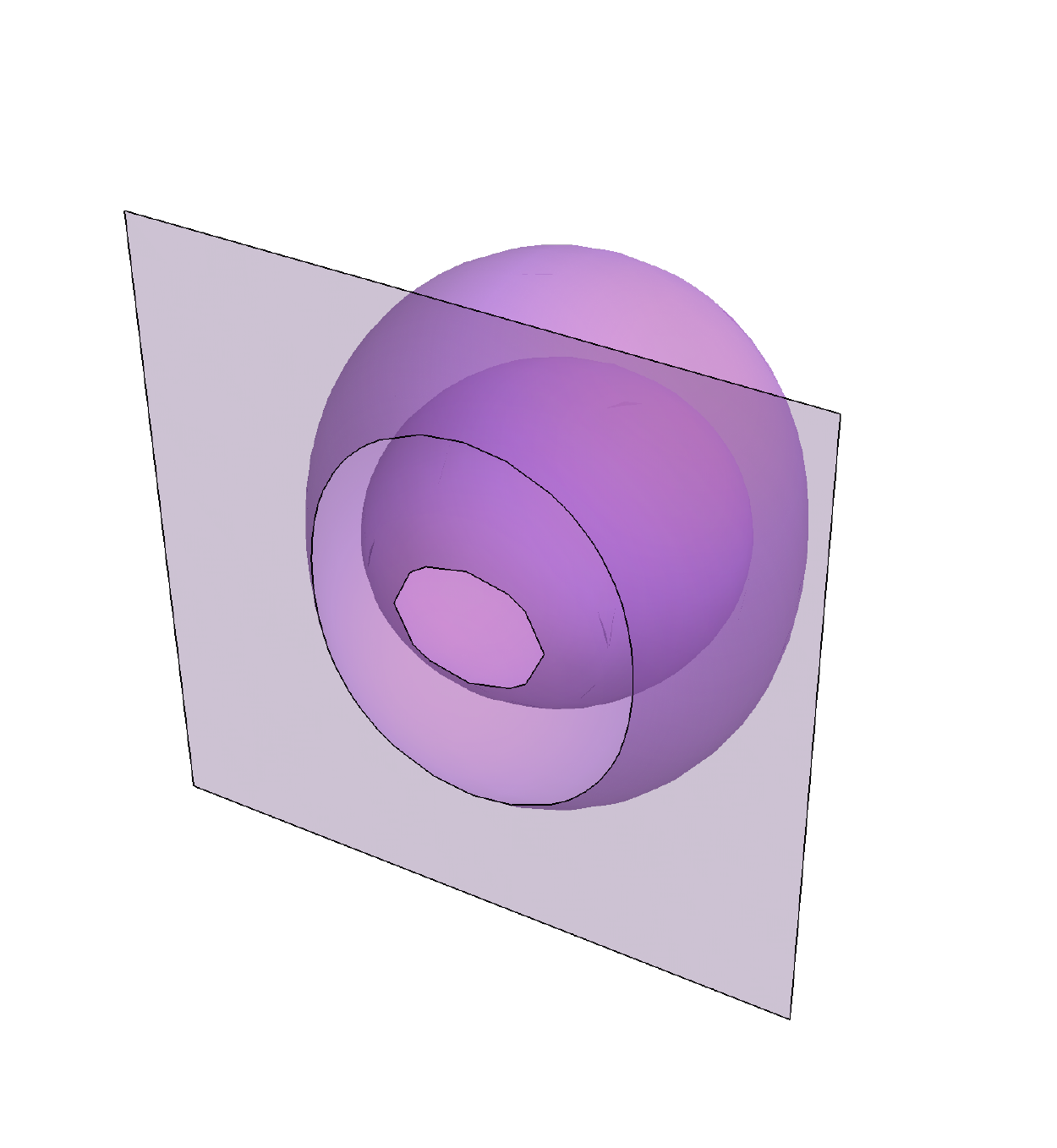}}
\;
\subfloat[][\label{fig:surf_yplanes2}]{
  \includegraphics[trim={0cm 0cm 0cm 0cm},clip,scale=0.2]
                  {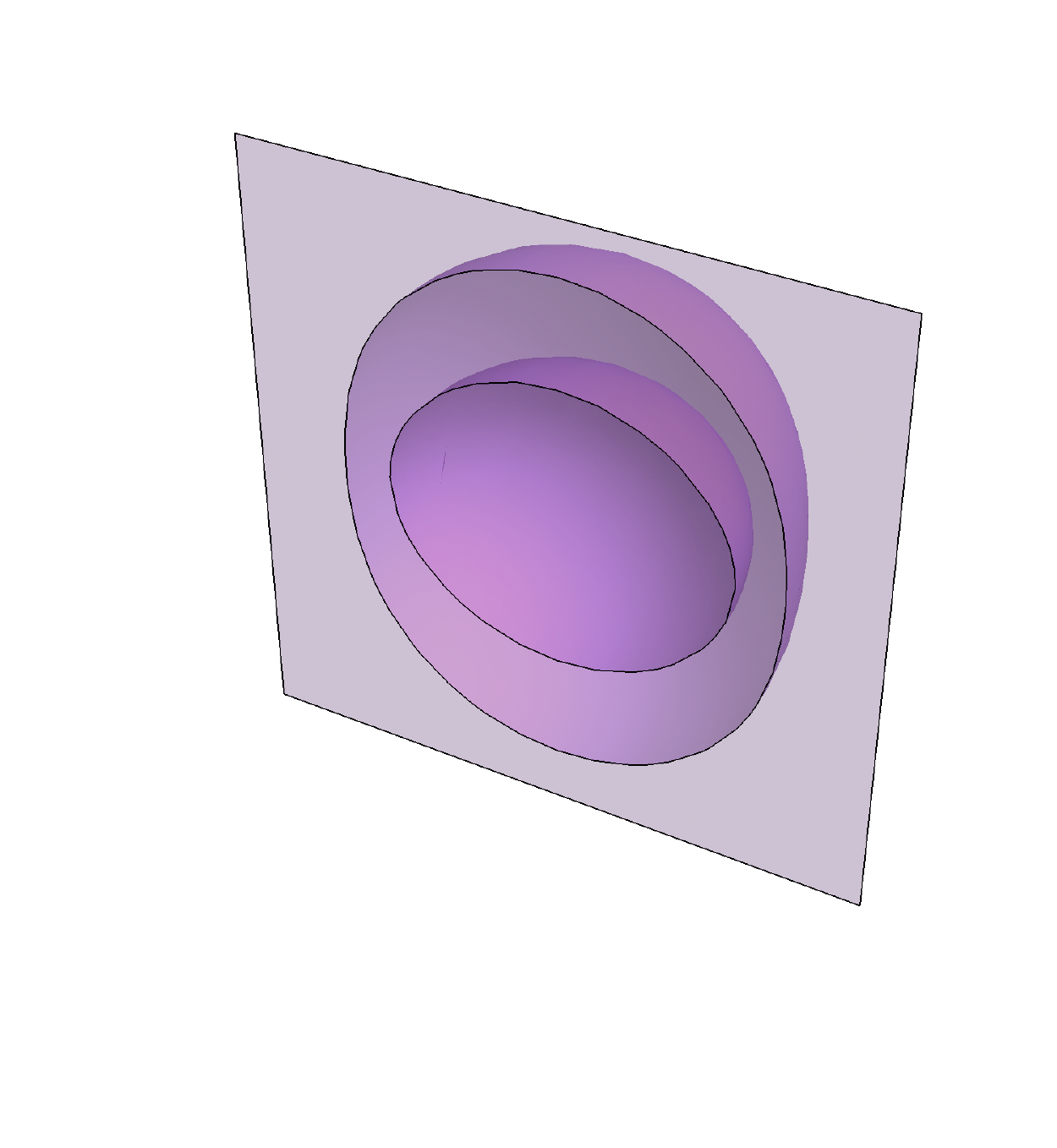}}
\;
\subfloat[][\label{fig:surf_yplanes3}]{
  \includegraphics[trim={0cm 0cm 0cm 0cm},clip,scale=0.2]
                  {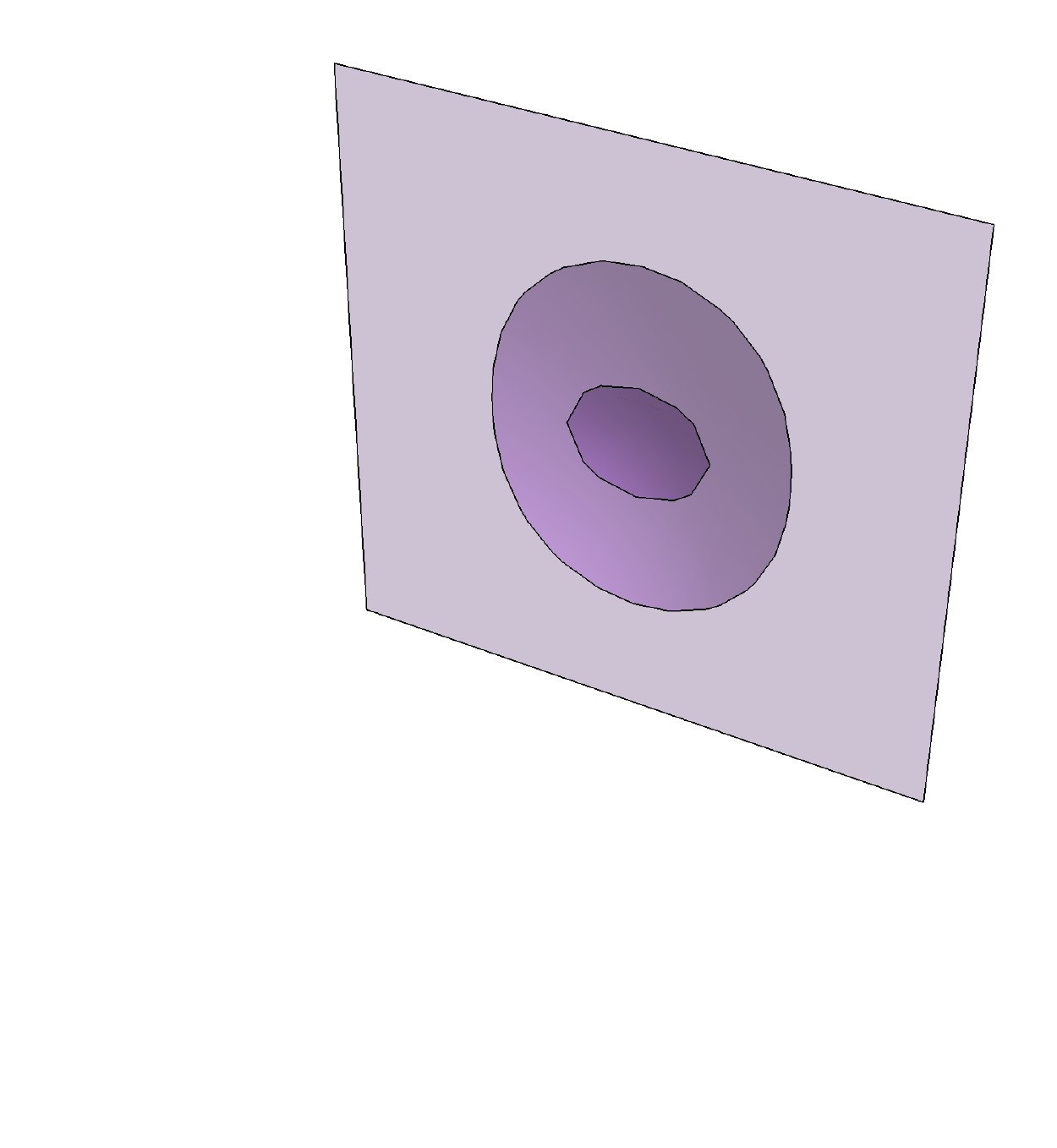}}
\;

\subfloat[][\label{fig:surf_xplanesall}]{
  \includegraphics[trim={0cm 0cm 0cm 0cm},clip,scale=0.2]
                  {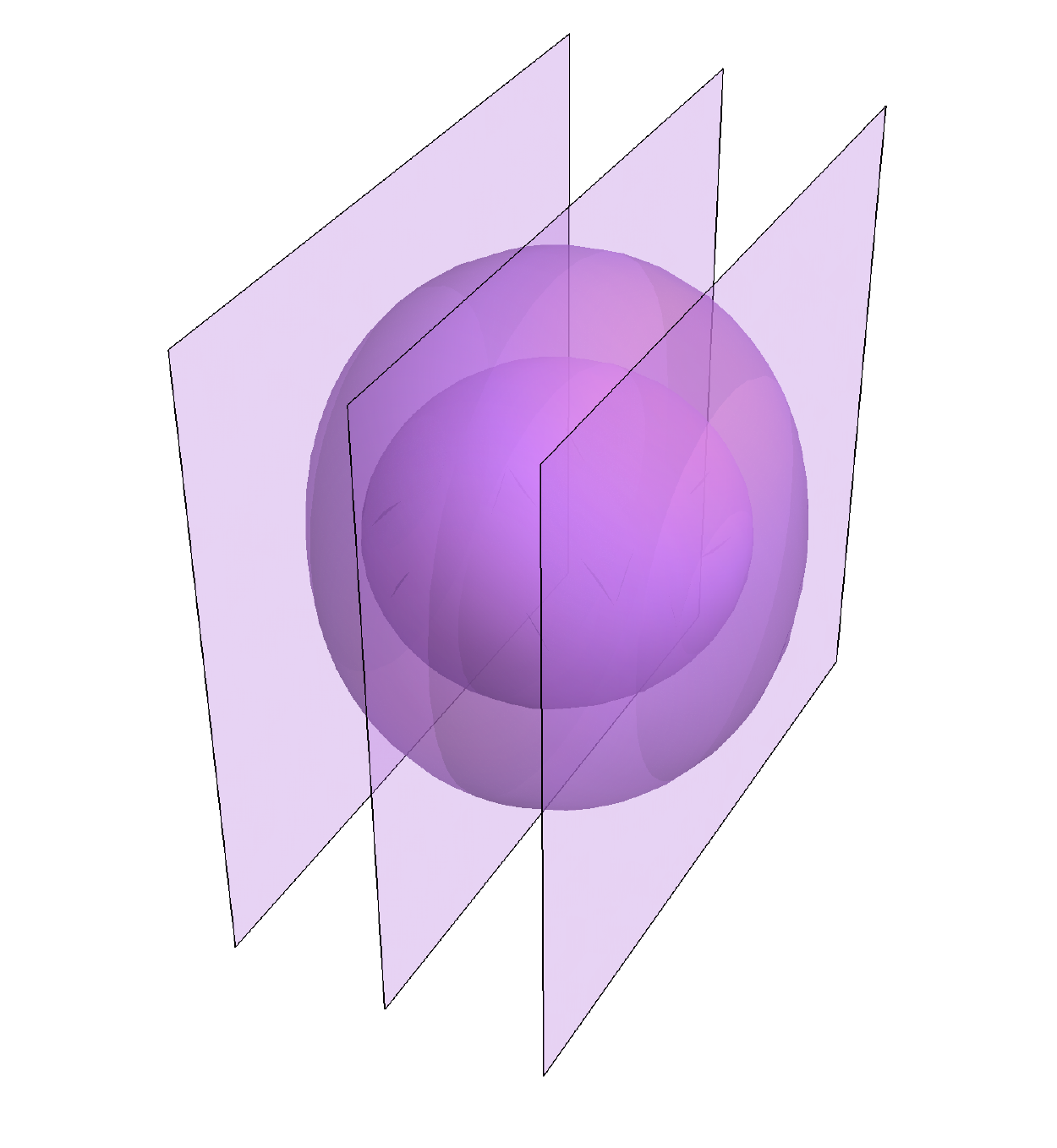}}
\;
\subfloat[][\label{fig:surf_xplanes1}]{
  \includegraphics[trim={0cm 0cm 0cm 0cm},clip,scale=0.2]
                  {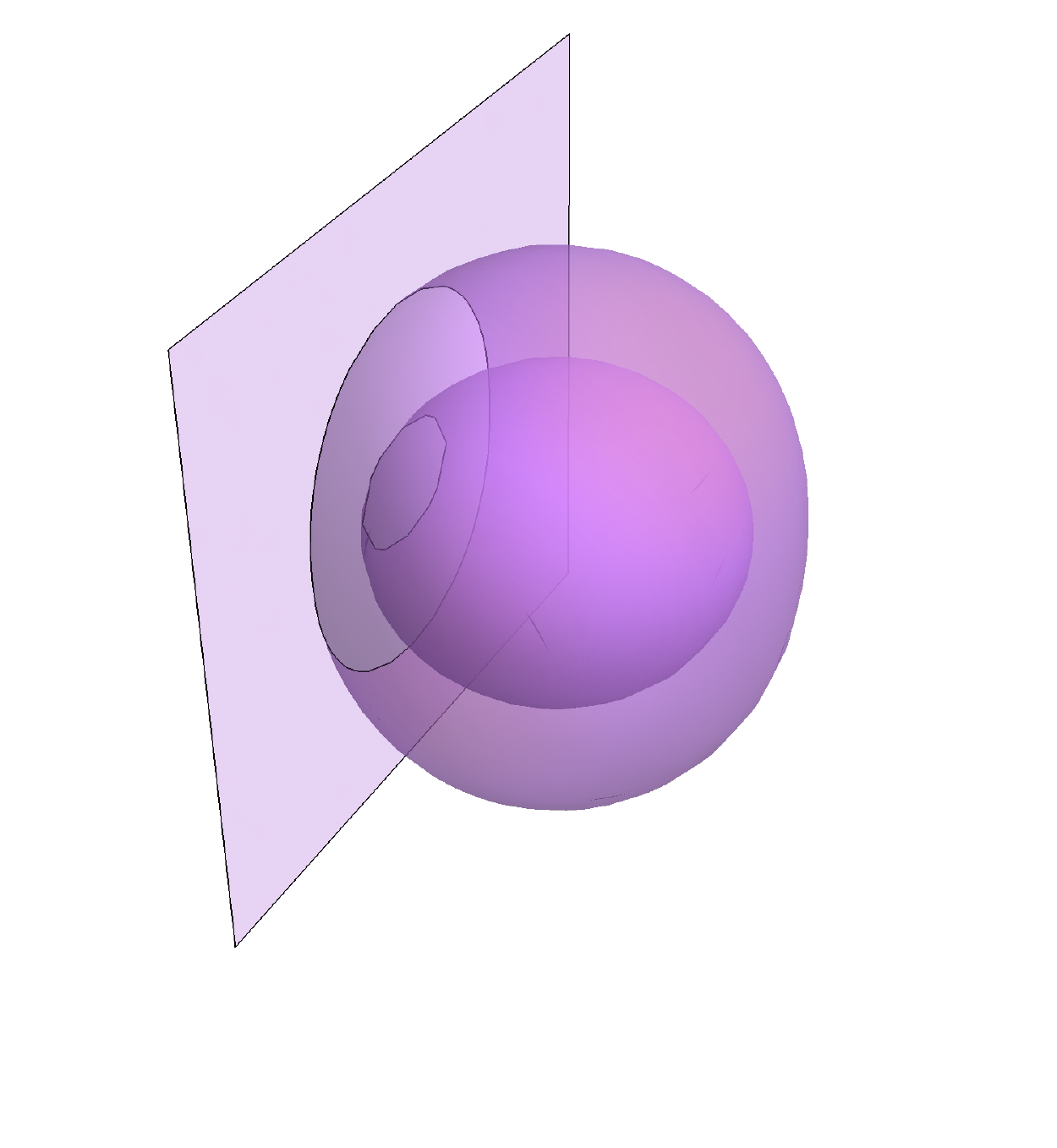}}
\;
\subfloat[][\label{fig:surf_xplanes2}]{
  \includegraphics[trim={0cm 0cm 0cm 0cm},clip,scale=0.2]
                  {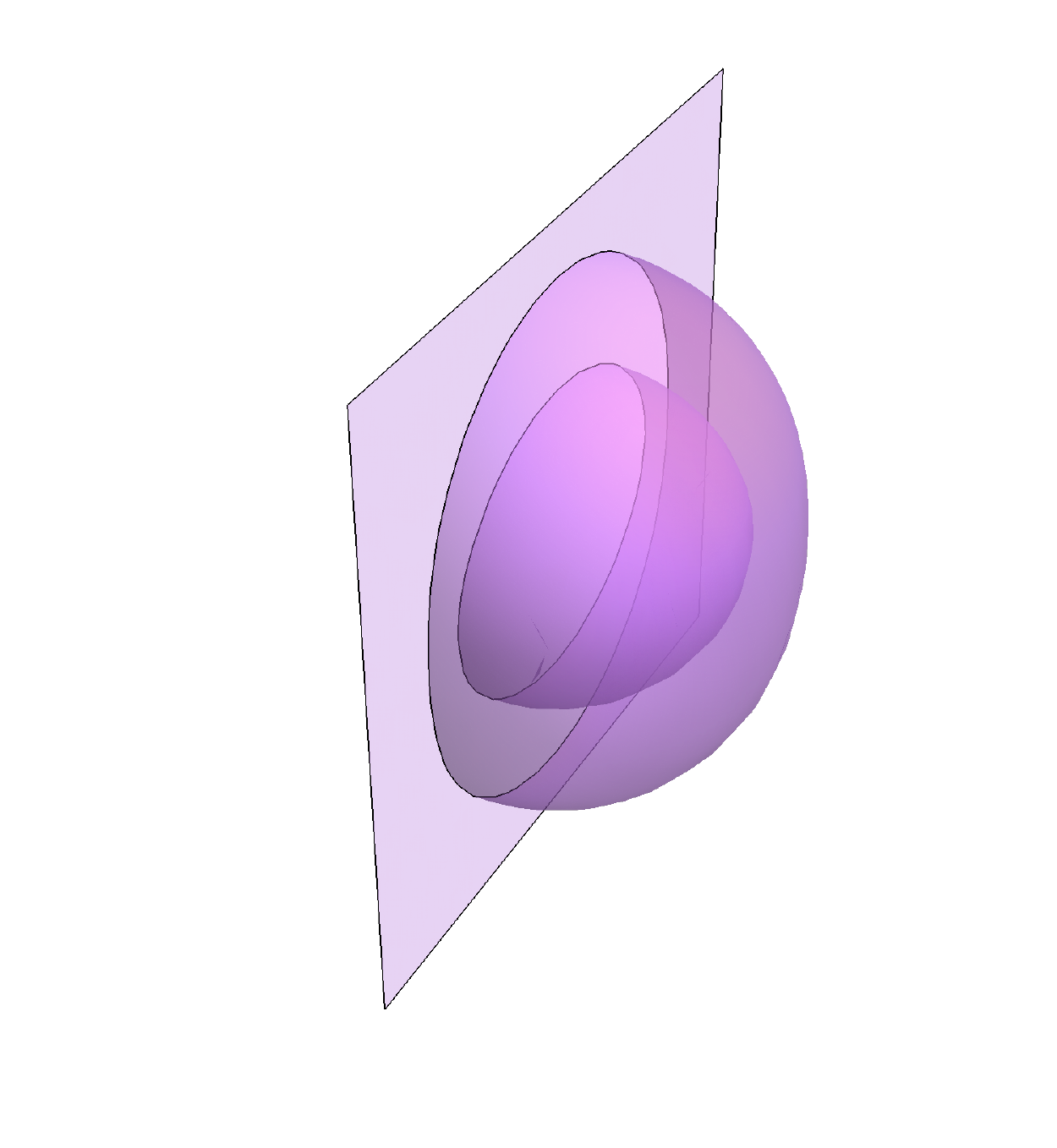}}
\;
\subfloat[][\label{fig:surf_xplanes3}]{
  \includegraphics[trim={0cm 0cm 0cm 0cm},clip,scale=0.2]
                  {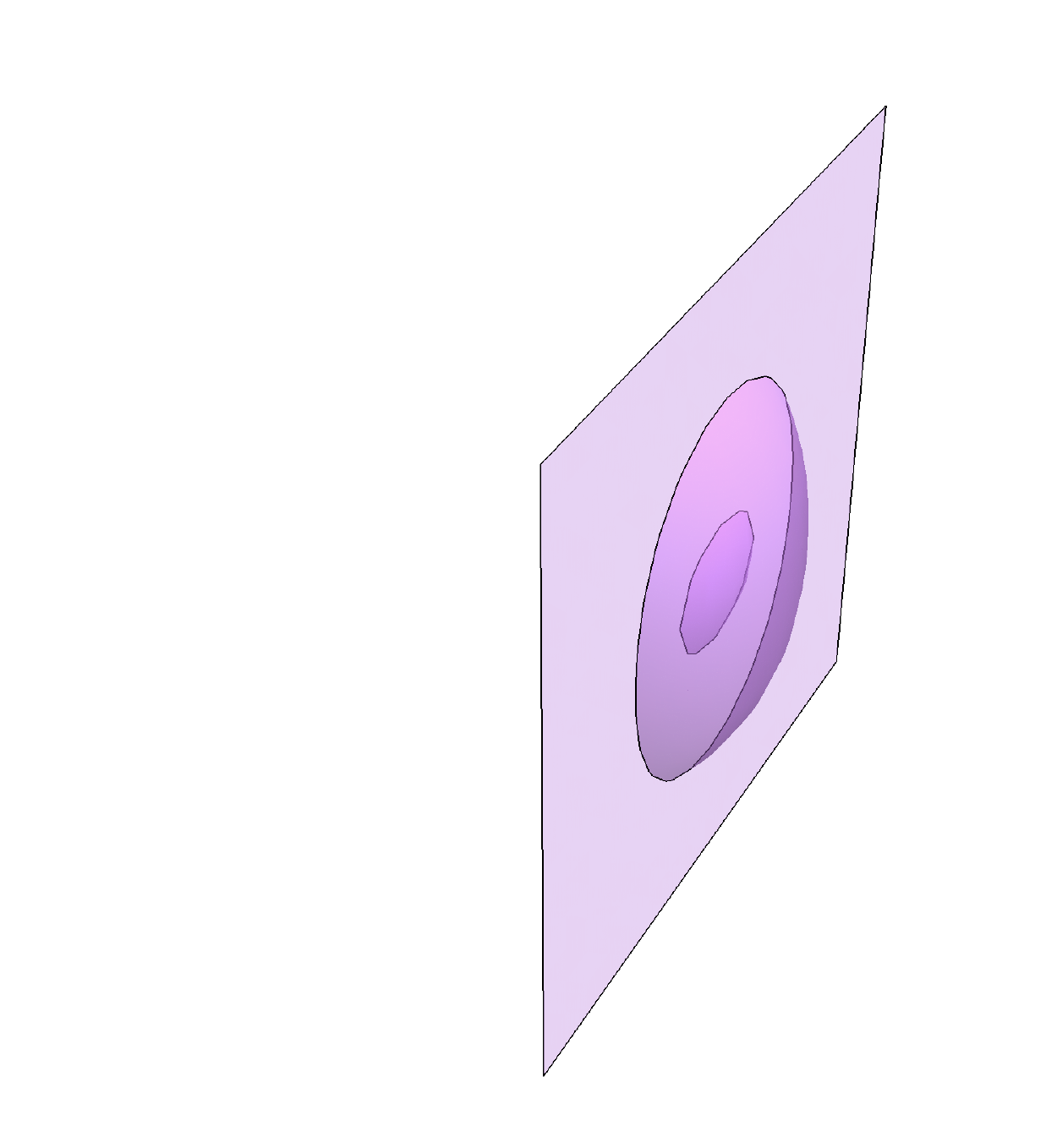}}
\;

\subfloat[][\label{fig:surf_zplanesall}]{
  \includegraphics[trim={0cm 0cm 0cm 0cm},clip,scale=0.2]
                  {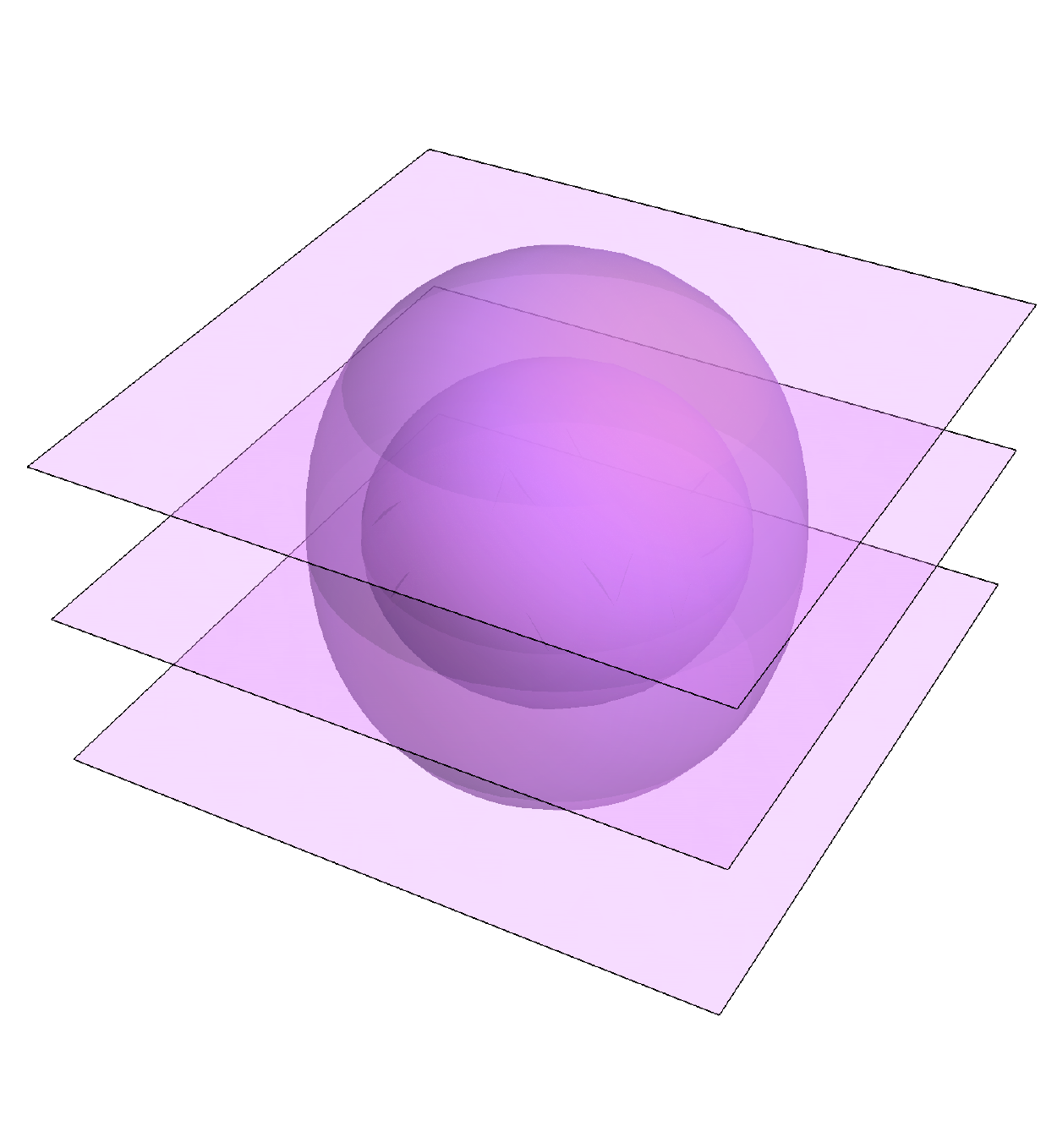}}
\;
\subfloat[][\label{fig:surf_zplanes1}]{
  \includegraphics[trim={0cm 0cm 0cm 0cm},clip,scale=0.2]
                  {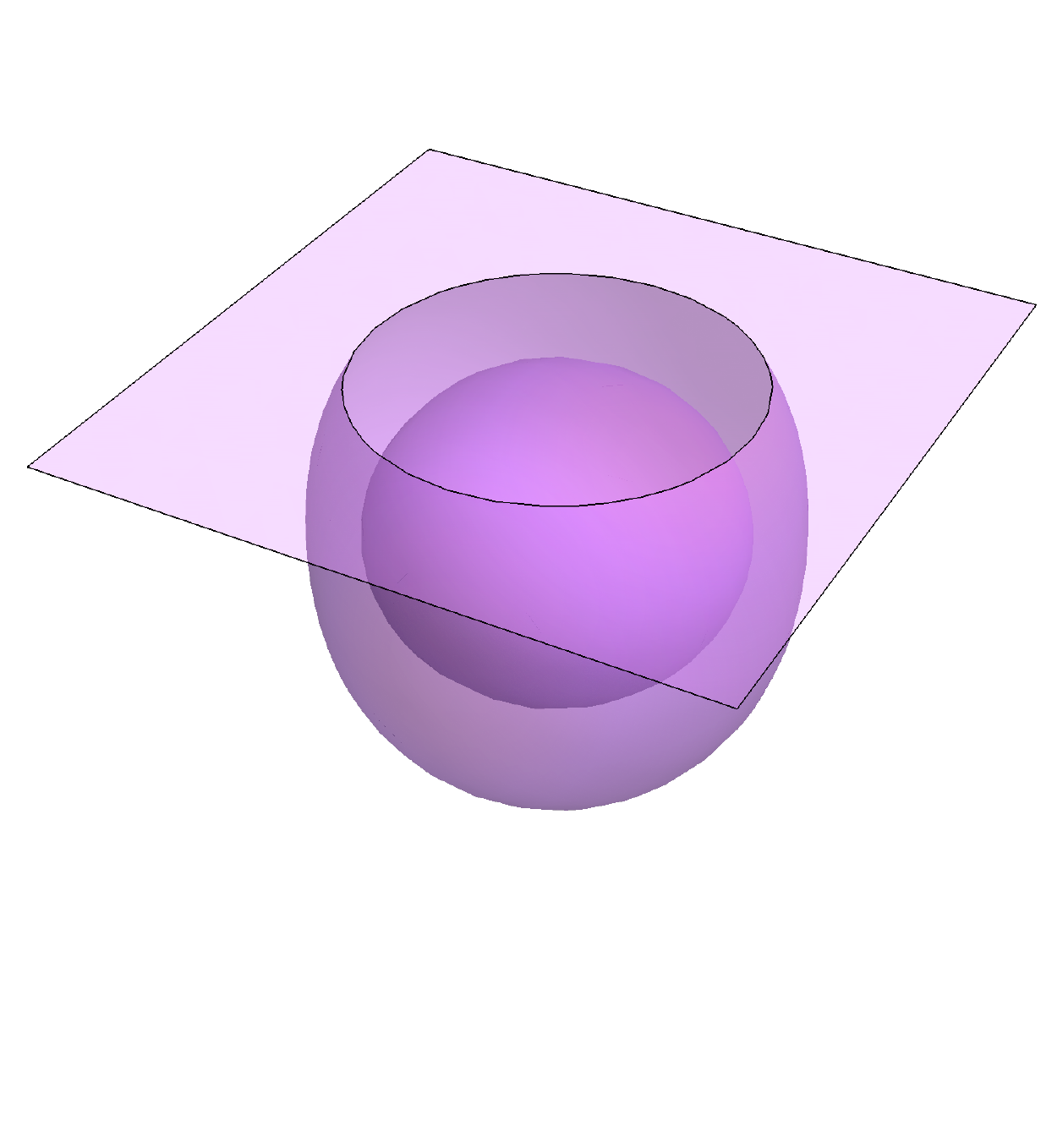}}
\;
\subfloat[][\label{fig:surf_zplanes2}]{
  \includegraphics[trim={0cm 0cm 0cm 0cm},clip,scale=0.2]
                  {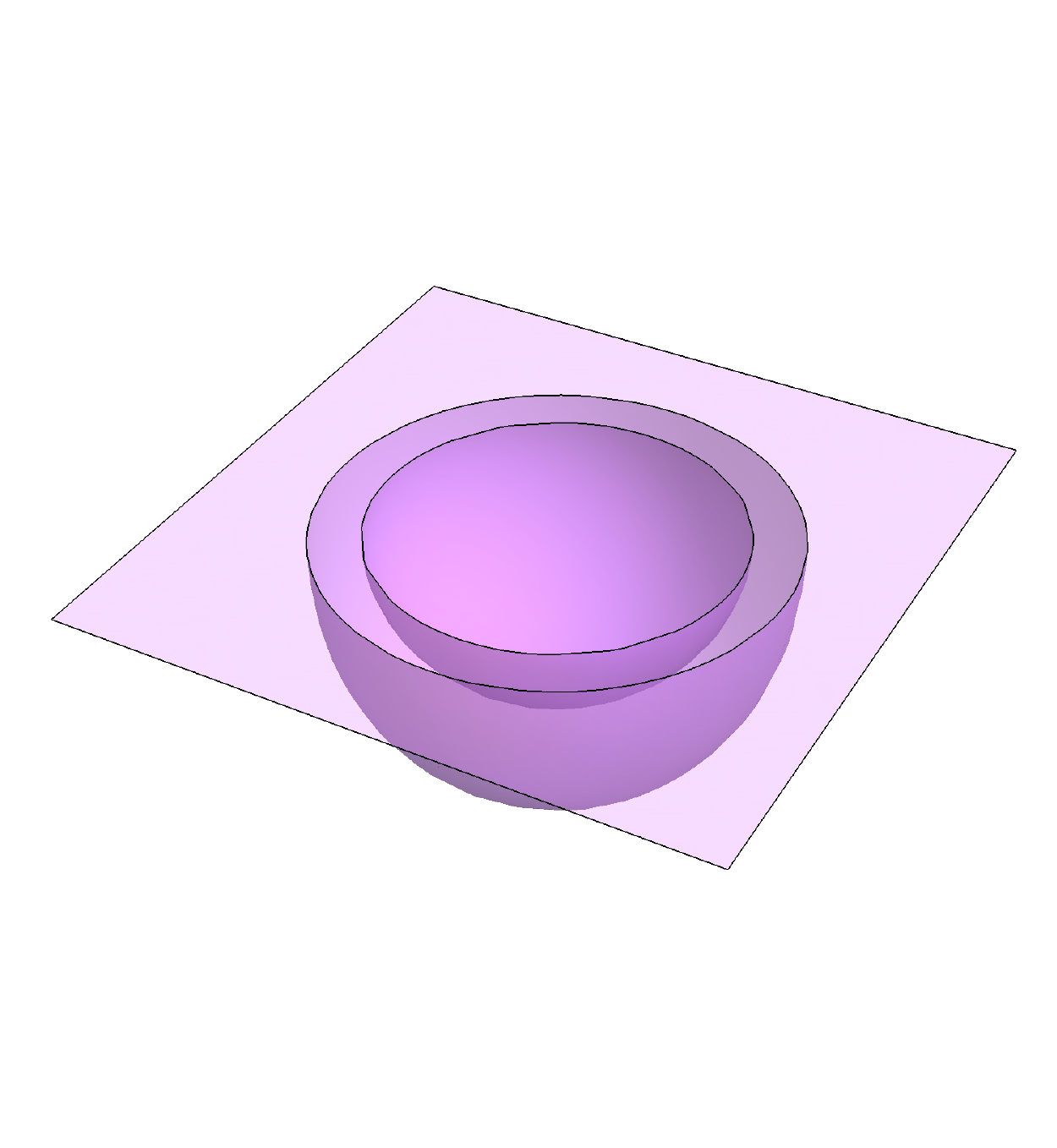}}
\;
\subfloat[][\label{fig:surf_zplanes3}]{
  \includegraphics[trim={0cm 0cm 0cm 0cm},clip,scale=0.2]
                  {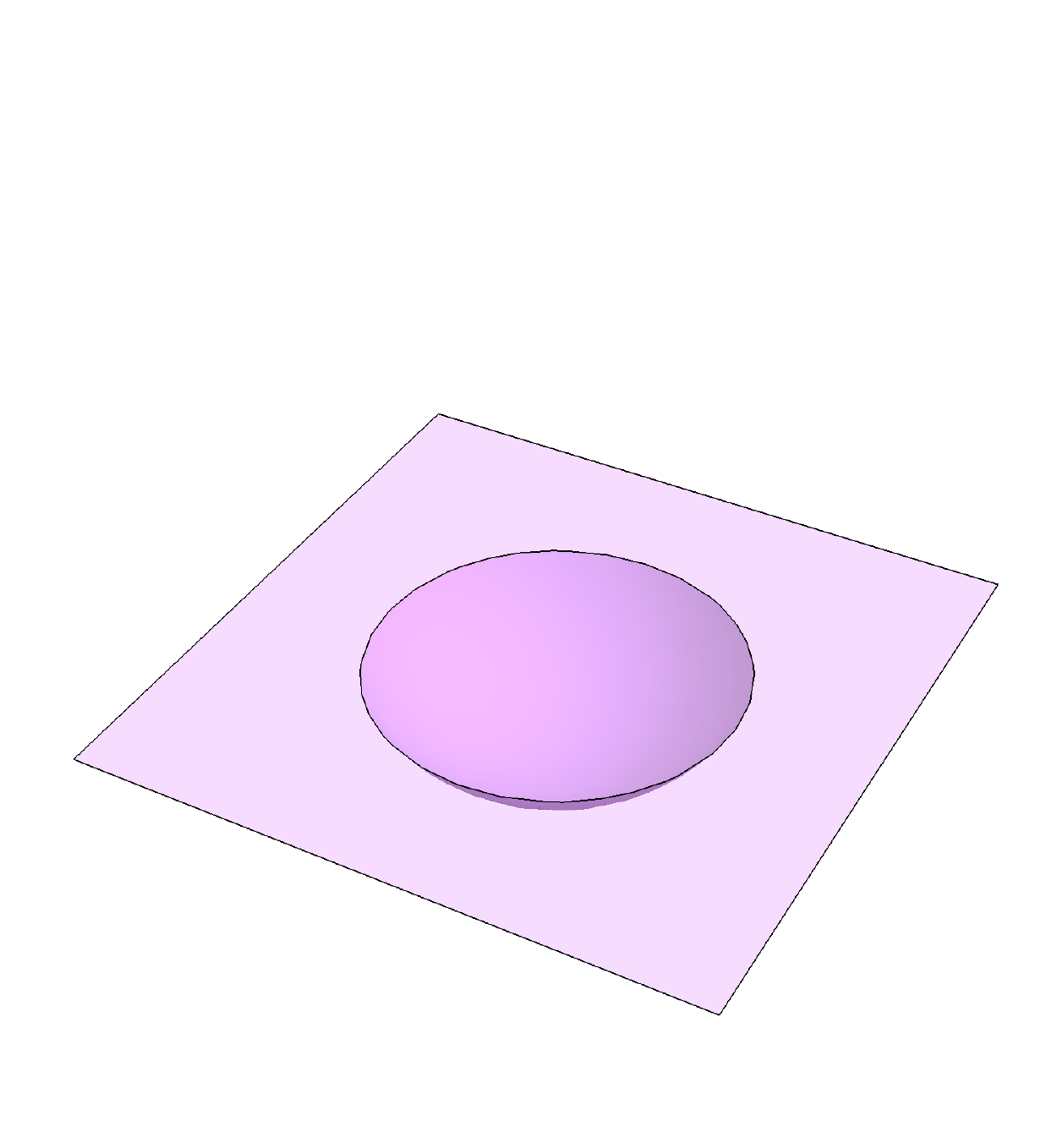}}
\caption{Extra sample of a surface in $\RR^3$.
\label{fig:surf_planes}}
\end{figure}
\end{ex}

\begin{ex}
	In this example we illustrate \algref{alg:reach} from Section \ref{sec:local_reach} by applying it to a simple quadric $X
	\subset S^3 \subset \RR^4$ defined by $xy+y^2-2zw=0$ and
	$x^2+y^2+z^2+w^2=1$. We first
	compute a coarse sample of $X$ (see \figref{fig:quadric}) and evaluate
	the minimum of $\eta(x)$ over the sample. The result is a bit more
	than $0.024$.
	\begin{figure}[ht]
		\centering
		\includegraphics[trim={2cm 3.5cm 2cm 3cm},clip,scale=0.5]{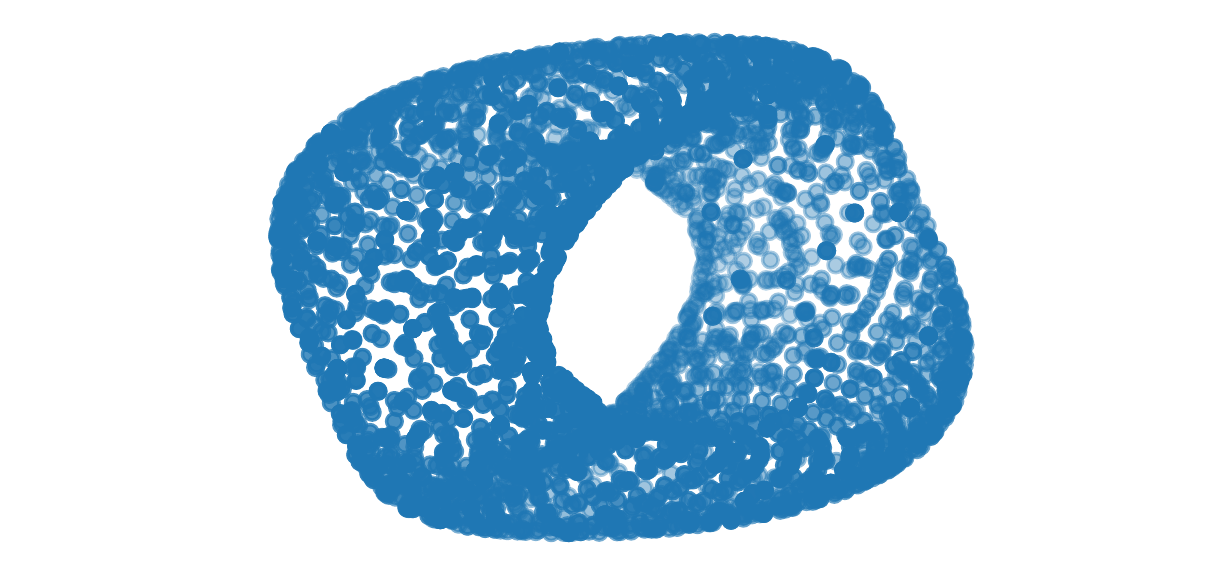}
		\caption{Quadric in $S^3$ projected to $\RR^3$.
			\label{fig:quadric}}
	\end{figure}
	To try to show $0.01 < \tau_X$ we run \algref{alg:reach} with
	$\epsilon_0=0.014$ whereby it terminates after producing one sample
	$E \subset X$. Recall that $E=E_{\delta} \cup E'_{\delta}$, where
	the grid size $\delta$ was chosen to be
	$\delta=\epsilon_0/\sqrt{4}=0.007$.
	\begin{itemize}
		\item Basic sample size: $|E_{\delta}|=780652$
		\item Extra sample size: $|E'_{\delta}|=3592$
		\item Estimated local reach: $0.024 < \min_{e \in E} \eta(e) < 0.0246$.
	\end{itemize}
	The result is that $\tau_X > 0.024-0.014=0.01$.
\end{ex}

\subsection{Density guarantees}

In this section we show properties guaranteeing density of the sample
$S_{\delta}$ described above.

Consider first a surface in $\RR^3$ as in \exref{ex:surf}. The extra
sample $E'_{\delta}$ is introduced to guarantee density as we will now
illustrate. Consider the cubical tessellation underlying the grid and
the 1-dimensional skeleton shown in \figref{fig:surf_fullgrid}. If $X$
intersects a cube $\pi$ in the tessellation, we want to have that there is a
sample point close by, for example on the boundary $\partial \pi$ of
$\pi$. This happens if the basic sample $E_{\delta}$ contains a
point on $\partial \pi$. On the other hand, that $E_{\delta}$ has no
point on $\partial \pi$ means that the surface $X$ does not intersect
any edge of $\pi$. Suppose that we in addition know that no component of
$X$ is contained in $\pi$. Then $X$ must intersect the boundary of $\pi$
and the intersection curve has a component $Y$ that is completely
contained in one of the faces of $\pi$. See \figref{fig:squeezedsurf}
for an illustration. Moreover, $Y$ is a connected component of one of
the plane curves $X\cap P$ where $P=\pi_t^{-1}(g)$ is one of the
planes occurring in the construction of the extra sample
(\ref{eq:extra}). As explained above, the extra sample picks up a
point on $Y$ which gives a sample point on $\partial \pi$.

We make sure that $X$ has no connected components completely contained
in $c$ by choosing the grid size $\delta$ small with respect to the
narrowest bottleneck of $X$.

\begin{figure}[ht]
  \centering
\includegraphics[trim={1.2cm 1.2cm 1.2cm 1.2cm},clip,scale=0.5]{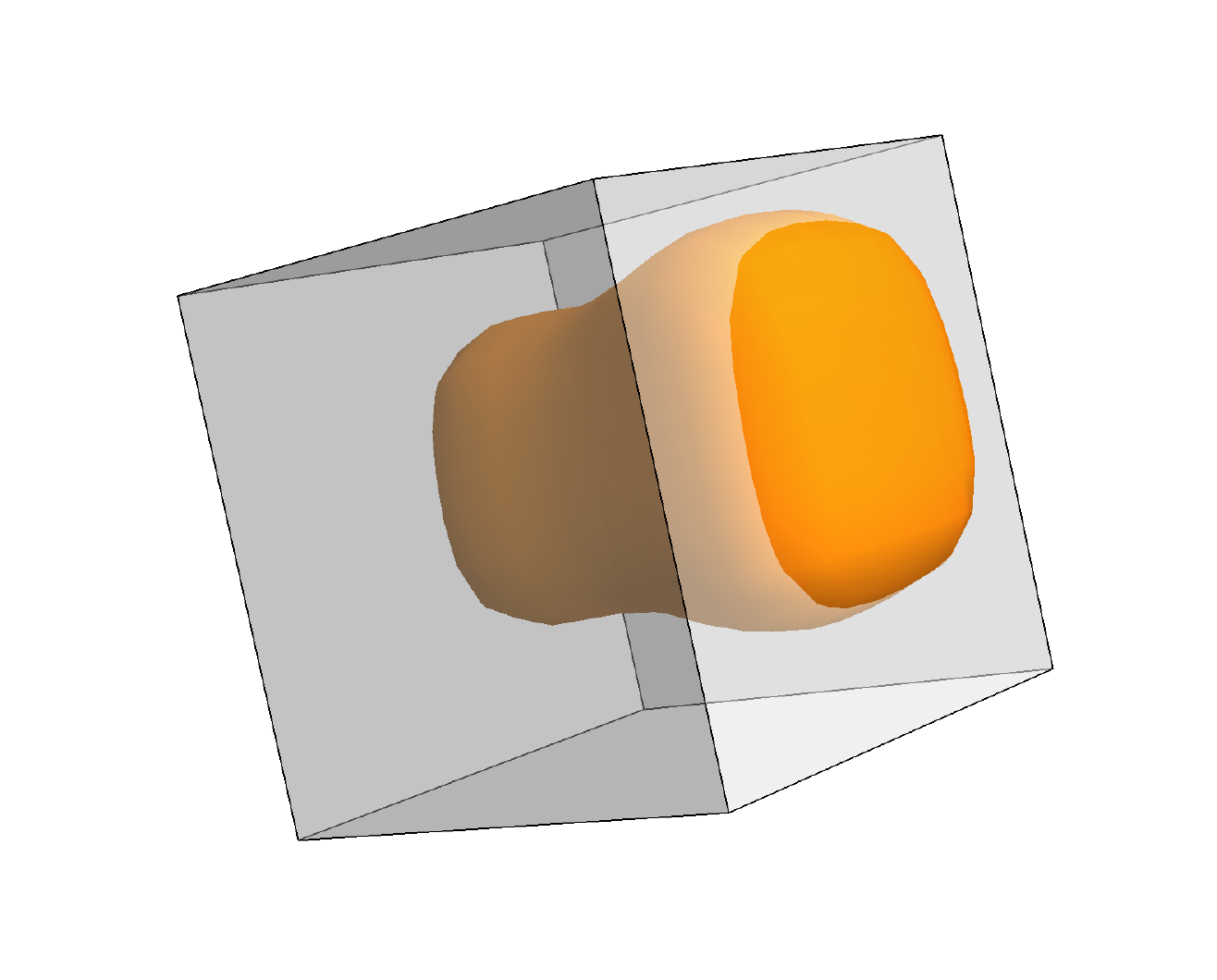}
\caption{Surface $X \subset \RR^3$ disjoint from the edges of a
  cube $\pi$.\label{fig:squeezedsurf}}
\end{figure}

\begin{thm} \label{thm:grid}
  Let $b_2$ be the radius of the narrowest bottleneck of $X$ and let
  $\epsilon > 0$. If $0<\delta \sqrt{n} < \min\{\epsilon, 2b_2\}$, then
  $S_{\delta}$ is an $\epsilon$-sample of $X$.
\end{thm}
\begin{proof}
Let $p \in X$. We must show that there is a point of
$S_{\delta}=E_{\delta} \cup E'_{\delta}$ at a distance at most
$\epsilon$ from $p$. There is a complex $C$ of closed hypercubes (a
tessellation) filling $\RR^n$ such that \[E_{\delta} = \bigcup_{F \in
  C_{n-d}} X \cap F\] where $d=\dim{X}$ and $C_{n-d}$ denotes the set
of $(n-d)$-dimensional faces of the hypercubes in $C$. Let $\pi$ be a
hypercube in $C$ such that $p \in \pi$. Since the diagonal of $\pi$ has
length $\sqrt{n} \delta < \epsilon$, it is enough to find a point in
$S_{\delta} \cap \pi$. The diagonal of $\pi$ is shorter than the diameter of the narrowest
bottleneck of $X$, which means that no connected component of $X$ can
be contained in $\pi$. Hence $X$ intersects the boundary of $\pi$. Any
face of $\pi$ of dimension smaller than $n-d$ is contained in some
$(n-d)$-dimensional face of $\pi$. Hence, if $x \in X \cap F \neq
\emptyset$ for some face $F$ of $\pi$ of dimension $n-d$ or smaller,
then $x \in E_{\delta}$ and we are done. This settles the case $d=1$
since no component of $X$ is contained in $\pi$. Now let $F$ be a face
of $\pi$ such that $X \cap F \neq \emptyset$ and which has minimal
dimension among such faces. By above we may assume $\dim{F} >
n-d$. Since $X$ does not intersect the boundary of $F$, $X \cap F$ has
a connected component $X_F$ that is completely contained in $F$. But
then $E'_{\delta}$ contains a point $x \in X_F$.
\end{proof}

\subsection{Complexity analysis}
As preparatory steps for sampling we need to find a bounding box of $X
\subset \RR^n$ and we need to compute the narrowest bottleneck of
$X$. Given a bounding box $B$ of $X$ and a grid size $\delta$ it is
straightforward to compute the basic sample $E_{\delta}$ in
(\ref{eq:sample}) using standard homotopy methods \cite{bertini, SW05,
  BHSW13}. We just have to intersect $X$ with the linear spaces
$\pi_t^{-1}(g)$ for $g \in G_t(\delta) \cap \pi_t(B)$. Likewise the
extra sample $E'_{\delta}$ may be computed as explained in
\secref{sec:edd} using for instance multihomogeneous homotopies.

%We thus have the following four computational problems:
%\begin{itemize}
%	\item the bounding box $B$
%	\item the bottlenecks of $X$, which will determine the grid size $\delta$
%	\item the basic sample $E_\delta$
%	\item the extra sample $E_\delta'$
%\end{itemize}
%Let $m<n$ denote the dimension of $X$. By assumption $X$ is compact and a generic complete intersection. Suppose $X$ is defined by the polynomials $f_1, f_2, \dots, f_m$ of degree $d_1\leq d_2 \leq \dots \leq d_m$. 
%Let $S$ denote the complexity of computing the bottlenecks of $X$. 

%Let $f_1, f_2, \dots, f_n \colon \RR^n \to \RR$ be polynomials of degree $d_1\leq d_2 \leq \dots \leq d_n$. Let $P(n, d_1, \dots, d_n)$ denote the complexity of numerically approximating solutions of the $n\times n$ system of polynomial equations up to some precision $\delta$. 
%
%
%Computing the EDD of $X$ then amounts to solving the system of $m+n$ equations in $m+n$ variables. Let $P(m+n, d_1, \dots, d_m, d_m, \dots, d_m)$. 
%
%To compute the bottlenecks we have $P(m+2n, d_1, \dots, d_m, d_m, \dots, d_m)$. 
%
%
%The total computational complexity of computing the sampling is then:
%$$P(m+n, d_1, \dots, d_m, d_m, \dots, d_m) + P(m+2n, d_1, \dots, d_m, d_m, \dots, d_m) +$$
%$$+ {n \choose m}b\delta^{-m}P(n, 1, \dots, 1, d_1, \dots, d_m) + \sum_{k=1}^{m-1} {n \choose k} b\delta^{-k} P(2n-k, 1, \dots, 1, d_1, \dots, d_m, d_m, \dots, d_m)$$
%
%
%

The bounding box may be found by computing the normal locus
$I(X,q)$ with respect to a general point $q \in \RR^n$ as explained in
\secref{sec:edd}. Assuming that the bottleneck locus is finite,
we may use for example the method presented in \cite{E18} to compute
the narrowest bottleneck. It is worth noting that the latter method
contains as a first step to compute a normal locus $I(X,q)$ so that we
may acquire a bounding box from the bottleneck computation alone. In
fact, one may derive a bounding box from the widest bottleneck as well
if another method to compute bottlenecks is chosen.

To analyze the complexity of computing $E_\delta$ and $E'_\delta$ we count the number of paths that have to be tracked using homotopy continuation methods. To compute $E_\delta$ we note that $\# T_k = {n \choose k}$ and that the number of grid points in $G_t(\delta)$ for each $t\in T_k$ is $b \delta^{-k}$, where $b$ is the width of the bounding box $B$. By using a parameter homotopy \cite{SW05}, it suffices to solve the system $X\cap L$ where $L$ is a generic affine linear space of complementary dimension and then track deg$(X)$ number of paths to obtain the solutions for each grid point $g\in G_t(\delta)$. Let $\#P(X\cap L)$ denote the number of paths that need to be tracked in order to solve the system $X\cap L$. Consequently we would in total have to track   $\#P(X\cap L) + {n \choose d}b\delta^{-d}\text{deg}(X)$ number of paths in order to compute the basic sample, since we have ${n \choose d}$ number of grids $G_t(\delta)$ and $b \delta^{-k}$ number of grid points for each grid. 

%for the parameter homotopies and solve ${n \choose d}$ start systems for a generic grid point. 

Let $L_k \subset \RR^n$ denote a generic affine linear subspace of codimension $k$. For the extra sample we have analogously that we need to track $\#P(I(X\cap L_k, q))$ number of paths for a generic $q\in \Bbb R^n$ in order to define the parameter homotopy. The number of solutions for this system is $\text{EDD}(X\cap L_k)$ and thus we have to track $\sum_{k=1}^{d-1} {n \choose k} b\delta^{-k}\text{EDD}(X\cap L_k)$ number of paths for the parameter homotopies in order to compute the extra sample. The total number of paths we have to track in order to compute the basic and extra sample using homotopy continuation is thus: 
$$\#P(X\cap L)+ {n \choose d}b\delta^{-d}\text{deg}(X) + \sum_{k=1}^{d-1} {n \choose k} b\delta^{-k} \text{EDD}(X\cap L_k) + \#P(I(X \cap L_k, q))$$

%Let $P(n, d)$ denote the complexity of solving an $n\times n$ system of polynomial equations of degree $\leq d$. 

As a final comment we note that the sampling algorithm is
parallelizable. The basic sampling can be parallelized
over the various projections to coordinate planes as well as the grid
points. The extra sample may be parallelized over the linear slices
and computations of normal loci. Furthermore, if homotopy methods are
used to solve the systems of equations we may parallelize one more
aspect of the computation in a straightforward way. Namely, noting that
homotopy methods are ultimately path tracking algorithms using
prediction/correction, we may parallelize over the paths that need to
be tracked.

\subsection{Computational comparison}\label{sec:comp}
In this section we present a comparison of our sampling method with the sampling method presented in the paper \cite{DEHH18} by Dufresne et. al. that also produces a provably dense sampling of algebraic varieties. Due to the spatial structure of their algorithm, an analysis of its complexity is complicated. Their paper does not contain a complexity analysis of their algorithm which makes a theoretical comparison difficult.  Because of this we present a computational comparison between the two algorithms where we analyse the number of paths that need to be tracked by a homotopy continuation software. Dufresne et. al. also use parameter homotopies to solve their problem and at each inner loop they have to track EDD$(X)$ number of paths. Note that for convenience we do not consider the number of paths that need to be tracked to compute the generic start point of the parameter homotopy.  Due to the structure of their algorithm, analysing the number of paths that have to be tracked is difficult as it relies on spatial data structures. We therefore present the following computational comparison. The varieties in the table below were chosen to be random compact complete intersections of the indicated dimension and degree.
\begin{table}[!htbp]
	\centering
	\begin{tabular}{cccccccccc}
		\midrule
		$n$ & dim$(X)$ & deg$(X)$ & $\delta$ & \multicolumn{2}{c}{\# tracked paths} & \multicolumn{2}{c}{\# sample points} &  \multicolumn{2}{c}{total CPU time}  \\
		\midrule
		{}	&{} & {} & {} & DEHH \cite{DEHH18}   & DEG    & DEHH \cite{DEHH18}   & DEG  & DEHH \cite{DEHH18}   & DEG\\
		
		2 & 1 & 2 & 0.5 & 1184  & 148 & 36 & 58 & 16s & 0.7s \\
		2 & 1 & 2 & 0.1 & 2940 & 724 & 176 & 294 & 40s & 0.9s \\
		
		3 & 2 & 2 & 0.5 & 476814  & 288 & 69 & 149 & 20m & 0.9s \\
		3 & 2 & 2 & 0.1 & 764688  & 4698 & 976  & 2479 & 30m & 1.4s \\
		
%		3 & 2 & 2 & 0.5 &   & 5040 &  & 2018 &  & 1.4s \\
%		3 & 2 & 2 & 0.1 &   &  &  &  &  & \\
		
		4 & 2 & 4 & 0.5 & 20 467 848  &  576  & 75 & 122 & 62h & 1.3s \\
		4 & 2 & 4 & 0.1 & $>10^8$  & 8640  & -  & 2884 & $> 72$h & 3s \\
	\end{tabular}
	\caption{Computational comparison with the algorithm by Dufresne et. al. \cite{DEHH18}.}
\end{table}

It is clear from the above table that the algorithm by Dufresne et. al. tends to compute a sparser sample, but at a significantly higher cost. As the ambient dimension increases their algorithm breaks down since the number of boxes that need to be checked in the spatial data structure explodes, while our algorithm remains stable.

The experiments for the algorithm by Dufresne et. al. was run using the implementation by Edwards \cite{sampling_edwards} on the Tegner supercomputer at Parallelldatorcentrum (PDC) KTH, using a computational node consisting of 24 Intel E5-2690v3 Haswell cores and 512Gb of RAM. A public implementation of our algorithm is available at \cite{sampling_bottlenecks2020} and the experiments were run on the same computer as above.

\subsection{Experiments} \label{sec:exp}

In this section we present some sampling experiments. The equation
solving required for the sampling was carried out using numerical
homotopy methods with the solver \emph{Bertini} \cite{bertini}. The
sampled varieties are random complete intersections $X \subset
\RR^n$. To get more interesting examples we choose homogeneous
polynomials and to get compact varieties we add a random ellipsoid
defined by $\sum_{i=1}^n (1+r_i)x_i^2 = 1$ with $r_i>0$ random.

We first produce samples of random surfaces in $\RR^4$ defined by the
ellipsoid and a quadratic, cubic, quartic or quintic random
polynomial. Coarse samplings projected to $\RR^3$ of the generated
surfaces are shown in \figref{fig:surfaces_R4}. In
\tabref{tab:data-R4} we gather some data from the sampling
process. The samples produced are $\epsilon$-samples for various
$\epsilon$ given in the table. The other columns are the grid size
$\delta$, the size of the basic sample $E_{\delta}$ and extra sample
$E'_{\delta}$ as well as the total CPU time used for all processes. We
write the CPU time as $a+b$ where $a$ is the time for computing
bottlenecks and $b$ is the time for sampling. For the quadric surface
$a$ is less than 5 seconds and is omitted in the table. We implemented
the algorithm using eight 1.80 GHz cores of type Intel i7-8550U.

Next we consider random complete intersection surfaces in $\RR^n$
defined by random quadrics. We conducted the experiments for three
surfaces in $\RR^5$, $\RR^6$ and $\RR^7$, respectively. We also did a
complete intersection curve in $\RR^6$ defined by random quadrics. As
above we take homogeneous quadrics except for the random ellipsoid
which is affine and included in all examples. Coarse samples projected
to $\RR^3$ may be seen in \figref{fig:cis}. Some sampling data is
presented in \tabref{tab:data-ci} where we, in addition to the fields
discussed above, also include the ambient dimension $n$, the dimension
of $X$ and the degree of $X$ as a complex variety.

% call them quadric in ellipsoid or something?
\begin{figure}[!ht]
\centering
\subfloat[][A quadric.]{
  \includegraphics[trim={2cm 3cm 2cm 3cm},clip,scale=0.1]
                  {quadric_R4.png}}
\qquad
\subfloat[][A cubic.]{
  \includegraphics[trim={1cm 1cm 1cm 1cm},clip,scale=0.09]
                  {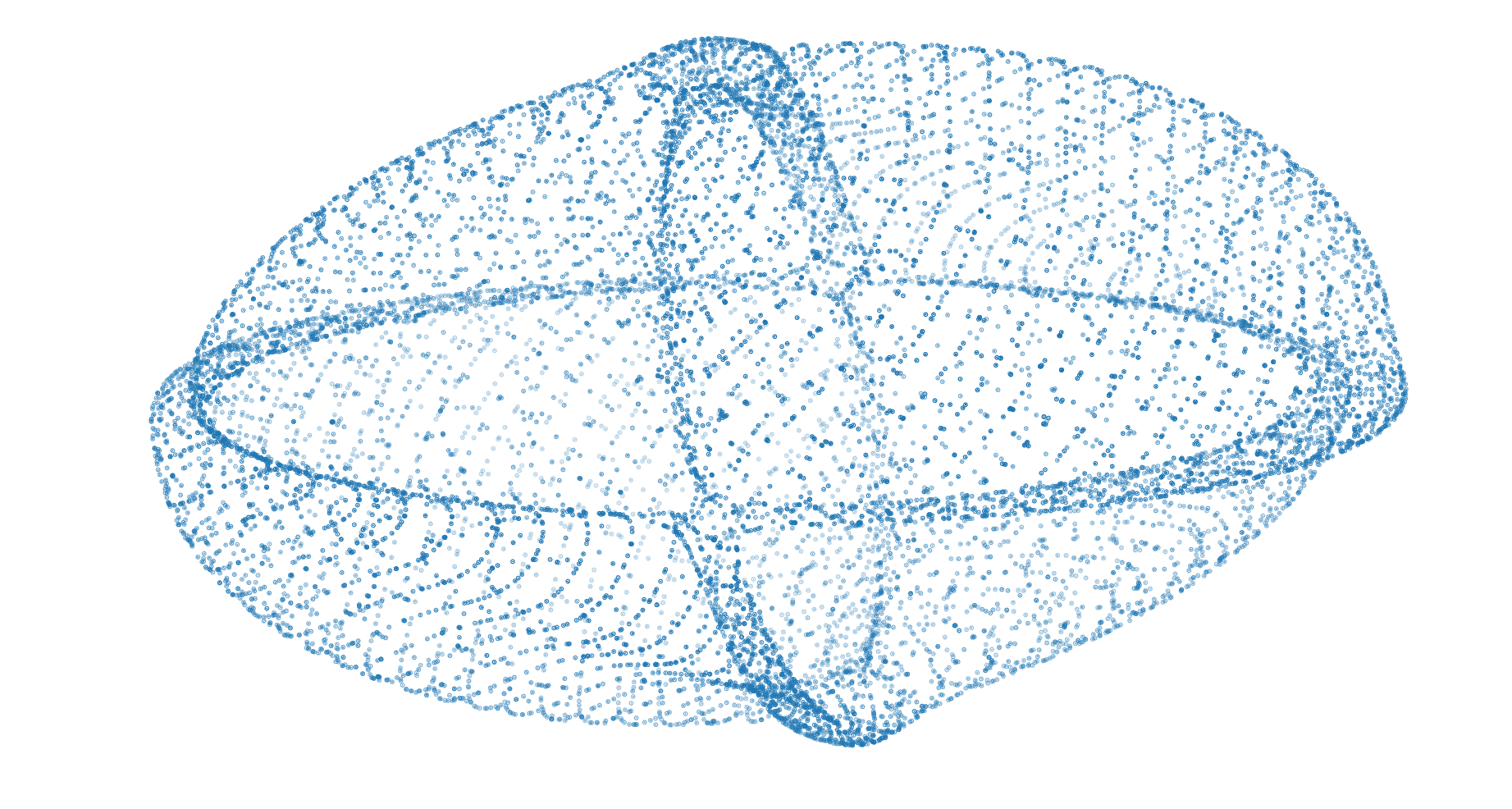}}
\qquad
\subfloat[][A quartic.]{
  \includegraphics[trim={2cm 3cm 2cm 2cm},clip,scale=0.1]
                  {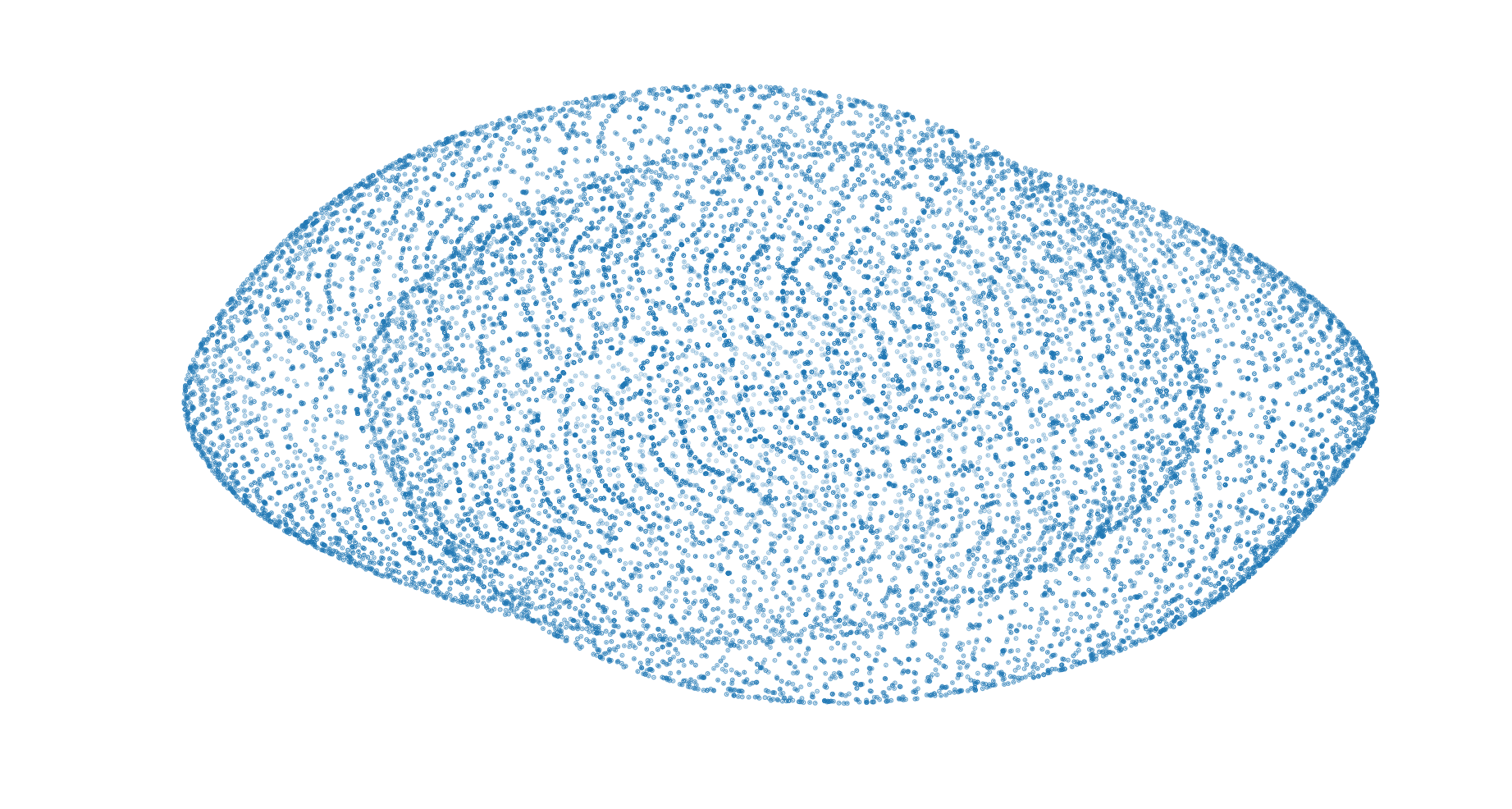}}
\qquad
\subfloat[][A quintic.]{
  \includegraphics[trim={2cm 3cm 2cm 2cm},clip,scale=0.1]
                  {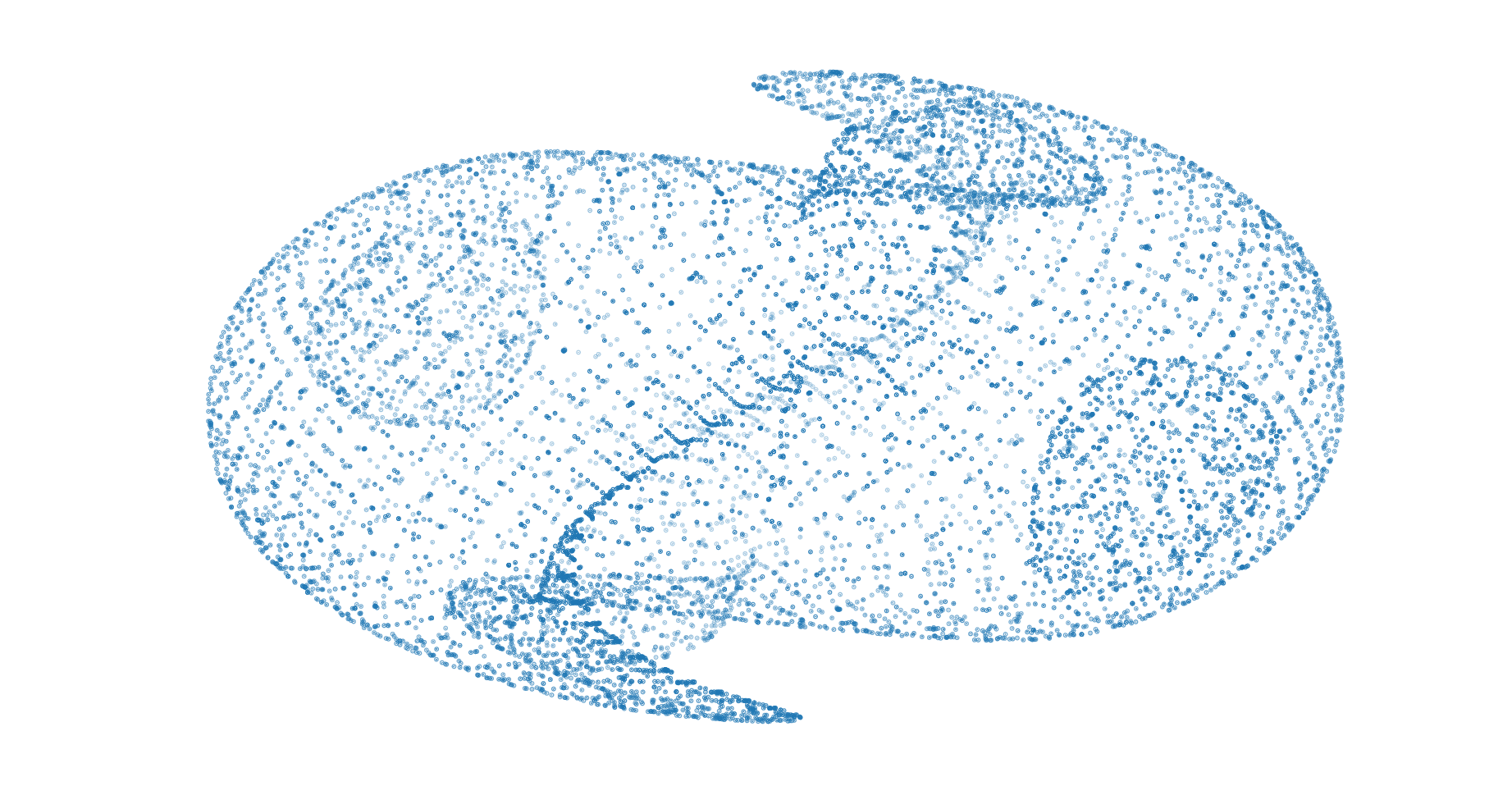}}
\caption{Plots of sampled surfaces in $\RR^4$.
\label{fig:surfaces_R4}}
\end{figure}

\begin{table}[!ht]
  \begin{center}
  
  \begin{tabular}{cccccc}
    \hline $X$ & $\epsilon$ & $\delta$ & $|E_{\delta}|$ &
    $|E'_{\delta}|$ & total CPU time (m) \\
    \hline \\
    Quadric & 0.1 & 0.05 & 11110 & 438 & 5 \\ % bn ok (5s)
    Quadric & 0.02 & 0.01 & 269900 & 2140 & 122 \\
    Quadric & 0.01 & 0.005 & 1075192 & 4313 & 470 \\

    Cubic & 0.06 & 0.03 & 35795 & 1150 & 1+26 \\ % bn ok (39s)
    Cubic & 0.02 & 0.01 & 315468 & 3512 & 1+224 \\

    Quartic & 0.06 & 0.03 & 40390 & 1026 & 8+49 \\ % bn ok (8m35s)
    Quintic & 0.1 & 0.05 & 10338 & 602 & 95+52 \\ % bn ok (95m)
  \end{tabular}
  \end{center}
  \caption{Sampling data for surfaces in $\RR^4$.\label{tab:data-R4}}
\end{table}

\begin{figure}[!ht]
\centering
\subfloat[][A surface in in $\RR^5$.]{
  \includegraphics[trim={2cm 3cm 2cm 3cm},clip,scale=0.08]
                  {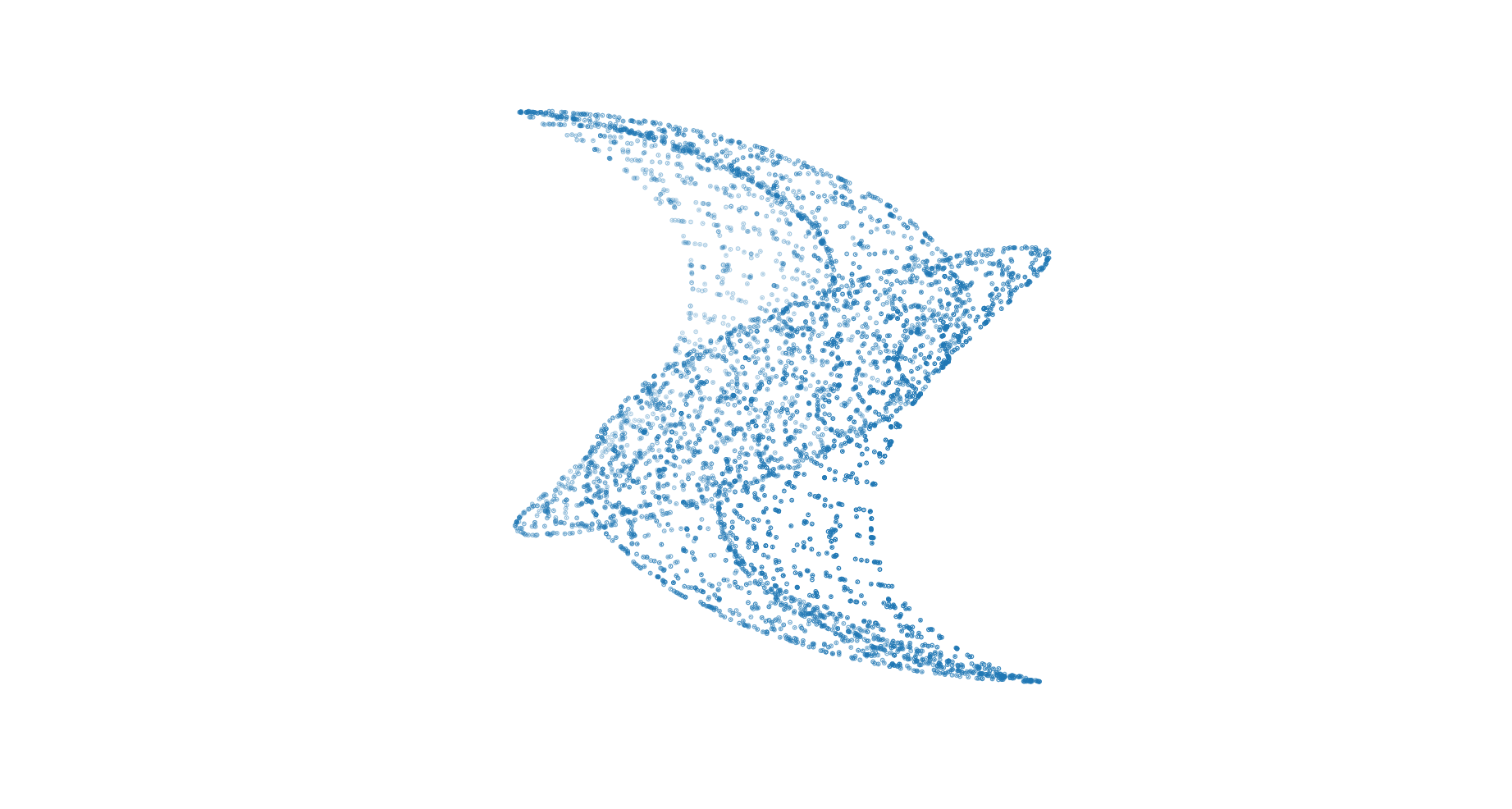}}
\qquad
\subfloat[][A surface in $\RR^6$.]{
  \includegraphics[trim={2cm 3cm 2cm 3cm},clip,scale=0.063]
                  {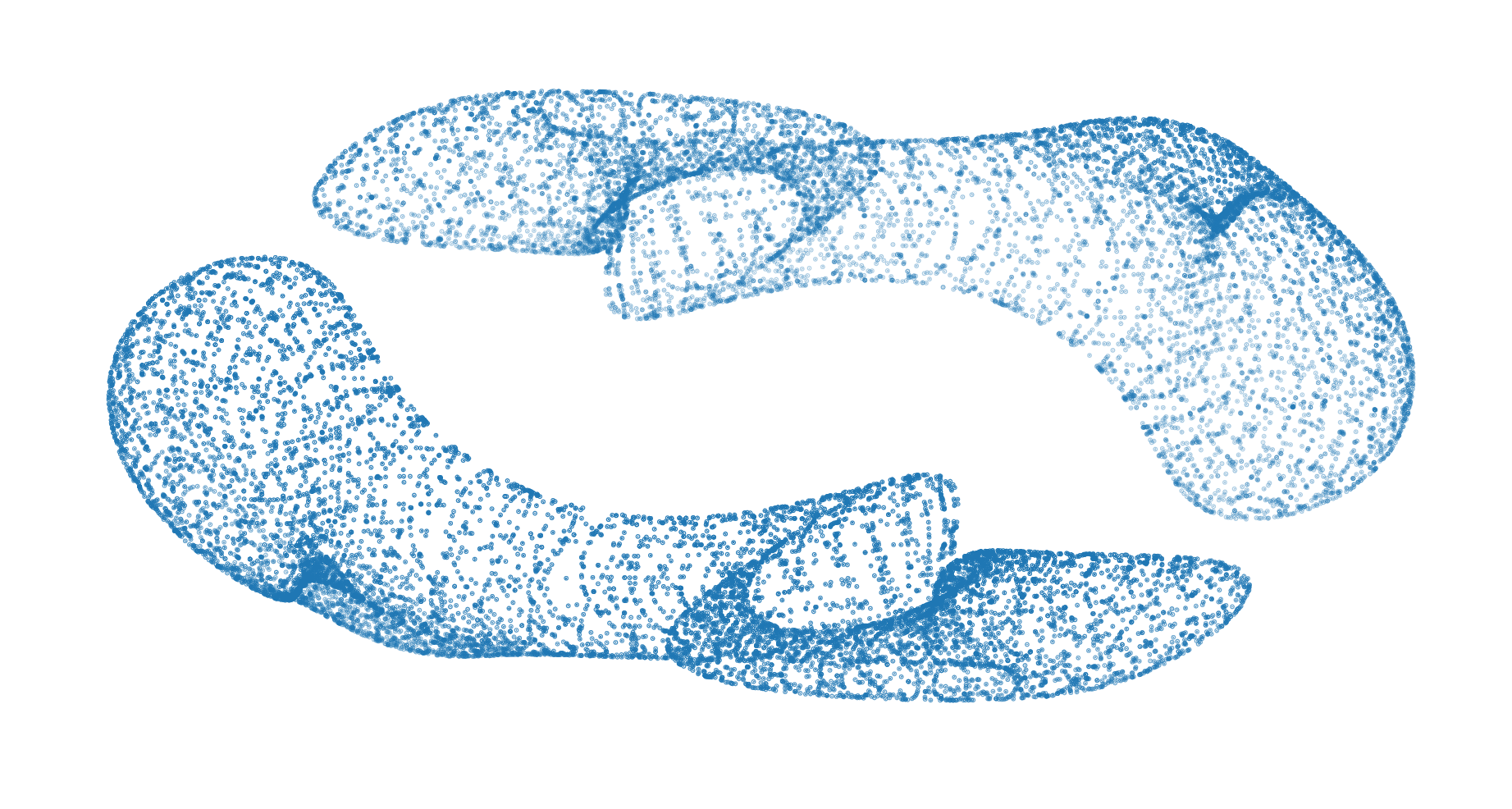}}
\qquad
\subfloat[][A curve in $\RR^6$.]{
  \includegraphics[trim={6cm 8cm 6cm 6cm},clip,scale=0.1]
                  {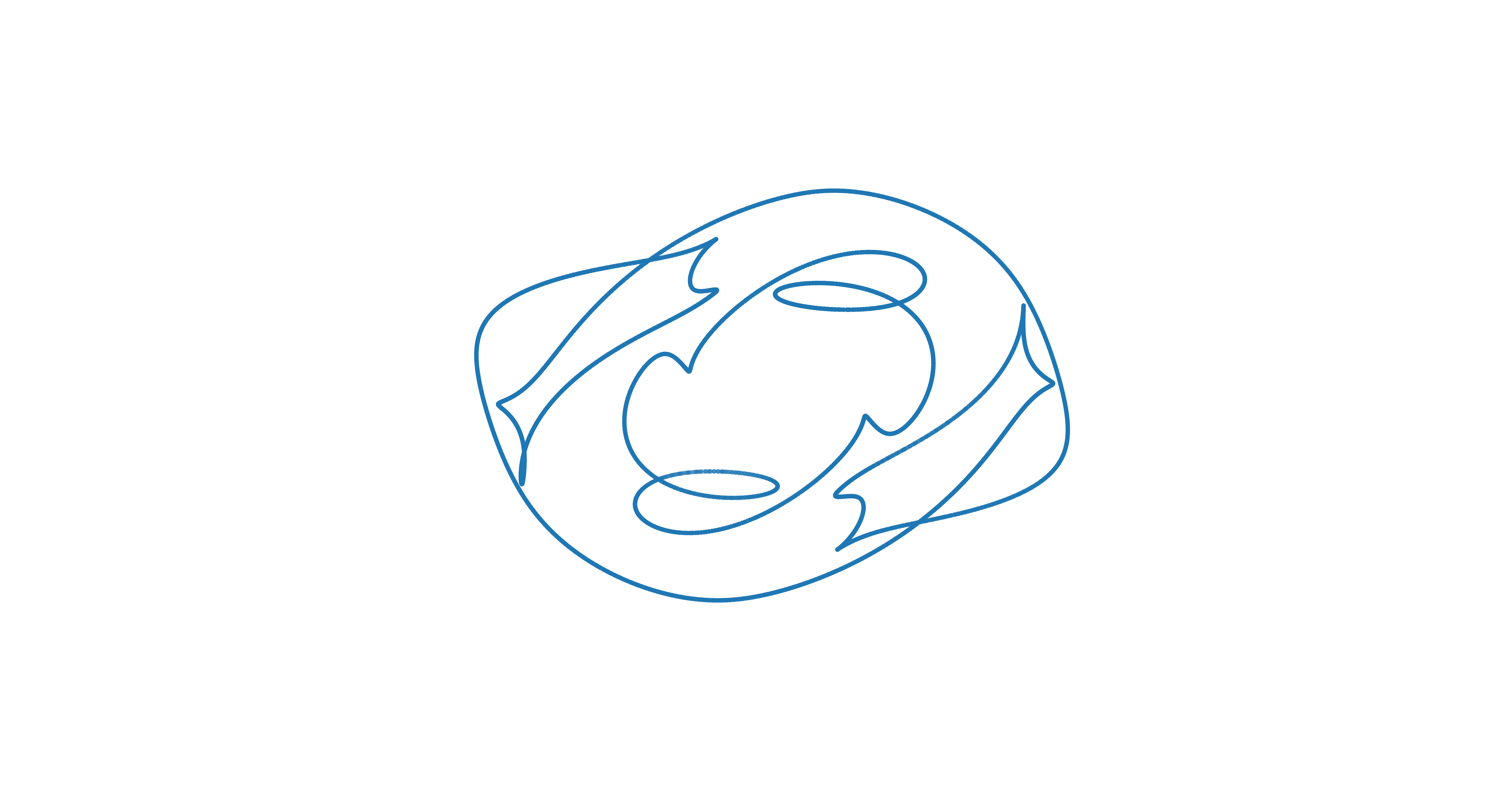}}
\qquad
\subfloat[][A surface in $\RR^7$.]{
  \includegraphics[trim={3cm 3cm 3cm 3cm},clip,scale=0.06]
                  {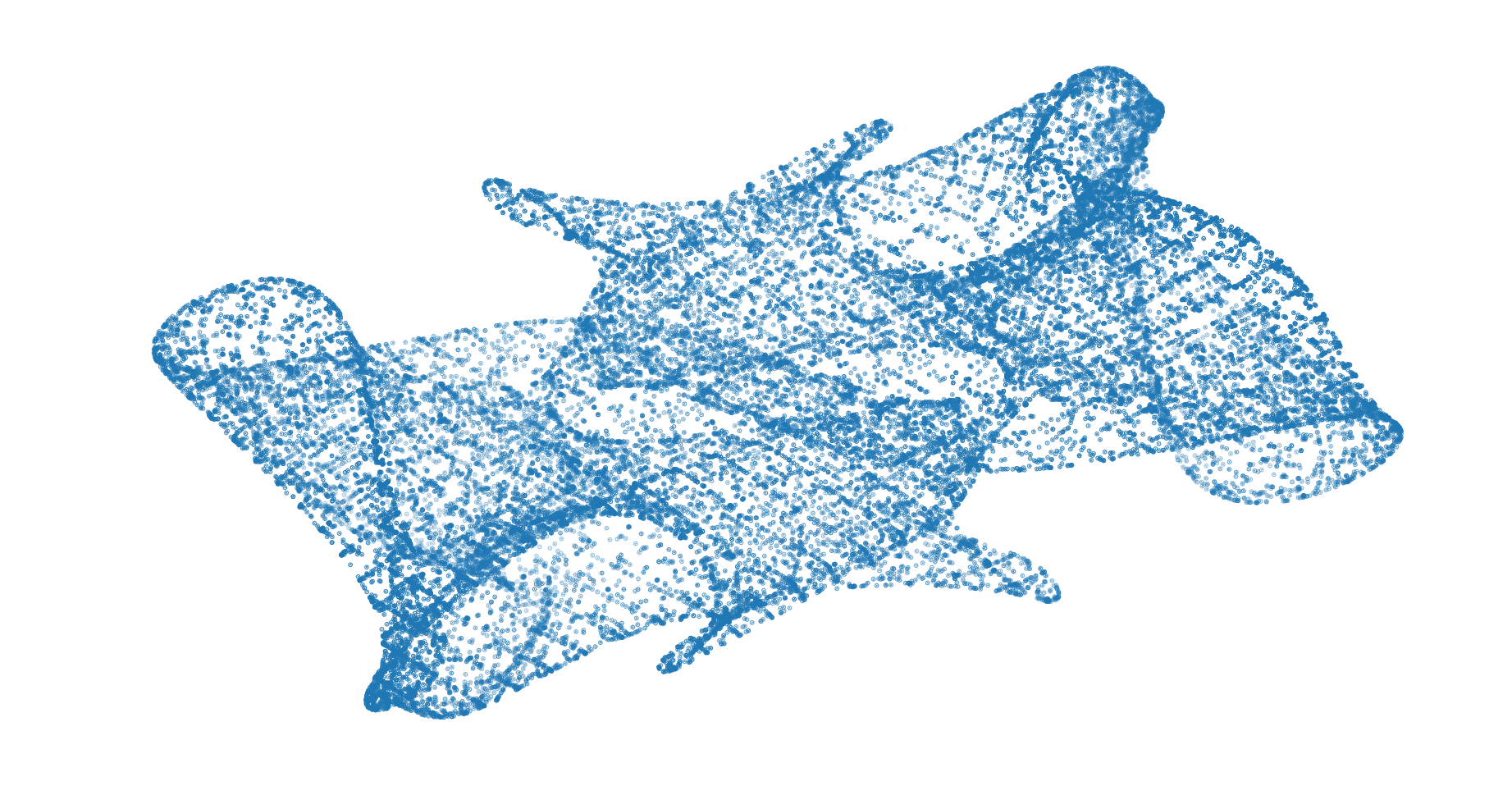}}
\caption{Plots of sampled complete intersection surfaces and a curve.
\label{fig:cis}}
\end{figure}

%{\color{red} \begin{rem}
%		If we sample a spherical algebraic set $X \subseteq S^{n-1} \subset
%		\RR^n$ defined by homogeneous equations using the sampling method of
%		this paper one has to be a bit careful regarding bottlenecks. Since
%		the antipodal map $S^{n-1} \to S^{n-1}:x\mapsto -x$ maps $X$ to
%		itself, we will have a residual component $\{(x,-x) \in X \times X: x
%		\in X\}$ of the bottleneck locus. These bottlenecks all have width
%		equal to 2. Of course, if all other components of the bottleneck locus
%		are $0$-dimensional, computing all the isolated bottlenecks is enough
%		to find the narrowest bottleneck of $X$.
%\end{rem}}

\begin{table}[ht!]
  \begin{center}
  \begin{tabular}{cccccccc}
    \hline
    $n$ & $\dim{X}$  & $\deg{X_{\CC}}$ & $\epsilon$ & $\delta$ & $|E_{\delta}|$ & $|E'_{\delta}|$ & total CPU time (m) \\
    \hline\\
    5 & 2 & 8 & 0.12 & 0.05 & 4428 & 304 & 1+22 \\ % bn ok (53s)
    6 & 2 & 16 & 0.13 & 0.05 & 22986 & 1070 & 19+70 \\ % bn ok (19m)
    6 & 1 & 32 & 0.013 & 0.005 & 9172 & 0 & 6+11 \\ % bn ok (6m)
    7 & 2 & 32 & 0.027 & 0.01 & 801472 & 8824 & 155+5192\\ % bn ok (155m)
  \end{tabular}
  \end{center}
    \caption{Sampling data for complete intersections $X \subset
      \RR^n$ defined by quadrics.\label{tab:data-ci}}
\end{table}
\newpage

\section{Homology guarantees} \label{sec:homology}
In this section we show, using the results of \cite{10.1016/j.gmod.2005.01.002}, that the homology of the thickenings of a compact smooth variety $X$ is controlled by its generalized bottlenecks. We show in Theorem \ref{thm:bn_hom} that if $\epsilon<wfs(X)$, then the zeroth and first homology of $X$ can be recovered from an $\epsilon$-sample using a modified Vietoris-Rips complex described in Section \ref{sec:modified}. 

Computing the weak feature size is a hard problem, but it is lower bounded by the reach of $X$ \cite{Chazal2009}. The last part of this section we present a variant of the theorem of Niyogi, Smale and Weinberger \cite{NSW08} stating that we can recover the correct homology of $X$ from an $\epsilon$-sample using local estimates of the reach.

\subsection{Homology from bottlenecks}\label{sec:modified} 
Recall the maps $\ell_k \colon B_k^{\Bbb R} \to \Bbb R$ from Definition \ref{def:wfs}. Let $L \coloneqq \cup_{k=1}^{EDD(X)} \text{im}( \ell_k)$ and let $X_\epsilon$ denote the $\epsilon$-thickening of $X$ defined as follows:
$$X_\epsilon \coloneqq \{ y\in \RR^n \mid \overline{B}_y(\epsilon) \cap X \neq \emptyset \}.$$ 
It is shown in  \cite[Theorem 1]{10.1016/j.gmod.2005.01.002} that if $\epsilon$ is less than the weak feature size of $X$, then $X_\epsilon$ is homotopy equivalent to $X$. We show the following result:

\begin{thm}\label{thm:bn_crit}
	Let $\epsilon_1 < \epsilon_2 \in \RR$. If $[\epsilon_1, \epsilon_2] \cap L = \emptyset$, then $$H_i(X_{\epsilon_1}) \cong H_i(X_{\epsilon_2})$$ In particular, if $\epsilon<wfs(X)$ then $X$ and $X_{\epsilon}$ are homotopy equivalent.
\end{thm}
\begin{proof}
	Let $\mathbb{B}_{\sigma}$ be an open ball centered at the origin containing the thickening $X_{\epsilon_2}$ of $X$. Consider now the sets $Y_1 \coloneqq \mathbb{B}_{\sigma} \setminus X_{\epsilon_1}$ and $Y_2 \coloneqq \mathbb{B}_{\sigma} \setminus X_{\epsilon_2}$. The thickening of $X$ preserves normality and from \cite{10.1016/j.gmod.2005.01.002} it follows that  the weak feature size of the thickening equals $wfs(Y_1) = \text{min} \{ t-\epsilon_1 \mid t \in L \text{ and } t > \epsilon_1 \}$. Now since $[\epsilon_1, \epsilon_2] \cap L = \emptyset$ it follows that $\epsilon_2-\epsilon_1 < wfs(Y_1)$ and thus by \cite[Theorem 1]{10.1016/j.gmod.2005.01.002} it follows that $Y_1$ and $Y_2$ are homotopy equivalent and that $Y_2$ is a deformation retract of $Y_1$. Consequently, $X_{\epsilon_1}$ and $X_{\epsilon_2}$ are homotopy equivalent as well.
	
	The latter statement follows from the above arguments and by the fact that $X$ is homotopy equivalent to $X_{\bar{\epsilon}}$ for $\bar{\epsilon} < \tau_X$ by \cite{BCL17}. By the same argument as in the previous paragraph, it follows that $X_\epsilon$ is homotopy equivalent to $X_{\bar{\epsilon}}$ which is then in turn homotopy equivalent to $X$.
%	To show the first statement we note that thickening $X$ preserves normality. Thus, $wfs(X_\epsilon) = \text{min} \{ t-\epsilon \mid t \in L \text{ and } t > \epsilon \}$. The statement then again follows from  \cite[Theorem 1]{10.1016/j.gmod.2005.01.002}.
%	
%	
%	The latter statement follows directly from  \cite[Theorem 1]{10.1016/j.gmod.2005.01.002}.
\end{proof}

\begin{ex}
	Consider the filtration of thickenings $\{ X_\epsilon \}_{\epsilon\in \RR}$ where $X_{\epsilon}\subseteq X_{\epsilon_2}$ for $\epsilon_1 \leq \epsilon_2$. Computing the homology of each $X_\epsilon$ in the filtration yields a functor sending $\epsilon$ to the $i$th homology group of $X_\epsilon$. By the above theorem $H_i (X_{\epsilon_1}) \cong H_i (X_{\epsilon_2})$ if $[\epsilon_1 , \epsilon_2] \cap L = \emptyset$. Thus the values of when the $i$th homology of the filtration  $\{ X_\epsilon \}_{\epsilon\in \RR}$ changes is a subset of set $L$. These are called the \textit{critical values} of the functor. The functors describing the homology of the filtration is in \cite{HOROBET2019101767} called the \textit{barcodes} of $X$. We conclude that the above theorem thus implies that the critical values of the barcodes of $X$ is a subset of the sizes of its generalized bottlenecks. 
\end{ex}

We will now show that the zeroth and first homology of $X$ can be recovered from a $\epsilon$-sample of $X$ for $\epsilon < b_2$ in the case when $b_2=wfs(X)$. To recover the homology from the sample we make a special Vietoris-Rips construction on the sample and show that the resulting simplicial complex has the same zeroth and first homology as $X$. The intuition is that the part of the reach coming from the curvature of $X$ does not matter for homology and the only thing that matters are the bottlenecks. However, computing the homology of an $\epsilon$-sample of $X$ with $\epsilon < b_2$ using e.g. a Vietoris-Rips complex does not immediately work since we run into trouble at the regions of high curvature. We thus have to find a way of excluding their contribution to the homology. 

Since we are concerned with only the zeroth and first homology of $X$ we consider only simplices up to dimension two. Let $S$ be the modified Vietoris-Rips complex constructed as follows:
\begin{enumerate}
	\item there is a 1-simplex between $a, b\in E$ if: \begin{enumerate}\item[($i$)] $\bar{B}_\epsilon(a) \cap \bar{B}_\epsilon(b) \neq \emptyset$ or \item[($ii$)]\label{cond:ii} there exists a $c\in E$ such that $\bar{B}_\epsilon(a) \cap \bar{B}_\epsilon(c) \neq \emptyset$, $\bar{B}_\epsilon(b) \cap \bar{B}_\epsilon(c) \neq \emptyset$ and $\bar{B}_{\sqrt{8}\epsilon}(a) \cap \bar{B}_{\sqrt{8}\epsilon}(b) \neq \emptyset$.\end{enumerate}
	\item there  is a 2-simplex between $a, b, c \in E$ if there is a 1-simplex between each pair.
\end{enumerate}
Since $\epsilon< b_2$ it follows that condition (1).($ii$) is only satisfied in regions of high curvature. The factor $\sqrt{8}$ ensures that the triangle between $a, b$ and $c$ can fit inside a quarter of a circle of radius $2\epsilon$. This means that if $a, c$ and $b, c$ are both at distance $2\epsilon$ from each other, then the triangle between $a, b$ and $c$ cannot have an angle larger than 90 degrees at $c$.

Let $U(E, \epsilon)$ denote the union of all closed balls $\bar{B}_\epsilon(x)$ for $x\in E$. We label a 1-simplex between $a, b\in E$ in $S$ as \textit{good} if the intersection of $X$ and the smallest closed ball containing both $a$ and $b$ has only one component. This means that the 1-simplex can be deformed in $U(E, \epsilon)$ to a path between $a$ and $b$ lying on $X$. Note that we will only have that the above intersection has two components in regions of high curvature since $\epsilon<b_2$. We will show that any 1-simplex in $S$ is homotopic to a path of good edges in $S$. This then implies that the first homology of $S$ and $X$ are the same. 

\begin{lemma}\label{lem:hom}
	Any 1-simplex in $S$ is homotopic to a path of good edges in $S$.
\end{lemma}
\begin{proof}
	Let $a, b\in E$ be such that there is a 1-simplex between $a$ and $b$ in $S$ which is not good. Without loss of generality, assume that there is no neighbour $c$ of $b$ such that $b, c$ is a good edge and $d(a, c) < d(a, b)$. If there exist such a $c$ the path $a, c, b$ is homotopic to $a, c$ and since $b, c$ is a good edge we need to show the statement for $a, c$. It then also follows that $a$ is the closest to $b$ of its neighbours with good edges. 
	
	\begin{figure}[ht]
		\includegraphics[trim={0cm 3cm 0cm 1cm},clip,width=0.5\linewidth]{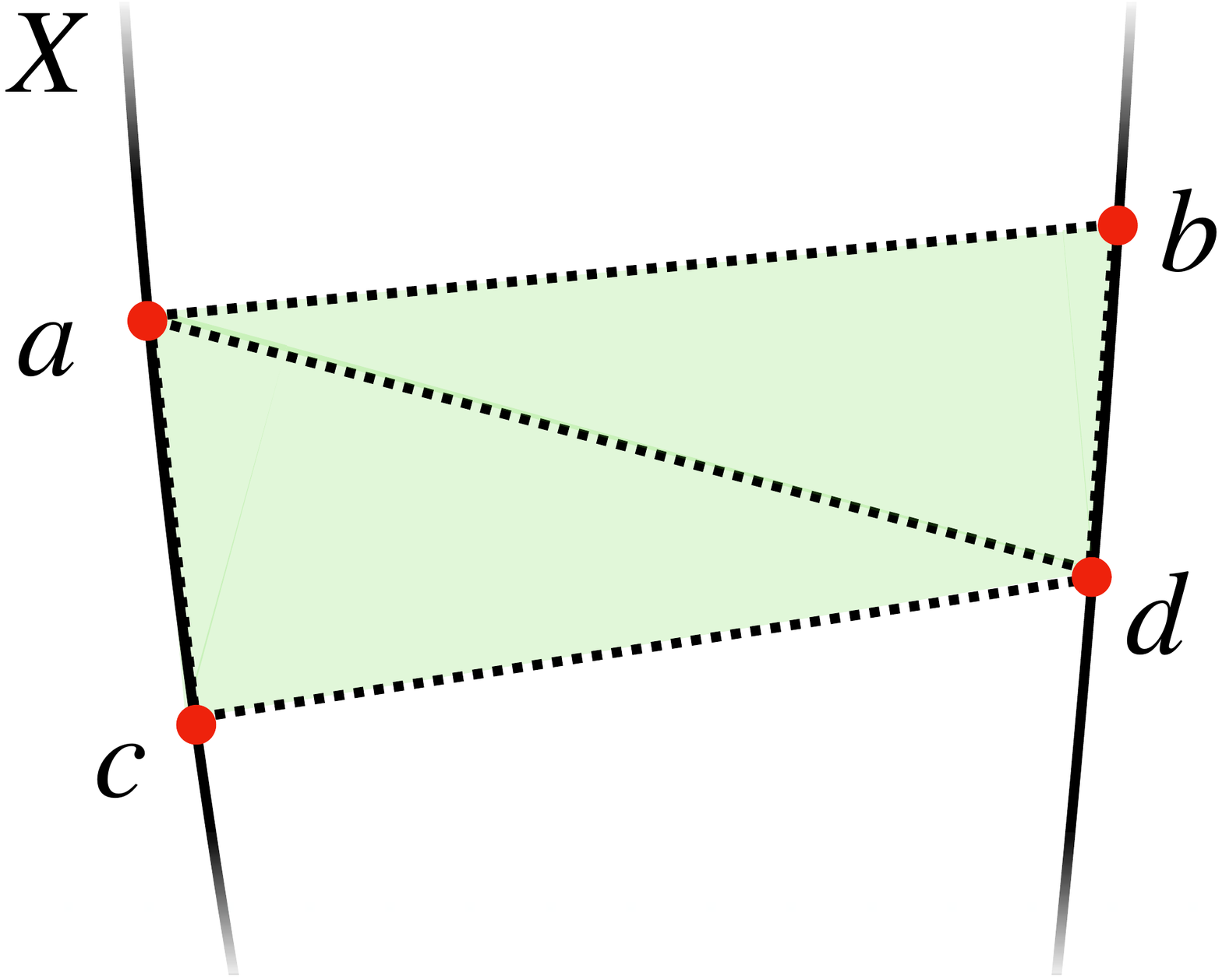}
	
		\caption{Illustration of the case where the angle of the triangle $a, b, c$ is larger than $90$ degrees at $a$.} 	\label{fig:good_edge}
	\end{figure}
	
	The union $\bar{B}_\epsilon(a) \cup \bar{B}_\epsilon(b)$ does not contain a bottleneck since $\epsilon < wfs(X)$. It follows that there are $c, d\in E$ such that $a, c$ and $b, d$ are good edges and $d(c, d) < d(a, b)$. If the angle at $a$ of the triangle $a, b, c$ is less than $90$ degrees we have both the 2-simplices $a, b, c$ and $b, d, c$ by condition $(1).(ii)$ in the construction of $S$. Otherwise, if the angle at $a$ is larger than 90 degrees we instead have the 2-simplices $a, c, d$ and $a, b, d$ by condition $(1).(ii)$ and $(2)$. This case is illustrated by Figure \ref{fig:good_edge}. Consequently, the path $a, b$ is homotopic to the path $a, c, d, b$ where $a, c$ and $d, b$ are good edges. If $c, d$ is a good edge then we are finished, otherwise we repeat the above argument and since $d(c, d) < d(a, b)$ and the number of samples are finite and $U(E, \epsilon)$ covers $X$ we will eventually terminate, resulting in a path of good edges homotopic to the original 1-simplex. 
\end{proof}
We use the above lemma to show the following theorem:
\begin{thm}\label{thm:bn_hom}
	Let $\epsilon < wfs(X)$ and let $i\in \{0, 1\}$. Let $E$ be an $\epsilon$-dense sample of $X$ and let $S$ be the modified Vietoris-Rips complex constructed from $E$ as above. Then, $$H_i(S) \cong H_i(X)$$
\end{thm}
\begin{proof}
	The zeroth homology counts the number of components of $X$. Note that the length of the smallest bottleneck of $X$ is smaller or equal to half of the smallest separation between any two components of $X$. Since $U(E, \epsilon)$ covers $X$ and $\epsilon < wfs(X)$ they have the same number of components. The covering $U(E, \epsilon)$ and its Vietoris-Rips complex have the same number of components. It then follows that $U(E, \epsilon)$ and $S$ have the same number of components since condition $1.(ii)$ only adds an edge between vertices who are already connected with an $\epsilon$-path. 
	
	First note that since $\epsilon< wfs(X)$ and $U(E, \epsilon)\subseteq X_\epsilon$ we have by Theorem \ref{thm:bn_crit} that it cannot happen that a cycle on $X$ which is not contractible is contractible in $S$ as a cycle of 1-simplices. It therefore remains to show that any cycle of 1-simplices on $S$ is homotopic to a cycle in $X$. This follows from Lemma \ref{lem:hom}. Since any 1-simplex of $S$ is homotopic to a path of good 1-simplices it follows by definition that it can be deformed to lie on $X$. This implies that any 1-cycle on $S$ can be deformed to lie on $X$. Since we also have that $U(E, \epsilon)$ covers $X$ and $\epsilon <wfs(X)$ it follows that any 1-cycle on $X$ can be deformed in $U(E, \epsilon)$ to lie on a subset of the 1-simplices of $S$. We thus have and isomorphism $H_1(S) \cong H_1(X)$. 
\end{proof}
%From the above theorem it is interesting to understand when the narrowest bottleneck equals the wfs. For varieties in $\Bbb R^2$ it seems to be the case that the narrowest bottleneck equals the wfs for all varieties except those of the form:

%By the nature of the conditions of Definition \ref{def:bn_locus} it seems likely that the narrowest bottleneck equals the wfs for a generic variety $X\subset \Bbb C^n$, although this requires further investigation.

In the following example we illustrate the difference between using the density given by the reach and by smallest bottleneck, in the case when $b_2=wfs(X)$. For varieties with regions of high curvature this can make a significant difference. 
\begin{ex} \label{ex:sparse} Consider the curve $X\subset \RR^2$ in the example \cite{reach-curve2019} illustrated by Figure \ref{fig:ex} and given by the equation:
	$$(x^3 - xy^2 + y + 1)^2  (x^2 + y^2 - 1) + y^2 - 5 = 0$$
	As noted in \secref{sec:edd}, the reach can be expressed as the minimum
	$$\tau_X = \text{min}\{ \frac{1}{\rho}, b_2 \}$$
	where $\rho$ is the maximal curvature and $b_2$ is the radius of the narrowest bottleneck of $X$. Computing these quantities yields
	$$\rho \approx 2097.17, \ \ b_2 \approx 0.13835$$
	which yields that $\tau_X \approx 4.79 \cdot 10^{-4}$. Using \cite[Proposition 2.2]{BCL17}, to recover the correct homology of $X$ we thus need a sampling density of $\epsilon < \tau_X \approx 4.79 \cdot 10^{-4}$. By Theorem \ref{thm:bn_hom} it suffices to choose the sampling density $\epsilon <  b_2 \approx 0.0692$. By Theorem \ref{thm:grid} we should then choose the grid size $\delta$ less than $\frac{\epsilon}{\sqrt{n}}$. The sampling algorithm returns the following number of sample points for each of the above densities:
	
	\begin{table}[ht!]
		\begin{center}
			\begin{tabular}{cccccccc}
				\hline
				$\delta$ & $|E_{\delta}|$ \\
				\hline 
				$0.0489$ & 574\\
				$3.71 \cdot 10^{-4}$ & 58070\\
			\end{tabular}
		\end{center}
		\caption{Comparison of $|E_{\delta}|$ when the density is given by $b_2$ and $\tau_X$ respectively.}
	\end{table}
	Note in the above table that we need significantly less number of sample points in order to recover the correct homology of $X$ when the density is given by the smallest bottleneck instead of the reach. 
	\begin{figure}[ht!]
		\includegraphics[trim={0cm 2cm 0cm 2.5cm}, clip, width=0.5\linewidth]
		{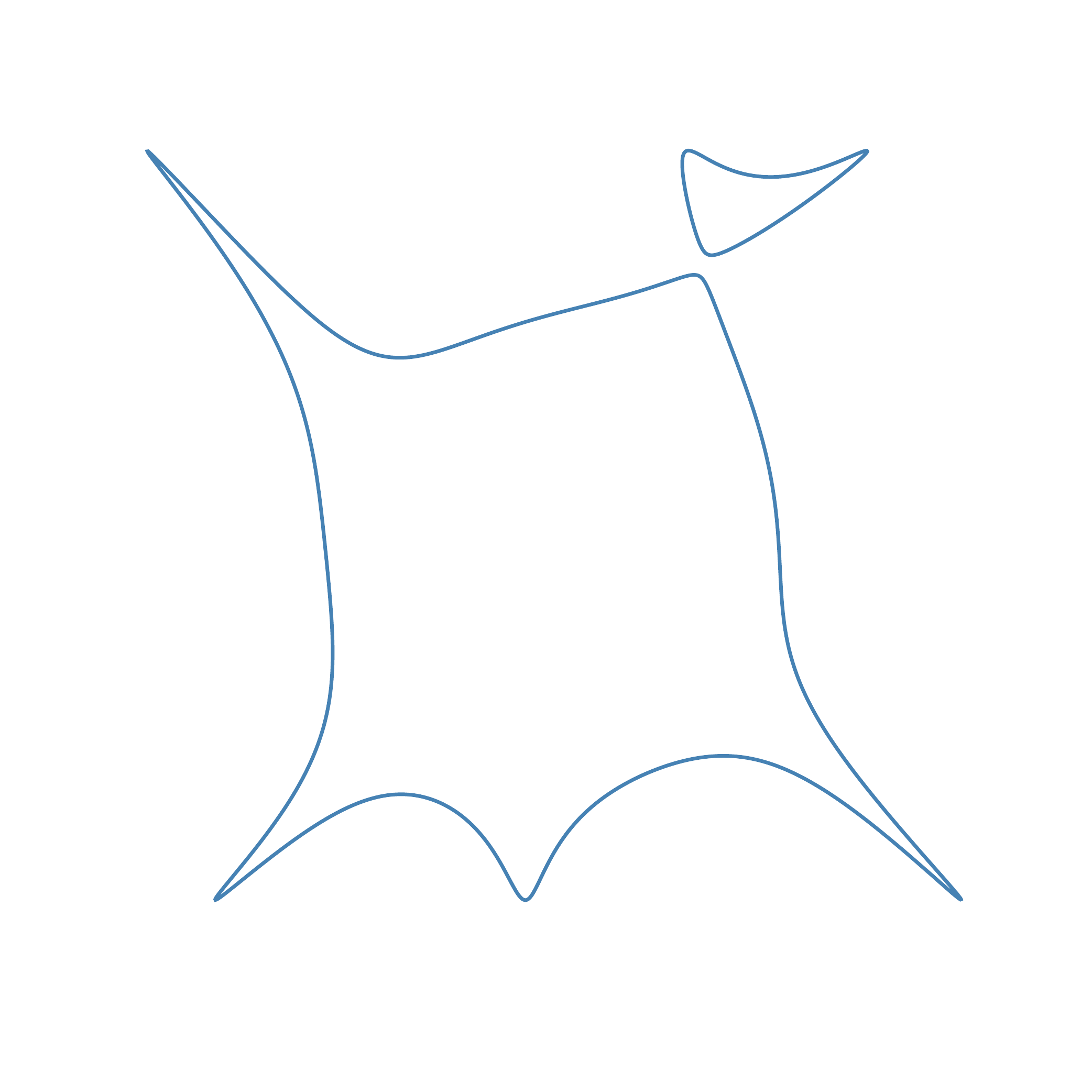} \vspace{-15pt}
		\caption{The curve $X\subset \RR^2$ in the example \cite{reach-curve2019}.}
		\label{fig:ex}
	\end{figure}
\end{ex}

\subsection{Homology using reach}

Let $X \subset \RR^n$ be a smooth compact variety. In this section we
present a variant of many statements that appear in the literature of
this topic, see for example \cite{BCL17,NSW08,CL05}. Roughly speaking,
these results show that a \u{C}ech complex associated to a dense and
accurate enough sample of $X$ deformation retracts onto $X$. In
particular $X$ and the complex have the same homology groups which
gives an algorithm to compute the homology. Given a finite
$\epsilon$-sample $E \subset X$ we have the open cover $X \subseteq
\bigcup_{e \in E} B_e(\epsilon)$ of open balls with radius
$\epsilon$. The \u{C}ech complex $C(E,\epsilon)$ is the nerve
associated to this cover. This means that the simplices of the
complex are the subsets $\sigma \subseteq E$ such that $\bigcap_{e \in
  \sigma} B_e(\epsilon) \neq \emptyset$. Our reasoning follows the
proof of Theorem 2.8 of \cite{BCL17} closely with the difference that
$\epsilon$ is compared to the local reach along the sample $E$ rather
than the global reach. This way we may find a complex with the same
homology groups as $X$ using a sample of the manifold $X$ rather than
a grid in the ambient space $S^{n-1}$ as is done in \cite{BCL17}. 

%It is an interesting problem to further investigate the implications of
%local approaches like this for the computational complexity of
%homology groups. Roughly speaking we expect the size of a sample of
%the ambient space to be exponential in $n$ but a sample of $X$ to be
%exponential in $\dim{X}$.

\begin{thm}\label{homology}
If $E \subset X$, $\epsilon > 0$, $E$ is $(\epsilon/2)$-dense and
$\epsilon < \frac{4}{5}\tau_X(e)$ for all $e \in E$, then $X$ is a
deformation retract of $C(E, \epsilon)$. In particular, $X$ and $C(E,
\epsilon)$ have the same homology groups.
\end{thm}

\begin{proof}
By the assumption $\epsilon < \tau_X(e)$ for all $e \in X$, the
projection $\pi_X:C(E,\epsilon) \rightarrow X$ is well defined and it
is continuous by \cite{BCL17} Proposition 2.2. We will show that
$(t\pi_X(x)+(1-t)x) \in C(E,\epsilon)$ for all $x \in
C(E,\epsilon)$. Once this is established the deformation retraction
is, as in \cite{BCL17,NSW08,CL05}, defined by $$C(E,\epsilon) \times
[0,1] \rightarrow C(E,\epsilon):(x,t) \mapsto t\pi_X(x)+(1-t)x.$$

For $a,b \in \RR^n$, let $[a,b]=\{tb+(1-t)a:t\in [0,1]\}$ and define
$[a,b)$, $(a,b]$ and $(a,b)$ similarly by excluding the indicated end
points. Let $e \in E$, let $B=B_e(\epsilon) \subseteq C(E,\epsilon)$
and let $x \in B$. We need to show that $[x,\pi_X(x)] \subseteq
C(E,\epsilon)$. If $\pi_X(x) \in B$ we are done since then
$[x,\pi_X(x)] \subseteq B$. Assume that $\pi_X(x) \notin B$ and let
$L$ be the line joining $x$ and $\pi_X(x)$. Since $L$ intersects $B$,
it intersects the boundary $\partial B$ in two distinct points on
$L$. Call these $p \in L$ and $q \in L$ were $p$ is closer to
$\pi_X(x)$ than $q$. We need to show that $||p-\pi_X(x)|| <
\epsilon/2$. Once this has been proven we may reason as follows. Since
$E$ is $(\epsilon/2)$-dense, there is a $e' \in E$ such that
$||\pi_X(x)-e'|| < \epsilon/2$. This means that $\pi_X(x) \in
B_{e'}(\epsilon)$ and $p \in B_{e'}(\epsilon)$ by the triangle
inequality. Hence $[p,\pi_X(x)] \subset B_{e'}(\epsilon)$. But $[x,p)
  \subset B_e(\epsilon)$ and hence $[x,\pi_X(x)] \subset
  C(E,\epsilon)$.

Let $y=\pi_X(x)$, $v=x-y$ and $\tau = \sup \{t \in \RR:
\pi_X(y+tv)=y\}$. Then $0<\tau$ since $\pi_X(x)=y$. Moreover, $\tau <
\infty$ since if $t_0 \in \RR$ is the unique real number such that
$y+t_0v$ intersects the hyperplane $\{u \in \RR^n : ||u-e||=||u-y||
\}$, then $\tau \leq t_0$ since $e \in X$. Put $z=y+\tau v$. By
\cite{F59} Theorem 4.8 (6), $z$ lies in the medial axis of $X$ and
hence $||e-z|| \geq \tau_X(e) > \frac{5}{4}\epsilon$. Also, by the
definition of $\tau$, $||z-y||\leq ||z-e||$. We need to show
$||y-p||<\epsilon/2$ so assume that $||y-p|| \geq \epsilon/2$. Then
$||p-z||= ||y-z||-||y-p|| \leq ||z-e||-\epsilon/2$. Since the line
segment $(q,p)$ is inside the bass $B_e(\epsilon)$, the angle at $p$
of the triangle formed by $e,p$ and $z$ is acute.  If follows from the
cosine theorem that $$||e-z||^2 \leq ||p-e||^2+||p-z||^2=\epsilon^2 +
||p-z||^2 \leq \epsilon^2 + (||z-e||-\epsilon/2)^2.$$
Hence $$||e-z||^2 \leq
\epsilon^2+||z-e||^2-||z-e||\epsilon+\epsilon^2/4$$ which implies that
$||z-e|| \leq \frac{5}{4}\epsilon$, a contradiction.
\end{proof}

\bibliographystyle{plain} \bibliography{collection}

\end{document}